\documentclass[reqno,a4paper,12pt]{article}
\usepackage{amsmath,amsthm,amssymb}
\usepackage[T1]{fontenc}
\usepackage{hyperref}
\usepackage{epigraph}
\usepackage{enumerate}
\usepackage{bbm}
\usepackage{bm}
\usepackage{mathtools}
\usepackage{mathrsfs}
\usepackage{nicefrac}
\usepackage{accents}
\usepackage{verbatim}
\usepackage{multicol}
\usepackage{paracol}
\usepackage{cleveref}
\usepackage[a4paper]{geometry}
\usepackage{xcolor}
\usepackage{tikz}
\tikzset{every picture/.style={line width=0.75pt}} 
\usetikzlibrary{positioning} 

\usepackage[sort&compress,numbers]{natbib}
\newtheorem{theorem}{Theorem}[section]
\newtheorem{lemma}[theorem]{Lemma}
\newtheorem{corollary}[theorem]{Corollary}
\newtheorem{proposition}[theorem]{Proposition}

\theoremstyle{definition}
\newtheorem{definition}[theorem]{Definition}

\newtheorem{remark}[theorem]{Remark}

\newcommand{\divides}{\mid}

\newcommand{\R}{{\mathbb{R}}}

\newcommand{\Z}{{\mathbb{Z}}}
\newcommand{\N}{{\mathbb{N}}}
\newcommand{\PP}{{\mathbb{P}}}

\newcommand{\mb}[1]{\mathbf{#1}}
\newcommand{\mc}[1]{\mathcal{#1}}

\makeatletter
\def\moverlay{\mathpalette\mov@rlay}
\def\mov@rlay#1#2{\leavevmode\vtop{%
   \baselineskip\z@skip \lineskiplimit-\maxdimen
   \ialign{\hfil$\m@th#1##$\hfil\cr#2\crcr}}}
\newcommand{\charfusion}[3][\mathord]{
    #1{\ifx#1\mathop\vphantom{#2}\fi
        \mathpalette\mov@rlay{#2\cr#3}
      }
    \ifx#1\mathop\expandafter\displaylimits\fi}
\makeatother

\providecommand{\noopsort}[1]{}

\renewcommand{\P}{\mathbb{P}}
\newcommand{\F}{\mathcal{F}}
\newcommand{\G}{\mathcal{G}}

\NewDocumentCommand{\Prob}{O{what} O{\left(} O{\right)}}{\mathbb{P}#2 #1 #3}
\newcommand{\Pro}[1][P]{\mathbb{#1}}
\newcommand{\E}{{\mathbb E}}

\newcommand{\p}[1][p]{\mathbf{#1}}
\NewDocumentCommand{\tos}{O{=} O{a.s.}}{\stackrel{\smash{\scriptscriptstyle\mathrm{#2}}}{#1}}
\NewDocumentCommand{\prp}{O{\p} O{n} O{T_n(\p)}}{#1^{\otimes #2} \left(#3\right)}

\NewDocumentCommand{\pof}{O{} O{} O{P} O{}}{\mathbb{#3}_{#2}^{#4}\left(#1\right)}
\NewDocumentCommand{\lof}{O{X} O{P}}{\mathbb{#2}_{#1}} 
\NewDocumentCommand{\cmv}{O{X} O{Y} O{}}{\E_{\mathbb{#3}}\left(#1|#2\right)}
\newcommand{\iof}[1][]{\mathbbm{1}_{#1}}

\NewDocumentCommand{\lo}{O{x} O{2}}{\log_{#2}\left(#1\right)}
\NewDocumentCommand{\h}{O{X} O{}}{\mathbf{H}_{#2}\left(#1\right)}
\NewDocumentCommand{\ph}{O{X} O{}}{\mathbf{H}_{#2}\left(\mb{#1}\right)}
\NewDocumentCommand{\pch}{O{X} O{Y}}{\mathbf{H}\left(\mathbf{#1} | \mathbf{#2}\right)}
\NewDocumentCommand{\ch}{O{X} O{Y} O{}}{\mathbf{H}_{#3}\left(#1\;|\; #2\right)}
\NewDocumentCommand{\cph}{O{X} O{Y}}{\mathbf{H}\left(\mb{#1}\;|\; \mb{#2}\right)}
\NewDocumentCommand{\nof}{O{1} O{in}}{\#_{#1}\left(#2\right)}
\NewDocumentCommand{\inffun}{O{X} O{\mc{G}}}{\mb{I}^{#2}\left(#1\right)}
 
\NewDocumentCommand{\prrh}{O{X} O{Y}}{D_p\left(#1||#2\right)}
\NewDocumentCommand{\prmi}{O{X} O{Y}}{I_p\left(#1;#2\right)}
\NewDocumentCommand{\prch}{O{X} O{Y}}{\mathbf{H}_p\left(#1|#2\right)}

\NewDocumentCommand{\preof}{O{}}{\pmb{\mathscr{P}}_{#1}}

\NewDocumentCommand{\proc}{O{X} O{\Z}}{\left(#1_{i}\right)_{i\in #2}}
\NewDocumentCommand{\crh}{O{X} O{Y} O{Z} O{W}}{\mathbf{H}\left(#1 | #2\Big|\Big| #3 | #4\right)}

\DeclarePairedDelimiter{\oci}{(}{]}

\title{Basic thermodynamical formalism for sandwich subshifts}
\author{Joanna Ku\l{}aga-Przymus \and Micha\l{} D.\ Lema\'{n}czyk \and Micha\l{} Rams}

\begin{document}
\maketitle

\begin{abstract}
Consider a partial order on $\{0,1\}^{\mathbb Z}: x\leq y$ when $x_i\leq y_i$ for all $i\in\mathbb{Z}$. A subshift $X\subset\{0,1\}^\Z$ is hereditary if together with any $x\in \{0,1\}^{\mathbb Z}$ it contains all $y\leq x$. Heuristically speaking, a hereditary subshift contains all the elements between maximal elements (with respect to this partial order) and the element $0^{\mathbb Z}$. In a particular situation when it suffices to take (the orbit closure of) all the elements between a single maximal element $x$ and the element $0^{\mathbb Z}$, we speak of subordinate subshifts. In this paper we investigate measure-theoretic properties of such subshifts, with a special emphasis on thermodynamical formalism. The key notion is a measure-theoretic counterpart of subordinate subshifts, where the role of a single maximal element is replaced with a single (maximal with respect to a certain order) invariant measure on $\{0,1\}^\Z$. 

We also introduce and investigate two-sided analogues of the above classes, we call them {\it sandwich hereditary}, {\it sandwich subordinate} and {\it sandwich measure-theoretically subordinate} subshifts. Sandwich hereditary subshifts can be thought of as sets of elements between some pairs of maximal and minimal elements satisfying certain assumptions. Sandwich subordinate subshifts occur when it suffices to take (the orbit closure of) all the elements between a single pair of sequences $(w,x)$, where $w\leq x$. In sandwich measure-theoretically subordinate subshifts, the role of a pair of sequences is replaced by a pair (precisely speaking: a joining) of two invariant measures on $\{0,1\}^\Z$.

The notions and results are motivated by those from the theory of so-called $\mathscr{B}$-free systems.
\end{abstract}

\tableofcontents

\section{Introduction}\label{se:intro}
This paper is dedicated to investigating a family of dynamical systems, called sandwich subshifts, through the lens of invariant measures and thermodynamic formalism. We will introduce this family shortly, however, for now, let us direct our attention toward a specific class of examples that serves as the primary motivation for our study. Given $\mathscr{B}\subseteq\mathbb{N}$, let $\mathcal{F}_\mathscr{B}:=\Z\setminus\bigcup_{b\in\mathscr{B}}b\Z$ be the corresponding set of \emph{$\mathscr{B}$-free integers} and consider the orbit closure $X_\eta$ of $\eta=\mathbf{1}_{\mathcal{F}_\mathscr{B}}$ under the left shift $\sigma$ (this is the so-called \emph{$\mathscr{B}$-free subshift}). Always when $\mathscr{B}$ is infinite, pairwise coprime with $\sum_{b\in\mathscr{B}}1/b<\infty$, then the $\mathscr{B}$-free subshift $X_\eta$ is \emph{hereditary}, that is, it is closed under taking sequences in $\{0,1\}^\Z$ smaller or equal coordinatewise that the ones that are already in $X_\eta$ (see~\cite{MR3428961}). In other words, we have $M(X_\eta\times \{0,1\}^\Z)=X_\eta$, where $M\colon \{0,1\}^\Z\times\{0,1\}^\Z\to \{0,1\}^\Z$ stands for the coordinatewise multiplication of sequences. For a general $\mathscr{B}\subseteq\N$, this might no longer be true (see~\cite{MR3803141}) and one defines so-called hereditary closure of $X_\eta$ by setting $\widetilde{X}_\eta:=M(X_\eta\times \{0,1\}^\Z)$. Clearly, $X_\eta\subseteq \widetilde{X}_\eta$. These two families of subshifts exhibit quite a variety of behaviours, setting them apart from other families classically examined in the literature. Let us give here a sample of results justifying this claim. Recall first the necessary notions from topological dynamics. A \emph{topological dynamical system} is a compact metric space $(X,d)$ equipped with a homeomorphism $T$ acting on it. We say that $X$ is \emph{proximal} if for any $x,y\in X$ we have $\liminf d(T^nx,T^ny)=0$.  A notion ``contrary'' to proximality is that of minimality: we say that $X$ is \emph{minimal} if is does not contain a non-trivial closed $T$-invariant subset. The only topological dynamical system that is at the same time minimal and proximal is $(\{\ast\}, id)$. The dynamical properties of $X_\eta$ are reflected in the properties of $\mathscr{B}$, e.g., we have the following.
\begin{enumerate}
\item $X_\eta$ is proximal $\iff$ $\ldots 000 \ldots \in X_\eta$ $\iff$ $\mathscr{B}$ contains an infinite pairwise coprime subset (Theorem B in~\cite{MR3803141}),
\item $X_\eta$ is minimal $\iff$ $\eta$ is a Toeplitz sequence (i.e.\ for any $n\in\mathcal{F}_\mathscr{B}$ there exists $s\geq 1$ such that $n+s\Z\subseteq \mathcal{F}_\mathscr{B}$) $\iff$ $\mathscr{B}$ does not contain any set of the form $c\mathcal{A}$, where $c\geq 1$ and $\mathcal{A}$ is infinite and pairwise coprime (Theorem B in~\cite{DKKu}, see also earlier partial results: Corollary 1.4 in~\cite{MR3803141} an Theorem B in~\cite{MR3947636}).
\item $X_\eta$ is \emph{essentially minimal}, i.e.\ has a unique closed invariant minimal subset. Moreover, there exists $\mathscr{B}^*\subseteq \mathbb{N}$ such that $X_{\eta^*}$ (where $\eta^*=\mathbf{1}_{\mathcal{F}_{\mathscr{B}^*}}$) is the unique minimal invariant subset of $X_\eta$ (see Theorem~A in~\cite{MR3803141} for the essential minimality and Corollary 5 and Lemma 3 in~\cite{MR4280951} for the form of the unique minimal subset).
\end{enumerate}
For a topological dynamical system $(X,T)$, we will denote by $\mathcal{M}(X,T)$ (or $\mathcal{M}(X)$ if $T$ is clear from the context) the set of all Borel probability $\sigma$-invariant measures on $X$. Crucial for understanding (in fact, even for formulating) results related to invariant measures in the $\mathscr{B}$-free setup, is the so-called Mirsky measure $\nu_\eta\in\mathcal{M}(X_\eta)$. Let $(N_i)$ be any sequence realizing the upper density of $\mathcal{F}_{\mathscr{B}}$, i.e.\ satisfying $\overline{d}(\mathcal{F}_\mathscr{B}):=\limsup_{N\to\infty}\frac{1}{N}|\mathcal{F}_\mathscr{B}\cap [1,N]|=\lim_{i\to\infty}\frac{1}{N_i}|\mathcal{F}_\mathscr{B} \cap [1,N_i]|$. Davenport and Erd\"os~\cite{MR43835} proved that 
\[
\overline{d}(\mathcal{F}_\mathscr{B})=\delta(\mathcal{F}_\mathscr{B}):=\lim_{N\to\infty}\frac{1}{\log N}\sum_{k\leq N, k\in\mathcal{F}_\mathscr{B}}1/k
\] 
(that is, the limit defining $\delta(\mathcal{F}_\mathscr{B})$, called the \emph{logarithmic density} of $\mathcal{F}_\mathscr{B}$, exists and equals the upper density of $\mathcal{F}_\mathscr{B}$). It turned out that then it is not hard to show (see Proposition~E in~\cite{MR3803141}) that, in fact, $\eta$ is quasi-generic for a certain invariant measure along $(N_i)$: the limit $\frac{1}{N_i}\sum_{n\leq N_i}\delta_{\sigma^n \eta}$ exists in the weak topology. The resulting measure is called the \emph{Mirsky measure} and is usually denoted by $\nu_\eta$ (for more details about $\nu_\eta$ see the relevant paragraph of Section~\ref{Bintro}). There are several natural questions related to the set of shift-invariant measures on $X_\eta$ and $\widetilde{X}_\eta$. We have the following result on the form of $\mathcal{M}(\widetilde{X}_\eta,\sigma)$:
\begin{enumerate}
\setcounter{enumi}{3}
\item For any $\mathscr{B}\subseteq \mathbb{N}$, 
\[
\mathcal{M}(\widetilde{X}_\eta,\sigma)=\{M_\ast(\mu) :\mu\in\mathcal{M}(\{0,1\}^\Z\times \{0,1\}^\Z,\sigma\times\sigma),\ (\pi_1)_\ast(\mu)=\nu_\eta\},
\] 
where $\pi_1\colon \{0,1\}^\Z\times\{0,1\}^\Z\to\{0,1\}^\Z$ stands for the projection onto the first coordinate, $M_\ast(\mu)$ is the image of $\mu$ via $M$, while $(\pi_1)_\ast)(\mu)$ is the image of $\mu$ via $\pi$. (See~\cite{MR3356811} for the Erd\"os case and~\cite{MR3803141} for the general case. Note that, by the definition of the hereditary closure, the inclusion ``$\supseteq$'' is here automatic.) 
\end{enumerate}
Recall that there are two classical notions of entropy given a topological dynamical system $(X,T)$: topological entropy $h_{top}(X,T)$ and measure-theoretic entropy $h(X,\mu)$. In case of a subshift $(X,\sigma)$, $h_{top}(X,\sigma)$ measures the speed of the exponential growth of the number of distinct $n$-blocks appearing in $X$, while $h(X,\mu)$ takes into account certain weights related to $\mu\in\mathcal{M}(X)$. These quantities are linked via the so-called variational principle: 
\[ h_{top}(X,T)=\sup\{h(X,\mu) : \mu\in\mathcal{M}(X,T)\}.\] 
If there is only one measure realizing this supremum, we say that the system is \emph{intrinsically ergodic}. We have the following results for $\widetilde{X}_\eta$:
\begin{enumerate}
\setcounter{enumi}{4}
\item $h_{top}(\widetilde{X}_\eta,\sigma)=\overline{d}(\mathcal{F}_\mathscr{B})$ (see~\cite[Proposition K]{MR3803141} and earlier results from~\cite{MR3430278,MR3428961}).
\item The subshift $\widetilde{X}_\eta$ is intrinsically ergodic and the unique measure of maximal entropy for $\widetilde{X}_\eta$ equals $M_\ast(\nu_\eta\otimes B_{1/2})$, where $B_{1/2}$ is the Bernoulli measure $(1/2,1/2)$ on $\{0,1\}^\Z$ (see~\cite{MR3803141} and the earlier papers~\cite{MR3430278,MR3356811}).
\end{enumerate}

Recently, in~\cite{ADJ}, the analogues of the above result on invariant measure were proved for $X_\eta$ under the extra assumption that $\eta^*$ is a regular Toeplitz sequence. Namely:
\begin{enumerate}
\item[4'.] For any $\mathscr{B}\subseteq\mathbb{N}$ such that $\eta^*$ is a regular Toeplitz sequence, we have 
\[
\mathcal{M}(X_\eta,\sigma)=\{N_\ast(\mu) : \mu\in\mathcal{M}((\{0,1\}^\Z)^3, (\pi_{1,2})_\ast(\mu)=\nu_{\eta^*}\triangle\nu_\eta\},
\] 
where:
	\begin{itemize}
	\item $\pi_{1,2}\colon (\{0,1\}^\Z)^3$ stands for the projection onto the first two coordinates,
	\item $N\colon (\{0,1\}^\Z)^3 \to \{0,1\}^\Z$ is given by $N(w,x,y)=(1-y)w+yx$,
	\item $\nu_{\eta^*}\triangle\nu_\eta$ stands for the weak limit of $\frac{1}{N_i}\sum_{n\leq N_i}\delta_{(\sigma^n\eta^*,\sigma^n\eta)}$ (where $(N_i)$ is any sequence realizing $\overline{d}(\mathcal{F}_\mathscr{B})$).
	\end{itemize}
\item[5']
For any $\mathscr{B}\subseteq\mathbb{N}$ such that $\eta^*$ is a regular Toeplitz sequence, $h_{top}(X_\eta,\sigma)=\overline{d}(\mathcal{F}_\mathscr{B})-d(\mathcal{F}_{\mathscr{B}^*})$.
\item[6']
For any $\mathscr{B}\subseteq\mathbb{N}$ such that $\eta^*$ is a regular Toeplitz sequence, the subshift $X_\eta$ is intrinsically ergodic and the unique invariant measure of maximal entropy for $X_\eta$ equals $N_\ast((\nu_{\eta^*}\triangle\nu_\eta)\otimes B(1/2))$.
\end{enumerate}
Since the focus of this paper will be not only on invariant measures but also on thermodynamic formalism, let us recall here the basic relevant notions and related known results in the $\mathscr{B}$-free setup. Given a subshift $X$, a measure $\mu\in\mathcal{M}(X)$ is said to have the \emph{Gibbs property} if there exists a constant $c>0$ such that for any block $C$ of length $n$ appearing in $X$ we have $\mu(C)\geq c\cdot e^{-n h_{top}(X)}$ (see~\cite{We0}). The following was shown in~\cite{MR4289651}.
\begin{enumerate}
\setcounter{enumi}{6}
\item For any $\mathscr{B}\subseteq \N$ such that $\nu_\eta$ is not atomic, the measure of maximal entropy for $\widetilde{X}_\eta$ does not have the Gibbs property.
\end{enumerate}
(Roughly speaking, a non-atomic Mirsky measure corresponds to a set $\mathscr{B}\subseteq \mathbb{N}$ that cannot be ``reduced'' to a finite set, see~\cite{MR4289651} for the details.) What makes the above result interesting is that the unique maximal entropy measures for sofic systems do have the Gibbs property~\cite{We0}. Thus, $\widetilde{X}_\eta$ is rarely a sofic system.\footnote{We skip here the lengthy technical definition(s) of sofic systems and refer the reader to~\cite{Li-Ma, We0}} However, it can be approximated by a descending family of sofic systems $\widetilde{X}_{\eta_K}$, where $\eta_K=\mathbf{1}_{\mathcal{F}_{\mathscr{B}_K}}$ and $\mathscr{B}_K=\{b\in\mathscr{B} : b<K\}$. We have $\widetilde{X}_{\eta_K}\supseteq \widetilde{X}_{\eta_{K+1}}$ for each $K\geq 1$, $\widetilde{X}_\eta \subseteq \bigcap_{K\geq 1}\widetilde{X}_{\eta_K}$ and $\mathcal{M}(\widetilde{X}_\eta)=\mathcal{M}(\bigcap_{K\geq 1}\widetilde{X}_{{\eta}_K})$ (see, e.g., Theorem~30 in~\cite{MR4544150}). (As we will see in this paper -- and is implicit in~\cite{ADJ} -- if $\mathscr{B}\subseteq \N$ is such that $\eta^*$ is a regular Toeplitz sequence then $X_\eta$ can be approximated by sofic shifts in a similar manner to $\widetilde{X}_\eta$.)

Results 4 and 4' were our motivation to introduce the following definitions. Let $X\subseteq \{0,1\}^\Z$ be a subshift.

\begin{definition}
We say that $X$ is a \emph{measure-theoretically subordinate subshift} if there exists a measure $\nu\in\mathcal{M}(\{0,1\}^\Z,\sigma)$ such that 
\[
\mathcal{M}(X,\sigma)=\{M_\ast(\mu) : \mu\in\mathcal{M}(\{0,1\}^\Z\times \{0,1\}^\Z, \sigma\times \sigma),\ (\pi_1)_\ast(\mu)=\nu\}.
\] 
Each measure $\nu$ as above is called a \emph{base measure}.
\end{definition}
Thus, the set of invariant measures of a measure-theoretically subordinate subshift consists of all measures located ``under'' any base measure.
\begin{remark}
In~\cite{doktorat}, the above notion appeared under a different name -- such subshifts were called convolution systems. We believe that the new name is more appropriate and carries the intuition of the existence of a measure that ``dominates'' all other measures. In fact, this is a measure-theoretic counterpart of subordinate subshifts introduced in~\cite{KuKwiLi}, cf.\ Remark~\ref{rk:podshifty} and the preceding definitions.
\end{remark}
\begin{definition}
We say that $X$ is a \emph{sandwich measure-theoretically subordinate subshift} if there exists a measure $\rho\in\mathcal{M}(\{0,1\}^\Z\times \{0,1\}^\Z,\sigma\times \sigma)$ such that 
\[
\mathcal{M}(X,\sigma)=\{N_\ast(\mu) : \mu\in\mathcal{M}((\{0,1\}^\Z)^3,\sigma^{\times 3}),\ (\pi_{1,2})_\ast(\mu)=\rho\}.
\] 
Each measure $\rho$ as above is called a \emph{pre-base measure}. If additionally $\rho(\{(w,x): w\leq x \text{ coordinatewise}\})=1$ then $\rho$ is called a \emph{base measure}.
\end{definition}
Thus, the set of invariant measures of a sandwich measure-theoretically subordinate subshift consists of measures ``located between the two marginals of any base measure''.

Since the above classes seem to be new and not really studied before (apart from the extensive study of $\mathscr{B}$-free systems and their hereditary closures), the aims of this paper are as follows:
\begin{enumerate}[(A)]
\item Describe basic properties of (sandwich) measure-theoretically subordinate subshifts and the relations between these two classes.
\item Provide examples of (sandwich) measure-theoretically subordinate subshifts.
\item Make first steps towards understanding (sandwich) measure-theoretically subordinate subshifts from the perspective of thermodynamic formalism.
\item Gain a deeper understanding of the mechanisms behind the earlier results on the invariant measures on $\widetilde{X}_\eta$ and $X_\eta$.
\end{enumerate}
Last, but not least:
\begin{enumerate}[(A)]
\setcounter{enumi}{4}
\item Provide a probabilistic perspective on the subject.
\end{enumerate}
Let us now summarize the main results corresponding to goals (A)-(E) and present the structure of the paper.
\paragraph{(A)} We show that measure-theoretically subordinate subshifts are, in fact, a subclass of sandwich measure-theoretically subordinate subshifts: any measure-theoretically subordinate subshift with base measure $\nu$ is a sandwich measure-theoretically subordinate subshift with base measure $\delta_{\mathbf{0}}\otimes \nu$ (Proposition~\ref{ppp}). On the other hand, the hereditary closure of any sandwich measure-theoretically subordinate subshift is a measure-theoretically subordinate subshift (Proposition~\ref{zwiazek}). The base measure for measure-theoretically subordinate subshifts turns out to be unique and ergodic. Moreover, it is the unique measure of maximal density, i.e.\ maximizing the measure of the cylinder $\{x\in X : x_0=1\}$ (Proposition~\ref{pr:1}). For sandwich measure-theoretically subordinate subshifts the marginals of any base measure are uniquely determined and ergodic (and one can describe them in terms of maximizing the measure of some cylinder sets), see Corollary~\ref{cor:1}. However, we do not know whether the base measure itself is unique (it is easy to see by taking the ergodic decomposition that if it is unique then it is automatically also ergodic).
\paragraph{(B)} The most basic ``building blocks'' to construct more complicated (sandwich) measure-theoretically subordinate subshifts (including our motivating examples of $\mathscr{B}$-free origin) are defined as follows:
\begin{itemize}
\item The hereditary closure $\widetilde{X}=M(X\times \{0,1\}^\Z)$ of any uniquely ergodic subshift $X\subseteq \{0,1\}^\Z$ is a measure-theoretically subordinate subshift whose base measure is the unique invariant measure of $X$ (Example (B) in Section~\ref{przyk}).
\item Let $N\colon (\{0,1\}^\Z)^3\to \{0,1\}^\Z$ be given by $N(w,x,y)=(1-y)w+yx$. If $Z\subseteq \{0,1\}^\Z\times \{0,1\}^\Z$
is uniquely ergodic then $N(Z\times \{0,1\}^\Z)$ is a sandwich measure-theoretically subordinate subshift whose pre-base measure is the unique invariant measure of $Z$ (Example (B) in Section~\ref{ppp}, proved in Section~\ref{bicobisu}).
\end{itemize}
Moreover, there is a certain approximation procedure allowing us to obtain new (sandwich) measure-theoretically subordinate subshifts: under some technical assumptions, the intersection of a descending family of (sandwich) measure-theoretically subordinate subshifts remains a (sandwich) measure-theoretically subordinate subshift, see Section~\ref{ap1} and Section~\ref{ap2}.

\begin{remark}
Most of the results related to (A) and (B) are covered separately for measure-theoretically subordinate subshifts and for sandwich measure-theoretically subordinate subshifts (in Section~\ref{CS} and Section~\ref{BCS}, respectively). An important tool -- interesting also on its own -- is a charaterization of (sandwich) measure-theoretically subordinate subshifts in terms of generic / quasi-generic points (see Section~\ref{gepo} and Section~\ref{ggg}). This is used, e.g., in the proof of the fact that the hereditary closure of a sandwich measure-theoretically subordinate subshift is a measure-theoretically subordinate subshift and to show the uniqueness and the ergodicity of the marginals of base measure for sandwich measure-theoretically subordinate subshifts.
\end{remark}

\paragraph{(C)} Recall that for a subshift $X\subseteq \{0,1\}^\Z$ and a H\"older continuous potential $\varphi\colon X\to\mathbb{R}$, one defines the corresponding \emph{topological pressure} as 
\[
\mathcal{P}_{X,\varphi}:=\lim_{n\to\infty}\frac{1}{n}\log\sum_{A\in\mathcal{L}_n(X)}2^{\sup_{x\in A}\varphi^{(n)}(x)},
\]
where $\mathcal{L}_n(X)$ stands for the family of all $n$-blocks appearing in $X$ and $\varphi^{(n)}(x)=\varphi(x)+\varphi(\sigma x)+\ldots + \varphi(\sigma^{n-1}x)$, $n\geq 1$. If $\varphi(x)$ depends only on $x|_{[0,k]}$ for some $k\geq 0$, we say that it is \emph{local} and we write $\varphi(x_0\dots x_k)$  instead of $\varphi(x)$.  Recall also the classical variational principle for the topological pressure:
\[
\mathcal{P}_{X,\varphi}=\sup\{h(T,X,\mu)+\int \varphi\, d\mu : \mu\in\mathcal{M}(X,T)\}=\sup\{h(T,X,\mu)+\int \varphi\, d\mu : \mu\in\mathcal{M}^e(X,T)\}.
\] 
If $X$ is a subshift then this supremum is always attained by some measure $\mu$ which is then called an \emph{equilibrium state}. If this measure is unique, we say that there is a \emph{unique equilibrium state}. Note that one can see the topological pressure as a ``weighted version'' of the topological entropy (for $\varphi\equiv 0$  we clearly have $\mathcal{P}_{X,\varphi}=h_{top}(X)$ and in this case the equilibrium measures are just the measures of maximal entropy).

Given a subshift $X$ and a potential $\varphi$, the two basic tasks are to compute $\mathcal{P}_{X,\varphi}$ and to determine whether there is a unique equilibrium state. Although it is well understood how to compute the topological pressure for local potentials for sofic shifts~\cite{MR390180} and (as noted above) $\mathscr{B}$-free subshifts are approximated by sofic shifts, this is not very helpful in computing $\mathcal{P}_{\widetilde{X}_\eta,\varphi}$. The underlying reason is that the complexity of the sofic approximations, and thus the size certain corresponding matrices, grows very fast. Therefore, finding meaningful relationships, even at the level of the values of the topological pressure for sofic approximations, proves to be challenging. 

In Section~\ref{se:jedna}, we develop a framework that allows us to deal with the simplest case when $\varphi(x)=\varphi(x_0)$ and to provide formulas and prove the uniqueness of the corresponding equilibrium states for (sandwich) measure-theoretically subordinate subshifts with zero entropy base measures. 

Already for $\varphi(x)=\varphi(x_0,x_1)$ (i.e.\ when $\varphi$ depends on two consecutive coordinates) the whole picture gets much more complicated (unless $\varphi(0,0)+\varphi(1,1)=\varphi(1,0)+\varphi(0,1)$ when we can easily get back to the setting of potentials depending on $x_0$ only, see the very beginning of Section~\ref{NNNN}). In Section~\ref{NNNN}, after imposing extra assumptions on the subshift, we get meaningful results for any  $\varphi(x)=\varphi(x_0,x_1)$. Namely, based on ideas coming form the case of the full $0$-$1$-shift, we provide a formula for the topological pressure for sandwich measure-theoretically subordinate subshifts with periodic bounds, see Theorem~\ref{OKRESOWE} in Section~\ref{DDD} (measure-theoretically subordinate subshifts are covered in the succeeding Corollary~\ref{OKRESOWE1}). Then, in Section~\ref{APEB}, we apply these results to (sandwich) measure-theoretically subordinate subshifts that can be approximated by the ``periodic'' ones. These results are tailored to fit the $\mathscr{B}$-free setup which we cover in Section~\ref{apb}. Our main results in this section are as follows.
\begin{enumerate}[(i)]
\item We present formulas for the topological pressure $\mathcal{P}_{\widetilde{X}_\eta,\varphi}$ and (under some extra assumptions on $\mathscr{B}$) for $\mathcal{P}_{X_\eta,\varphi}$.
\item In the special case when $2\in\mathscr{B}$, we show how to compute $\mathcal{P}_{\widetilde{X}_\eta,\varphi}$ for $\varphi(x)=\varphi(x_0,x_1,x_2,x_3)$. The main idea is to study a certain induced transformation and reduce the situtation to that of a potential depending on two coordinates (for a new $\mathscr{B}$-free system related to the original one). Note that Lin and Chen in~\cite{LC} were only able to cover the case when $2\in\mathscr{B}$ and $\varphi(x)=\varphi(x_0,x_1)$.
\item Last but not least, we study the asymptotics of terms in our formula(s) in the Erd\"os case (i.e.\ when $\mathscr{B}$ is pairwise coprime, with summable series of reciprocals).
\end{enumerate}
\paragraph{(D)} This goal is clearly achieved via the results on (sandwich) measure-theoretically subordinate subshifts. 
Let us note here that the intrinsic ergodicity of $\widetilde{X}_\eta$ was proved before the description of invariant measures was obtained (the same goes for $X_\eta$). We show that, in fact, the first of these results is a consequence of the latter one: any (sandwich) measure-theoretically subordinate subshift with a zero entropy base measure is necessarily intrinsically ergodic (see Theorem~\ref{toto} and Corollary~\ref{GGG}). For a detailed discussion see Remark~\ref{commentsB} and Remark~\ref{commentsB1}.

\paragraph{(E)} We finish the paper with an appendix devoted to a probabilistic approach to dynamical systems under consideration as we believe that (some of) our results might be of an interest to people working in this area. Since the results are not stronger than the ones obtained with purely ergodic tools, we will give only a sample of results and -- for simplicity -- stick to measure-theoretically subordinate subshifts only.

\section{Definitions and notation}
\subsection{Densities}

Given a subset $A\subseteq \N$, we define its natural \emph{lower density} and \emph{upper density} in the following way:
\[
\underline{d}(A):=\liminf_{n\to\infty}\frac{1}{n}|[1,n]\cap A| \leq \limsup_{n\to\infty}\frac{1}{n}|[1,n]\cap A|=:\overline{d}(A).
\]
If the two above quantities are equal, we speak of the \emph{(natural) density} of $A$ and denote their commmon value by $d(A)$. The \emph{lower logarithmic density} and \emph{upper logarithmic density} of $A\subseteq \N$ are defined as follows:
\[
\underline{\delta}(A):=\liminf_{n\to\infty}\frac{1}{\log n}\sum_{k\in A,k\leq n}\frac{1}{k} \leq\limsup_{n\to\infty}\frac{1}{\log n}\sum_{k\in A,k\leq n}\frac{1}{k}=: \overline{\delta}(A).
\]
If the two above quantities are equal, we speak of the \emph{logarithmic density} of $A$ and denote their commmon value by $\delta(A)$. For $A\subseteq \Z$ by the  (lower / upper) density / logarithmic density we mean the corresponding quantity for $A\cap \N$. For any $A\subseteq \Z$, we have
\(
\underline{d}(A) \leq \underline{\delta}(A) \leq \overline{\delta}(A) \leq \overline{d}(A).
\)

\subsection{Dynamical systems}
\paragraph{Topological and measure-theoretic dynamics}
A \emph{topological dynamical system} is a pair $(X,T)$, where $T$ is a homeomorphism of a compact metric space $X$. We equip $X$ with the sigma-algebra of Borel subsets and denote the set of probability Borel $T$-invariant measures on $X$ by $\mathcal{M}(X,T)$ or just $\mathcal{M}(X)$ if $T$ is clear from the context. We say that $\mu$ is \emph{ergodic} if $\mu(A\triangle T^{-1}A)=0$ implies $\mu(A)\in \{0,1\}$. The subset of $\mathcal{M}(X,T)$ consisting of ergodic measures is denoted by $\mathcal{M}^e(X,T)$ or $\mathcal{M}^e(X)$. For each $\mu\in\mathcal{M}(X,T)$, the triple $(X,T,\mu)$ is a \emph{measure-theoretic dynamical system}. The \emph{topological support} of $\mu\in\mathcal{M}(X)$, denoted by $\text{supp }\mu$, is the smallest closed subset of $X$ with $\mu(X)=1$. 

Given two topological dynamical systems $(X,T)$ and $(Y,S)$, we say that $(Y,S)$ is a \emph{factor} of $(X,T)$ whenever there exists a continuous map $\pi\colon X\to Y$ such that $\pi\circ T = S\circ \pi$. A measure-theoretic dynamical system $(Y,S,\nu)$ is factor of $(X,T,\mu)$ whenever there exists a measurable map $\pi\colon X\to Y$ such that $\pi\circ T = S\circ \pi$ $\mu$-a.e.\ and the image $\pi_\ast (\mu)$ of $\mu$ via $\pi$ equals $\nu$.

\paragraph{Subshifts}
We will often deal with subshifts, i.e.\ $X\subseteq \mathcal{A}^\Z$, where $\mathcal{A}$ is a finite alphabet (usually, we will have $\mathcal{A}=\{0,1\}^\Z$), such that $X$ is closed and invariant under the \emph{left shift} $\sigma\colon \mathcal{A}^\Z\to\mathcal{A}^\Z$. Given $a\in\mathcal{A}^\Z$, we denote by $X_a$ the orbit closure of $a$ under $\sigma$, i.e.\ $X_a=\overline{\{\sigma^n a : n\in\mathbb{Z}\}}$. (If $a\in \mathcal{A}^\N$ then $X_a$ is defined as the subset of all bi-infinite sequences $x$ over the alphabet $\mathcal{A}$ such that any finite block that appears in $x$, appears also in $a$.)
We will use the following notation for \emph{cylinder sets} in a subshift $X\subseteq \{0,1\}^\Z$: given $A\subseteq \Z$, we set
\(C_A:=\{x\in X : x_i=1 \text{ for }i\in A\}\). Moreover, for any $a,\dots, a_{n-1}\in\mathcal{A}$, $a_0 \dots a_{n-1}:=\{x\in X : x_i=a_i \text{ for }i=0,\dots, n-1\}$. In particular, we have $C_{\{0\}}=1$. Moreover, we will use the following convention for Cartesian products of cylinders: if $C_0,C_1\in \mathcal{A}^\ell$ for some $\ell\geq 1$ then instead of $C_0\times C_1\subseteq \mathcal{A}^\ell\times \mathcal{A}^\ell$ we will write $\begin{matrix}C_0 \\ C_1\end{matrix}$. For $a\in\mathcal{A}$, we denote by $\mathbf{a}$ the two-sided sequence whose all entries are equal to $a$. 

\paragraph{Heredity}A subshift $X\subseteq \{0,1\}^\Z$ is called \emph{hereditary} if it is closed under taking sequences in $\{0,1\}^\Z$ that are smaller or equal coordinatewise that some member of $X$:
\[
(x\in X \text{ and }z\in \{0,1\}^\Z, z\leq x) \implies z\in X.
\]
The \emph{hereditary closure} $\widetilde{X}$ of $X$ is the smallest hereditary subshift containing $X$, i.e.\ $\widetilde{X}=M(X\times \{0,1\}^\Z)$. Clearly, 
\begin{equation}\label{vv}
X\text{ is hereditary if and only if }X=M(X\times \{0,1\}^\Z).
\end{equation}

Another closely related notion of so-called subordinate subshift was introduced in~\cite{KuKwiLi}. We say that a subshift $X\subseteq \{0,1\}^\Z$ is \emph{subordinate} if 
\[
X=\overline{\{\sigma^ny :y\leq x, n\in\Z\}}
\]
for some $x\in X$. More precisely, we then say that $X$ is a \emph{subordinate system under $x$}. Note that
\[
X\text{ is subordinate under }x\text{ if and only if }X=\widetilde{X}_x=M(X_x\times \{0,1\}^\Z)
\]
(indeed, taking the closures in the product topology commutes with the coordinatewise multiplication $M$, the same applies to any finite code). Thus, any subordinate subshift is hereditary.

Let us now introduce ``two-sided counterparts'' of the above notions. 
\begin{definition}
We say that $X\subseteq \{0,1\}^\Z$ is a \emph{sandwich hereditary subshift} if $X=N(Z\times \{0,1\}^\Z)$ for some subshift $Z\subseteq\{0,1\}^\Z\times \{0,1\}^\Z$, cf.~\eqref{vv}. Moreover,
we say that $ X $ is a \emph{sandwich subordinate subshift} if there exists a pair $(w,x)$ with $w\leq x$ such that
\begin{equation*}
 X =\overline{\{\sigma^ny: w\leq y\leq x\}}.
\end{equation*}
More precisely, we say then that $ X $ is a \emph{sandwich subordinate subshift between $w$ and $x$}.
\end{definition} 
Note that, by the comment from the preceding paragraph on the finite codes, the following are equivalent:
\begin{itemize}
\item $X$ is a sandwich subordinate subshift between $w$ and $x$,
\item $X=N(X_{(w,x)}\times \{0,1\}^\Z)$,
\end{itemize}
where by $X_{(w,x)}$ we mean the orbit closure of the pair $(w,x)$ under $\sigma\times\sigma$. Thus, any sandwich subordinate subshift is a sandwich hereditary subshift. Sometimes for $w,x\in \{0,1\}^\Z$, with $w\leq x$, we will write
\(
[w,x] \text{ instead of }\{\sigma^ny :w\leq y\leq x,n\in\Z\}.
\)
In particular, $\overline{[\mathbf{0},x]}$ stands for $\widetilde{X}_x$.

\begin{remark}\label{rk:podshifty}
Let us summarize here our nomenclature. Alltogether, we have defined 6 types of subshifts:
\begin{itemize}
\item (sandwich) measure-theoretically subordinate subshifts,
\item (sandwich) subordinate subshifts,
\item (sandwich) hereditary subshifts.
\end{itemize}
It is clear that (sandwich) measure-theoretically subordinate subshifts are a measure-theoretic counterpart of (sandwich) subordinate subshifts: instead of ``bounds'' given by sequences, we have ``bounds'' given by invariant measures. 
\end{remark}

\paragraph{Generic points}
Given a topological dynamical system $(X,T)$, we say that $x\in X$ is \emph{generic} for $\mu\in\mathcal{M}(X.T)$ if $\frac{1}{N}\sum_{n\leq N}\delta_{T^nx}\to\mu$ in the weak topology. If the convergence takes place along a subsequence $(N_i)$, we say that $x$ is \emph{quasi-generic} for $\mu$ along $(N_i)$. For subshifts, $x$ is (quasi-)generic for $\mu$ if the frequencies of blocks (over the given alphabet) on $x$ exist and are equal to the values of $\mu$ of the corresponding cylinder sets.

\paragraph{Minimality and proximality}
A topological dynamical system $(X,T)$ is \emph{minimal} if there are no closed non-empty proper subsets of $X$ invariant under $T$. If $(X,T)$ has a unique minimal subsystem, we say that it is \emph{essentially minimal}. We say that $(X,T)$ is \emph{proximal} if for any points $x,y\in X$ we have $\liminf_{n\to\infty}d(T^nx,T^ny)=0$.

\paragraph{Toeplitz subshifts}
A sequence $x\in\mathcal{A}^\Z$ is called \emph{Toeplitz} if for any $k\in\Z$ there exists $s\in \N$ such that $x|_{k+s\Z}$ is constant. The orbit closure $X_x$ of a Toeplitz sequence is called a \emph{Toeplitz subshift}. Any Toeplitz subshift is minimal~\cite{MR756807}.
A Toeplitz sequence $x$ is called \emph{regular} if $\lim_{r\to\infty}d\left(\{k\in \Z : x|_{k+s\Z} \text{ is constant for some } 1\leq s\leq r\} \right)=1$. (This is equivalent to the usual definition via so-called period structure which can be found, e.g.,~in~\cite{MR2180227}.) The remaining Toeplitz sequences are called~\emph{irregular}. For more information on Toeplitz sequences, we refer the reader, e.g., to~\cite{MR2180227}.

\paragraph{Topological entropy}
The \emph{topological entropy} of $(X,T)$ will be denoted by $h(X,T)$ (or $h(X)$ if the map is clear from the context). In case of subshifts, $h(X,T)=\lim_{n\to\infty}\frac{1}{n}\log p_n$, where $p_n$ stands for the number of distinct blocks of length $n$ ($n$-blocks) appearing in $X$. For the general definition, we refer the reader, e.g., to~\cite{MR2809170}. The \emph{measure-theoretic entropy} of $(X,T,\mu)$ will be denoted by $h(X,T,\mu)$ (or just $h(T,\mu)$ or even $h(\mu)$); for the definition also see~\cite{MR2809170}. There is the following \emph{variational principle}: 
\begin{equation}\label{VP1}
h(X,T)=\sup\{h(X,T,\mu) : \mu\in\mathcal{M}(X,T)\}=\sup\{h(X,T,\mu) : \mu\in\mathcal{M}^e(X,T)\}.
\end{equation}
If $X$ is a subshift, there is always at least one measure $\mu$ realizing this supremum. If this measure is unique, we say that $(X,T)$ is \emph{intrinsically ergodic}.

\paragraph{Topological pressure}
Given a \emph{potential} (i.e.\ a measurable function) $\varphi\colon X\to \R$, the corresponding \emph{topological pressure} of $(X,T)$ will be denoted by $\mathcal{P}_{X,\varphi}$. Again, we skip the lengthy general definition (see, e.g.,~\cite{MR2656475}) and recall that for the subshift $X$ and the potential $\varphi$, the topological pressure is defined as follows: 
\[
\mathcal{P}_{X,\varphi}=\lim_{n\to\infty}\frac{1}{n}\log\sum_{A\in\mathcal{L}_n(X)}2^{\sup_{x\in A}\varphi^{(n)}(x)},
\] 
where $\mathcal{L}_n(X)$ stands for the family of all $n$-blocks appearing in $X$ and $\varphi^{(n)}(x)=\varphi(x)+\varphi(\sigma x)+\ldots + \varphi(\sigma^{n-1}x)$, $n\geq 1$. If $\varphi(x)$ depends only on $x|_{[0,k]}$ for some $k\geq 0$, we will write $\varphi(x_0\dots x_k)$  instead of $\varphi(x)$.

Recall also the variational principle for the topological pressure:
\[
\mathcal{P}_{X,\varphi}=\sup\{h(T,X,\mu)+\int \varphi\, d\mu : \mu\in\mathcal{M}(X,T)\}=\sup\{h(T,X,\mu)+\int \varphi\, d\mu : \mu\in\mathcal{M}^e(X,T)\}.
\] 
If $X$ is a subshift then this supremum is always attained by some measure $\mu$ which is then called an \emph{equilibrium state}. If this measure is unique, we say that there is a \emph{unique equilibrium state}.

\paragraph{Induced maps}
Given a measure-theoretic dynamical system $(X,T,\mu)$ and a measurable subset $A\subseteq X$ with $\mu(A)>0$, we define the corresponding \emph{induced map} $T_A$, acting on $A$ (with the $\sigma$-algebra of measurable subsets inherited from $X$) and defined by $T_Ax=T^{n_A(x)}x$, where $n_A(x)=\min\{n\geq 1 : T^nx\in A\}$. Because of the Poincar\'e recurrence theorem, this map is well-defined and it is measure-preserving with respect to $\mu_A$ given by $\mu_A(C)=\frac{\mu (C)}{\mu(A)}$ for any measurable $C\subseteq A$. If additionally $T$ is ergodic then $h(X,T,\mu)=\mu(A)\cdot h(A,T_A,\mu_A)$ (this is the so-called Abramov's formula, see, e.g.,~\cite{petersen_1983}, Section 6.1.C). Moreover, if $(X,T)$ is a topological dynamical system and $A\subseteq X$ is measurable and such that $\mu(A)>0$ for any $\mu\in\mathcal{M}(X,T)$ then there is 1-1-correspondence between the elements of $\mathcal{M}(X,T)$ and $\mathcal{M}(A,T_A)$.

\paragraph{Joinings}
Given two dynamical systems $(X,T,\mu)$, $(Y,S,\nu)$, we say that a measure $\rho$ on the product space $X\times Y$ (equipped with the product $\sigma$-algebra) is a \emph{joining} of $T$ and $S$ if $\rho$ is invariant under $T\times S$ and its projection onto the first and second coordinate equals $\mu$ and $\nu$ respectively. We then sometimes write $\rho=\mu\vee\nu$. The product measure $\mu\otimes \nu$ is always a joining. Moreover, if $(Y,S,\nu)=(X,T,\mu)$, we have the so-called \emph{diagonal (self-)joining} $\mu\triangle\mu$ given by $\mu\triangle \mu(A\times B)=\mu(A\cap B)$ for measurable $A,B\subseteq X$. Last but not least, for $R\colon (X,T,\mu)\to(Y,S,\nu)$, we will denote by $\triangle_R$ the graph joining of $(S,Y,\nu)$ and $(T,X,\mu)$ given by $\triangle_R(A\times B)=\mu(R^{-1}A\cap B)$ for any measurable $A\subseteq Y$, $B\subseteq X$. (Note that usually $\triangle_R$ stands for the joining of $T$ and $S$ where the coordinates are written in the opposite order.)

\subsection{Dynamical diagrams}
A certain category theory language related to dynamical systems and factoring maps between them was introduced in~\cite{ADJ}. Let us give a short summary here, as it will be convenient also in the present paper:
\begin{itemize}
\item
an \emph{object} is a triple of the form $(X,T,\mathcal{M}_X)$, where $(X,T)$ is a topological dynamical system and $\emptyset\neq\mathcal{M}_X\subseteq \mathcal{M}(X)$; if $\mathcal{M}_X=\mathcal{M}(X)$, we skip it and write $(X,T)$ instead of $(X,T,\mathcal{M}(X))$;
\item
a \emph{morphism} from $(X,T,\mathcal{M}_X)$ to $(Y,S,\mathcal{M}_Y)$ is a map $f\colon (X,T,\mathcal{M}_X)\to (Y,S,\mathcal{M}_Y)$, i.e.\ $f\colon X_0\to Y$, where $X_0\subseteq X$ is $T$-invariant with $\mu(X_0)=1$ for any $\mu\in\mathcal{M}_X$, $f_\ast(\mathcal{M}_X)\subseteq \mathcal{M}_Y$ and $S\circ f=f\circ T$ on $X_0$.
\end{itemize}
Any graph whose vertices are the above-defined objects and arrows denote morphisms is called a \emph{dynamical diagram}. We identify two morphisms $f,g\colon (X,T,\mathcal{M}_X)\to (Y,S,\mathcal{M}_Y)$, whenever $f$ and $g$ agree on a subset $X_0\subseteq X$ that is of full measure for any measure $\mu\in\mathcal{M}_X$. We define the \emph{composition} of morphisms $f\colon (X,T,\mathcal{M}_X)\to (Y,S,\mathcal{M}_Y)$ and $g\colon (Y,S,\mathcal{M}_Y)\to (Z,R,\mathcal{M}_Z)$ as the composition $g\circ f\colon (X,T,\mathcal{M}_X)\to (Z,R,\mathcal{M}_Z)$.

\begin{definition}
We will say that a dynamical diagram \emph{commutes} if for any choice of objects $(X,T,\mathcal{M}_X)$ and $(Y,S,\mathcal{M}_Y)$ the composition of morphisms along any path connecting $(X,T,\mathcal{M}_X)$ with $(Y,S,\mathcal{M}_Y)$ does not depend on the choice of the path, {including the trivial (zero) path}. 
\end{definition}
\begin{definition}
We will say that a morphism $f\colon (X,T,\mathcal{M}_X) \to (Y,S,\mathcal{M}_Y)$ is \emph{surjective} if $f_\ast(\mathcal{M}_X)=\mathcal{M}_Y$. We will say that a dynamical diagram is \emph{surjective} if every morphism in this diagram is surjective. If $f\colon(X,T,\mathcal{M}_X)\to(Y,S,\mathcal{M}_Y)$ is surjective, we will sometimes just say that the morphism $f$ is surjective. Notice that this notion is not the same as the surjectivity of the map $f\colon X\to Y$.
\end{definition}

\subsection{$\mathscr{B}$-free systems}\label{Bintro}
Given $\mathscr{B}\subseteq\mathbb{N}$, consider the corresponding \emph{set of multiples} $\mathcal{M}_\mathscr{B}=\bigcup_{b\in\mathscr{B}}b\Z$ and its complement $\mathcal{F}_\mathscr{B}=\mathbb{Z}\setminus \mathcal{M}_\mathscr{B}$, i.e.\ the set of \emph{$\mathscr{B}$-free integers}. We tacitly assume that $\mathscr{B}$ is primitive: if $b\divides b'$ for $b,b'\in\mathscr{B}$ then $b=b'$. For any $\mathscr{B}\subseteq \N$, there exists a primitive subset $\mathscr{B}^{prim}\subseteq \mathscr{B}$ with $\mathcal{M}_\mathscr{B}=\mathcal{M}_{\mathscr{B}^{prim}}$. 
\paragraph{Number-theoretic origins}
Sets of multiples were studied from the number-theoretic point of view already in the 1930's (for references see, e.g.,~\cite{MR1414678}). Natural examples of such sets include, e.g., the set $\mathbf{A}$ of \emph{abundant numbers} (recall that $n\in\N$ is abundant if the sum of its divisors exceeds $2n$; the set of all such numbers is closed under taking multiples and, thus, $\mathbf{A}$ is a set of multiples). In~\cite{BH} Bessel-Hagen asked whether the natural density $d(\mathbf{A})$ of $\mathbf{A}$ exists. The answer to this question is positive~\cite{abundantes,chowla,MR1574879}. A natural question arose, whether the same result holds for any set of multiples and Besicovitch~\cite{MR1512943} constructed a whole class of counterexamples. However - and this is one of the most striking results in this area from that time - we have the following.
\begin{theorem}[Davenport and Erd\"os~\cite{DE}]\label{daverd}
For any $\mathscr{B}\subseteq \N$, the logarithmic density of $\mathcal{M}_\mathscr{B}$ exists. In fact,
\[
\delta(\mathcal{M}_\mathscr{B})=\underline{d}(\mathcal{M}_\mathscr{B})=\lim_{K\to\infty}d(\mathcal{M}_{\mathscr{B}_K}),
\]
where $\mathscr{B}_K=\{b\in\mathscr{B} : b<K\}$, $K\geq 1$.
\end{theorem}
Another important example of number-theoretic origin is related to the M\"obius function $\mu\colon\N\to\{-1,0,1\}$ (recall that $\mu(1)=1$, $\mu(p_1\cdot\ldots p_n)=(-1)^n$ if $p_1,\dots,p_n$ are distinct primes and $\mu(n)=0$ for all other values of $n$). Namely, if $\mathscr{B}$ be the set of squares of all primes then $\mathcal{F}_\mathscr{B}$ is the so-called set of \emph{square-free integers} and the characteristic function of its restriction to $\N$ is nothing but the square of the M\"obius function. The density of $\mathcal{F}_\mathscr{B}$, which we will often denote by $d$, in this case equals $6/\pi^2$. In fact, the frequencies of all 0-1 blocks on $\mu^2$ exist~\cite{MR28334,MR21566}.

\paragraph{Fundamental classes of $\mathscr{B}$-free systems}
Let us list here some important classes of sets~$\mathscr{B}$. Following~\cite{MR1414678}, we say that $\mathscr{B}\subseteq \N$ is:
\begin{itemize}
\item \emph{Besicovitch} if $d(\mathcal{M}_\mathscr{B})$ exists,
\item \emph{Erd\"os} if $\mathscr{B}$ is infinite, pairwise coprime and $\sum_{b\in\mathscr{B}}1/b<\infty$,
\item \emph{Behrend} if $\delta(\mathcal{M}_\mathscr{B})=1$,
\item \emph{taut} if $\delta(\mathcal{M}_{\mathscr{B}\setminus \{b\}})<\delta(\mathcal{M}_\mathscr{B})$ for any $b\in\mathscr{B}$.
\end{itemize}
The above classes appear already in the number-theoretic literature (see, e.g.,~\cite{MR1414678}), as we will see later, they are also important from the point of view of dynamics. We have the following basic relations: 
\begin{itemize}
\item if~$\mathscr{B}$ is Erd\"os or Behrend then $\mathscr{B}$ is Besicovitch,
\item $\mathscr{B}$ is taut iff there is no $c\geq 1$ and $\mathcal{A}$ being a Behrend set such that $c\mathcal{A}\subseteq \mathscr{B}$. 
\end{itemize}

\paragraph{Dynamics comes into play} 

Sarnak in 2010~\cite{PS} suggested to study the orbit closure $X_{\mu^2}$ of $\mu^2$ under the left shift $\sigma$. This subshift is a topological factor of $X_\mu$ (via $\pi(x)=|x|$) and, thus, the knowledge about this system gives us some knowledge about $\mu$ itself. Sarnak formulated a certain ``program'' for $\mu^2$. Let us recall here samples of the results for $\mu^2$ related to our present work:
\begin{itemize}
\item $\mu^2$ is a generic point for some measure on $\{0,1\}^\Z$;
we call this measure the \emph{Mirsky measure}~\cite{MR28334,MR21566} and denote it by $\nu_{\mu^2}$; the corresponding measure theoretic dynamical system is ergodic and of zero entropy,
\item $X_{\mu^2}=\{x\in \{0,1\}^\Z : |\text{supp }x\bmod p^2|<p^2 \text{ for each prime }p\}$; the latter subshift is the so-called \emph{admissible subshift},
\item $X_{\mu^2}$ is proximal and the singleton $\{\boldsymbol{0}\}$ is the unique minimal subset of $X_{\mu^2}$,
\item $h_{top}(X_{\mu^2},\sigma)=\frac{6}{\pi^2}$.
\end{itemize}
In~\cite{MR3428961}, these ideas were extended to the Erd\"os case and (among other results), the analogues of the above statements were proved for any $\mathscr{B}$ that is Erd\"os. A systematic study of $\mathscr{B}$-free systems was triggered by a general question of Boshernitzan during a conference in Toru\'n (2014), whether the same is true without the extra assumptions on $\mathscr{B}$. It turned out not to be the case and there is a whole variety of behaviours that can be observed within $\mathscr{B}$-free systems. The first results on this subject can be found in~\cite{MR3803141}, see also other works mentioned below.

\paragraph{Three subshifts}
Let $\eta:=\mathbf{1}_{\mathcal{F}_\mathscr{B}}\in\{0,1\}^\Z$ and consider:
\begin{itemize}
\item the corresponding \emph{$\mathscr{B}$-free subshift}, i.e.\ $X_\eta=\overline{\{\sigma^n \eta : n\in\Z\}}$, 
\item the \emph{hereditary closure} $\widetilde{X}_\eta=M(X_\eta\times \{0,1\}^\Z)=\{x\in \{0,1\}^\Z : x\leq y \text{ for some }y\in X_\eta\}$ of $X_\eta$,
\item the \emph{$\mathscr{B}$-admissible subshift} $X_\mathscr{B}=\{x\in \{0,1\}^\Z : |\text{supp }x \bmod b|<b \text{ for each }b\in\mathscr{B}\}$.
\end{itemize}
We have
\(
X_\eta\subseteq \widetilde{X}_\eta\subseteq X_\mathscr{B}.
\)
Both inclusions can be strict, see the relevant examples in~\cite{MR3803141}; if $\mathscr{B}$~is Erd\"os then all three subshifts are equal~\cite{MR3428961}. The number-theoretic properties of $\mathscr{B}$ and the dynamical properties of $X_\eta$ are often intertwined. Apart from the characterization of the minimality and the proximality of $X_\eta$ (see Section~\ref{se:intro}), we have, e.g., the following (see~\cite{AKL}):
\begin{itemize}
\item $\mathscr{B}$ is Erd\"os iff $X_\eta=X_\mathscr{B}$ and $h(X_\eta)>0$,
\item $\mathscr{B}$ is Behrend iff $\mathscr{B}$ contains an infinite pairwise coprime subset and $h(X_\eta)=0$.
\end{itemize}

\paragraph{Mirsky measure}
A central role in the theory of $\mathscr{B}$-free systems is played by the \emph{Mirsky measure}. There are several equivalent ways to define it.
\begin{enumerate}
\item[(A)] We set $\nu_\eta:=\varphi_\ast({m_H})$, where $H:=\overline{\{(n,n,n,\dots) : n\in\Z\}}\subseteq \prod_{b\in\mathscr{B}}\Z/b\Z$, $m_H$ is the Haar measure on $H$ and the map $\varphi\colon H\to \{0,1\}^\Z$ is given by 
\[
\varphi(h)(n)=1 \iff h_b+n\not\equiv 0\bmod b \text{ for all }b\in\mathscr{B}.
\]
\item[(B)]
In the Erd\"os case, $\eta$ is a generic point for $\nu_{\eta}$ (see~\cite{MR3428961}).
         In general, $\eta$ may fail to be a generic point; in this case, even the natural density of $\mathcal{M}_\mathscr{B}$ fails to exist~\cite{MR1512943}. However, if $(N_i)$ is a sequence realizing the lower density of $\mathcal{M}_\mathscr{B}$ then
           then $\eta$ is quasi-generic for $\nu_\eta$ along $(N_i)$,
           see~\cite[Theorem 4.1]{MR3803141}.
\item[(C)] The Mirsky measure $\nu_\eta$ is the unique measure of maximal density on $X_\eta$, see Theorem 4 and Corollary 4 in~\cite{MR3784254}, cf. also Chapter 7 in~\cite{MR3136260}.    
\end{enumerate}
        Note that if $\mathscr{B}$ is finite then $\eta$ is periodic, with period equal to $s:={\rm lcm}\mathscr{B}$. It follows immediately that $X_{\eta}$ is also finite and the unique shift-invariant probability measure on $X_{\eta}$ is given by $\nu_{\eta} = \frac{1}{s}\sum_{i = 0}^{s - 1} \sigma^i \delta_{\eta}$.
        \begin{remark}
        Notice that if we set $\eta_K:=\mathbf{1}_{\mathcal{F}_{\mathscr{B}_K}}$ then $(\eta_K)$ is non-increasing and $\eta$ is its coordinatewise limit. Moreover, by Theorem~\ref{daverd}, we have we have $\nu_{\eta_K}(1)\to\nu_\eta(1)$.
        \end{remark}
\paragraph{Invariant measures on $\widetilde{X}_\eta$}
It was shown (in~\cite{MR3356811} in the Erd\"os case and later, in~\cite{MR3803141}, in the general case) that
\[
\mathcal{M}(\widetilde{X}_\eta,\sigma) = \{M_\ast(\nu_{\eta}\vee\kappa) : \kappa \in \mathcal{M}(\{0,1\}^\Z,\sigma)\},
\]        
i.e., that we deal with a measure-theoretic subordinate subshift, with the Mirsky measure being its base measure.  The topological entropy of $\widetilde{X}_\eta$ equals $\delta(\mathcal{F}_{\mathscr{B}})$ (see~\cite[Proposition K]{MR3803141} and earlier results from~\cite{MR3430278,MR3428961}). Moreover, $\widetilde{X}_\eta$ is intrinsically ergodic and its unique measure of maximal entropy is given by $M_\ast(\nu_{\eta}\otimes B_{1/2})$ (see~\cite{MR3803141} and the earlier papers~\cite{MR3430278,MR3356811}). Cf.\ Remark~\ref{commentsB} in Section~\ref{NNN}.

\paragraph{Tautification and toeplitzisation of $\mathscr{B}$}        
        Given $\mathscr{B}\subseteq\mathbb{N}$ one defines the ''tautification'' $\mathscr{B}'$ and ''toeplitzisation'' $\mathscr{B}^*$ of $\mathscr{B}$ in the following way:
        \begin{itemize}
        \item $\mathscr{B}':=(\mathscr{B}\cup C)^{prim}$, $C=\{c\in \N : c\mathcal{A}\subseteq \mathscr{B} \text{ for some Behrend set }\mathcal{A}\}$,
        \item $\mathscr{B}^*:=(\mathscr{B}\cup D)^{prim}$, $D=\{d\in \N : d\mathcal{A}\subseteq \mathscr{B} \text{ for some infinite pairwise coprime set }\mathcal{A}\}$.
        \end{itemize}
        Since any Behrend set contains an infinite pairwise coprime subset (Theorem 3.7 in~\cite{MR3803141}), it follows by the above definitions that $\eta^*\leq \eta'\leq \eta$, where $\eta'=\mathbf{1}_{\mathcal{F}_{\mathscr{B}'}}$ and $\eta^*=\mathbf{1}_{\mathcal{F}_{\mathscr{B}^*}}$ and, thus, $\widetilde{X}_{\eta^*}\subseteq \widetilde{X}_{\eta'}\subseteq \widetilde{X}_\eta$. In fact, we have more than that: $X_{\eta^*}\subseteq X_{\eta'}\subseteq X_\eta$ (see Remark 3.22 in~\cite{MR3803141} for the first inclusion, and (27) in~\cite{MR4289651} for the second one). Moreover, $\mathscr{B}'$ is the unique taut set such that $\nu_{\eta}=\nu_{\eta'}$ (in fact $\mathcal{M}(X_\eta)=\mathcal{M}(X_{\eta'})$)~\cite[Theorem C]{MR3803141}, while $\eta^*$ is a Toeplitz sequence (see~\cite{DKKu} and~\cite{MR3947636}) and $X_{\eta^*}$ is the unique minimal subset of $X_\eta$~\cite{MR4280951}. The subshift $X_\eta$ is minimal if and only if $\eta$ itself is a Toeplitz sequence~\cite{MR3803141} (Corollary 1.4) (which may or may not be regular~\cite{IrK}).

\paragraph{Invariant measures on $X_\eta$}
Recently, it was shown in~\cite{ADJ} (Theorem B therein) that $X_\eta$ is a sandwich measure-theoretic subordinate subshift whenever $\eta^*$ is a regular Toeplitz sequence. More precisely, the pair $(\eta^*,\eta)$ is a quasi-generic point along any sequence $(N_i)$ realizing the lower density of $\mathcal{M}_\mathscr{B}$ (Lemma 2.3 in~\cite{ADJ}). The resulting limit measure is denoted by $\nu_{\eta^*}\triangle \nu_\eta$ and it is a certain off-diagnoal joining of $\nu_{\eta^*}$ and $\nu_\eta$. As both $\nu_{\eta^*}$ and $\nu_\eta$ are zero entropy measures, so is their joining $\nu_{\eta^*}\triangle\nu_\eta$. Thus, in view of Corollary~\ref{GGG}, $X_\eta$ is intrinsically ergodic, its topological entropy equals $\nu_\eta(1)-\nu_{\eta^*}(1)$ and the maximal entropy measure is given by $N_\ast((\nu_{\eta^*}\triangle\nu_{\eta})\otimes B_{1/2}$ (as proved ``by hand'' in~\cite{ADJ}).

Let us comment a little bit on the crucial tool used in~\cite{ADJ}. Namely, recall that by the Davenport-Erd\"os theorem, $\eta$ is always approximated from above by periodic sequences $\eta_K=\mathbf{1}_{\mathcal{F}_{\mathscr{B}_K}}$, while the assumption that $\eta^*$ is a regular Toeplitz sequence allows one to find useful periodic approximations $\underline{\eta}_K$ of $\eta^*$ from below. More precisely, we have the following:
\begin{itemize}
\item $\underline{\eta}_K\leq \eta^*\leq \eta\leq \eta_K$ for $K\geq 1$, $(\underline{\eta}_K)$ is non-decreasing, $(\eta_K)$ is non-increasing,\hfill\refstepcounter{equation}\textup{(\theequation)}\label{last2}%
\item $(\eta^*,\eta)$ is the coordinatewise limit of $(\underline{\eta}_K,\eta_K)$,\hfill\refstepcounter{equation}\textup{(\theequation)}\label{last1}%
\item 
$\lim_{i\to\infty}\frac{1}{N_i}\sum_{n\leq N_i}(\eta_K(n)-\eta(n))=\lim_{N\to\infty}\frac{1}{N}\sum_{n\leq N}(\eta^*(n)-\underline{\eta}_K(n))=0$
\hfill\refstepcounter{equation}\textup{(\theequation)}\label{last}%
\end{itemize}
(see Lemma 2.8 in~\cite{ADJ}).

\section{Measure-theoretic subordinate subshifts}\label{CS}
Let $X\subseteq \{0,1\}^\Z$ be a subshift and let $M\colon \{0,1\}^\Z\times \{0,1\}^\Z \to \{0,1\}^\Z$ stand for the coordinatewise multiplication of sequences. It will be convenient to translate the definition of a measure-theoretic subordinate subshift to the language of joinings: $X$ is a {measure-theoretic subordinate subshift} if there exists a measure $\nu\in\mathcal{M}(\{0,1\}^\Z)$ such that 
\begin{equation}\label{isconv}
\mathcal{M}(X)=\{M_\ast(\rho) : \rho=\nu\vee\kappa \text{ for some }\kappa\in \mathcal{M}(\{0,1\}^\Z)\}.
\end{equation}
In other words, $X\subseteq \{0,1\}^\Z$ is a measure-theoretic subordinate subshift whenever the diagram 
\[
(X\times \{0,1\}^\Z,\sigma^{\times 2},\{\nu\vee\kappa : \kappa\in\mathcal{M}(\{0,1\}^\Z)\})\xrightarrow{M}(X,\sigma)
\]
is surjective.

\subsection{Basic properties}\label{bapro}
Notice that $\nu$ satisfying~\eqref{isconv} is necessarily a member of $\mathcal{M}(X)$ since $\nu=M_\ast(\nu\otimes \delta_{\mathbf{1}})$. Moreover, we have $\delta_{\mathbf{0}}=M_\ast(\nu\otimes \delta_{\mathbf{0}})\in\mathcal{M}(X)$.
\begin{proposition}\label{pr:1}
If $X$ is a measure-theoretic subordinate subshift then $\nu\in\mathcal{M}(\{0,1\}^\Z)$ satisfying~\eqref{isconv} is unique and ergodic. Moreover, it is the unique {\it measure of maximal density}, that is the unique measure satisfying
\begin{equation}\label{eq:1}
\nu(1)=\sup_{\mu\in\mathcal{M}(X)}\mu(1).
\end{equation}
\end{proposition}
\begin{proof}
It follows by~\eqref{isconv} that 
\begin{equation}\label{teta}
\nu'(1)=(\nu\vee\kappa)\left(\begin{matrix}1 \\ 1\end{matrix}\right)\leq \nu(1)
\end{equation}
for any $\nu'\in\mathcal{M}(X)$, which yields~\eqref{eq:1}.

Suppose now that $\nu'\in\mathcal{M}(X,\sigma)$ is such that $\nu'(1)=\nu(1)$. In view of~\eqref{teta}, this yields immediately that
\begin{equation}\label{gwi}
(\nu\vee\kappa)\left(\begin{matrix}1 \\ 0\end{matrix}\right)=0.
\end{equation}
Therefore, for any finite $A\subseteq \Z$,
\begin{equation}\label{bazow1}
\nu'(C_{A})=(\nu\vee\kappa)\left(\begin{matrix}C_A \\ C_A\end{matrix}\right)=\nu(C_{A}).
\end{equation}
(Indeed, take, e.g., $A=11$. In a simplified form,~\eqref{bazow1} becomes
\[
\nu'(11)=(\nu\vee\kappa)\left(\begin{matrix}11\\ 11\end{matrix}\right)=(\nu\vee\kappa)\left(\begin{matrix}11 \\ 11 \end{matrix}\cup \begin{matrix}01 \\ 11\end{matrix}\cup \begin{matrix}10 \\ 11 \end{matrix} \right)=\nu(11),
\]
where the middle equality follows from the fact that using~\eqref{gwi}, we have $(\nu\vee\kappa)\left(\begin{matrix}ij\\ 11\end{matrix}\right)=0$ whenever $i$ or $j$ equals $0$.)
However,~\eqref{bazow1} means that $\nu'=\nu$ and we obtain that the measure of maximal density is unique. The ergodicity of $\nu$ follows immediately by the ergodic decomposition, as we already know that the base measure is the unique measure of maximal density.
\end{proof}

\subsection{Examples}\label{przyk}
\begin{enumerate}[(A)]
\item The full shift $\{0,1\}^\Z$ is a measure-theoretic subordinate subshift, with the base measure $\delta_{\mathbf{1}}$ (recall that $\nu=M_\ast(\delta_{\mathbf{1}}\otimes \nu)$).
\item If $X\subseteq \{0,1\}^\Z$ is an arbitrary uniquely ergodic subshift, with $\mathcal{M}=\{\nu\}$ then $\widetilde{X}=M(X\times \{0,1\}^\Z)$ is a measure-theoretic subordinate subshift with base measure $\nu$ (see property (1) in~\cite{MR3589826} and the proof of Corollary 1 therein). In particular, we can take as $X$ any of the following (zero entropy) systems:
\begin{enumerate}[(i)]
\item
the orbit of a periodic $0$-$1$ sequence,
\item
a Sturmian dynamical system (recall that they arise by taking an irrational rotation $Tx=x+\alpha$ on $\mathbb{T}$, fixing an interval $J=[a,b)\subset \mathbb{T}$ such that $|J|$ and $\alpha$ are independent over $\mathbb{Q}$; then $X=\overline{\{  (\mathbbm{1}_J(T^n\omega)) : \omega\in\mathbb{T}\}}$),
\item
a uniquely ergodic Toeplitz dynamical system (notice that we have an overlap with hereditary closures of $\mathscr{B}$-free systems as minimal $\mathscr{B}$-free systems are Toeplitz~\cite[Corollary 1.5]{MR3803141}).
\end{enumerate}
In fact, this list is far from complete -- recall that any ergodic measure-theoretic dynamical system with entropy lower than $\log 2$ has a certain $0$-$1$ subshift as its uniquely ergodic model~\cite{MR0252604, MR0393402}.
\item
The hereditary closure $\widetilde{X}_\eta$ of any $\mathscr{B}$-free system is a measure-theoretic subordinate subshift with the Mirsky measure being its base measure, see Remark~\ref{commentsB} in Section~\ref{NNN} for the details and references.
\end{enumerate}
\begin{remark}
Notice that there are several overlaps in the above list:
\begin{itemize}
\item for a finite set $\mathscr{B}\subseteq \N$, $\eta$ is a periodic sequence -- we have an overlap of (B)(i) and (C);
\item since minimal $\mathscr{B}$-free systems are Toeplitz~\cite[Corollary 1.5]{MR3803141} and it is possible to have a minimal uniquely ergodic $\mathscr{B}$-free system (also other than the singleton $\{\mathbf{0}\}$, see e.g. Theorem~2 in~\cite{IrK}), there is an overlap of (C) and (B)(iii).
\end{itemize}
\end{remark}

\subsection{Measure-theoretic subordinate subshifts vs.\ hereditary and subordinate subshifts}\label{cohe}
In this subsection we will describe relations between measure-theoretic subordinate subshifts, hereditary subshifts and subordinate subshifts.
\begin{proposition}\label{sub}
Let $X\subseteq\{0,1\}^\Z$ be a measure-theoretic subordinate subshift with the base measure $\nu$. Then 
\[
X':=\widetilde{\text{supp}(\nu)}=M(\text{supp}(\nu)\times \{0,1\}^\Z)
\] 
satisfies $\mathcal{M}(X)=\mathcal{M}(X')$.
\end{proposition}
\begin{proof}
Consider $\mu:=M_\ast(\nu\otimes B_{1/2})\in\mathcal{M}(X)$. Then, clearly, $\text{supp}(\mu)\subseteq X$. If $A$ is a 0-1-block such that $\nu(A)>0$ and $A'\leq A$ (coordinatewise) for some block $A'$ of the same length as $A$, then 
\[
\mu(A')=(\nu\otimes B_{1/2})(M^{-1}A')\geq (\nu\otimes B_{1/2})(A\times C)=\nu(A)\cdot B_{1/2}(C)>0,
\] 
where $C$ is any 0-1-block (of the same length) such that $M(A\times C)=A'$. It follows immediately that $X'=\widetilde{\text{supp}(\nu)}\subseteq \text{supp}(\mu)\subseteq X$.

Now, take $\mu\in\mathcal{M}(X)$. Then $\mu=M_\ast(\nu\vee\kappa )$ for some $\kappa$ and thus
\begin{multline*}
\mu(X')=(\nu\vee\kappa)(M^{-1}(M(\text{supp}(\nu))\times \{0,1\}^\Z))\\
\geq (\nu\vee\kappa)(\text{supp}(\nu)\times \{0,1\}^\Z)=\nu(\text{supp}(\nu))=1.
\end{multline*}
It follows that $\mathcal{M}(X')=\mathcal{M}(X)$.
\end{proof}

\begin{remark}\label{portut2}
The hereditary subshift $X'$ from Proposition~\ref{sub} is, in fact, subordinate. Indeed, if $x\in\text{supp }\nu$ is generic for $\nu$ then $X_x=\text{supp }\nu$. It follows that $\overline{[\mathbf{0},x]}=\widetilde{\text{supp }\nu}$.
\end{remark}
\begin{remark}
Notice that $X:=\widetilde{X}_{(110)^\infty}\cup \widetilde{X}_{(10)^\infty}$ is clearly hereditary and $\frac{1}{3}(\delta_{(110)^\infty}+\delta_{(101)^\infty}+\delta_{(011)^\infty})$ is its unique measure of maximal density. However, $X$ is not a measure-theoretic subordinate subshift. Indeed, if this was the case then we would have
\[
\frac{1}{2}(\delta_{(10)^\infty}+\delta_{(01)^\infty})=M_\ast(\frac{1}{3}(\delta_{(110)^\infty}+\delta_{(101)^\infty}+\delta_{(011)^\infty})\vee\kappa).
\]
However,
\[
\frac{1}{2}(\delta_{(10)^\infty}+\delta_{(01)^\infty})(101010)=\frac{1}{2},
\]
while for every $\kappa$ and every joining $\lambda=\frac{1}{3}(\delta_{(110)^\infty}+\delta_{(101)^\infty}+\delta_{(011)^\infty})\vee\kappa$, we have $M_\ast(\lambda)(101010)=0$.

Note that $X$ is contained in a subordinate system with the same set of invariant measures. Indeed, it suffices to consider $\widetilde{X}_x$, where 
\[
x = \ldots B_{-2} B_{-1} B_0 B_1 B_2 \ldots
\]
and $B_i = (01)^{n_i} 0^{n_i} (001)^{n_i} 0^{n_i}$,
with $n_i\to \infty$ for $|i|\to\infty$.
\end{remark}

\begin{remark}\label{kiedyzero}
Let $X$ be a measure-theoretic subordinate subshift with base measure $\nu$. Let $X':=\widetilde{\text{supp }\nu}$ (cf.\ the proof of Proposition~\ref{sub}). Then, by the fact that $\mathcal{M}(X)=\mathcal{M}(X')$ and by the variational principle, we have $h_{top}(X)=h_{top}(X')$. Take any point $x\in X'$ that is generic for $\nu$. Since $X'$ is hereditary, it follows immediately that $h_{top}(X)=h_{top}(X')\geq \nu(1)$. Therefore, $\nu\neq \delta_{\boldsymbol{0}} \iff \nu(1)>0 \Rightarrow h_{top}(X)>0$. On the other hand, $\nu=\delta_{\boldsymbol{0}}$ implies $\mathcal{M}(X)=\{\delta_{\boldsymbol{0}}\}$ and thus $h_{top}(X)=0$. Therefore,
\begin{equation}\label{degen}
h_{top}(X)>0 \iff \nu(1)>0 \iff \nu\neq \delta_{\boldsymbol{0}}.
\end{equation}
\end{remark}

\subsection{Generic points viewpoint}\label{gepo}
We provide now a characterization of measure-theoretic subordinate subshifts in terms of so-called visible measures. 
\begin{definition}
Given $y\in\{0,1\}^\Z$, let $V(y)$ be the set of invariant measures for which $y$ is quasi-generic. We call the members of $V(y)$ the \emph{measures visible from $y$}. More generally, given a subsequence $(N_i)$, we define as $V_{(N_i)}(y)$ the subset of $V(y)$ of invariant measures for which $y$ is quasi-generic along a subsequence of $(N_i)$.
\end{definition}
\begin{proposition}\label{charakt}
Let $X\subseteq\{0,1\}^\Z$ be a subshift. Then the following are equivalent:
\begin{enumerate}[(a)]
\item $\mathcal{M}(X)=\{M_\ast(\nu\vee\kappa) : \kappa\in\mathcal{M}(\{0,1\}^\Z)\}$ (i.e.\ $X$ is a measure-theoretic subordinate subshifts with the base measure $\nu$),
\item $\mathcal{M}^e(X)\subseteq \{M_\ast(\nu\vee\kappa): \kappa\in\mathcal{M}(\{0,1\}^\Z)\} \subseteq \mathcal{M}(X)$,
\item $\mathcal{M}^e(X)\subseteq \bigcup_{y\leq x}V(y) \subseteq \mathcal{M}(X)$ for some point $x$ generic for $\nu$,
\item $\mathcal{M}^e(X)\subseteq \bigcup_{y\leq x}V(y) \subseteq \mathcal{M}(X)$ whenever $x$ generic for $\nu$,
\item $\mathcal{M}^e(X)\subseteq \bigcup_{y\leq x}V_{(N_i)}(y) \subseteq \mathcal{M}(X)$ for some $x$ quasi-generic for $\nu$ along~$(N_i)$,
\item $\mathcal{M}^e(X)\subseteq \bigcup_{y\leq x}V_{(N_i)}(y) \subseteq \mathcal{M}(X)$ whenever $x$ is quasi-generic for $\nu$ along~$(N_i)$.
\end{enumerate}
\end{proposition}
\begin{remark}
Notice that conditions (c) and (d) are, in fact, contained in (e) and (f); one just takes $N_i=i$.
\end{remark}

To prove Proposition~\ref{charakt}, we will need a certain ``lifting lemma'' from~\cite{MR4525753} for quasi-generic points to joinings. We formulate it here for $\Z$-actions, while the original version is more general (the result is true for actions of countable cancellative semigroups and arbitrary Følner sequences, see~\cite{MR4525753}).
\begin{theorem}[Theorem 5.16 in~\cite{MR4525753}]\label{downar}
Let $\mathcal{A}_1,\mathcal{A}_2$ be finite alphabets. If $x\in\mathcal{A}_1^\Z$ is quasi-generic for $\nu$ and $\nu\vee\kappa\in \mathcal{M}(\mathcal{A}_1^\Z\times \mathcal{A}_2^\Z)$ then there exists $y\in\mathcal{A}_2^\Z$ such that the pair $(x,y)$ is quasi-generic for $\nu\vee\kappa$:
\[
\nu \in V(x) \implies \{\nu\vee\kappa: \kappa\in\mathcal{M}(\mathcal{A}_2^\Z)\}\subseteq \bigcup_{y}V((x,y)).
\]
More precisely, if $x$ is quasi-generic for $\nu$ along $(N_i)$ then there exists $y$ such that $(x,y)$ is quasi-generic for $\nu\vee\kappa$ along some subsequence of $(N_i)$:\footnote{Whether the same is true without passing to a subsequence is unknown, see Question 5.17 in~\cite{MR4525753}.}
\begin{equation}\label{tut}
\frac{1}{N_i}\sum_{n\leq N_i}\delta_{\sigma^nx}\to \nu \implies \{\nu\vee\kappa: \kappa\in\mathcal{M}(\mathcal{A}_2^\Z)\}\subseteq \bigcup_{y}V_{(N_i)}((x,y)).
\end{equation}
In fact, we have
\(
\{\nu\vee\kappa: \kappa\in\mathcal{M}(\mathcal{A}_2^\Z)\}= \bigcup_{y}V_{(N_i)}((x,y)).
\)
\end{theorem}
\begin{corollary}\label{downar1}
For $x\in\{0,1\}^\Z$, we have
\[
\frac{1}{N_i}\sum_{n\leq N_i}\delta_{\sigma^nx}\to \nu\implies \{M_\ast(\nu\vee\kappa): \kappa\in\mathcal{M}(\{0,1\}^\Z)\}=\bigcup_y M_\ast(V_{(N_i)}((x,y)))=\bigcup_{y\leq x}V_{(N_i)}(y).
\]
\end{corollary}
\begin{proof}
The first equality from the assertion follows directly by Theorem~\ref{downar}. To obtain the second equality, notice that
\begin{itemize}
\item
$M_\ast(V_{(N_i)}((x,y)))\subseteq V_{(N_i)}(M(x,y))\subseteq \bigcup_{y\leq x}V_{(N_i)}(y)$
\item 
for $\mu\in V_{(N_i)}(y)$ with $y\leq x$ we have $\nu\vee\mu \in V_{(N_i)}((x,y))$ (for some joining $\nu\vee\mu$) and since $y=M(x,y)$, we conclude that $\mu=M_\ast(\nu\vee\mu)\in M_\ast(V_{(N_i)}((x,y)))$.
\end{itemize}
\end{proof}
\begin{proof}[Proof of Proposition~\ref{charakt}]
The implication from (a) to (b) is immediate.
The implication from (b) to (a) follows easily from the fact that $M_\ast$ commutes with taking generalized convex combinations of measures. That is, $\{M_\ast(\nu\vee\kappa): \kappa\in\mathcal{M}(\{0,1\}^\Z)\}$ is closed under taking generalized linear combinations and, hence, it contains the convex hull of $\mathcal{M}^e(X)$, which in turn is is equal to $\mathcal{M}(X)$, so (a) follows.

The equivalences $(b) \iff (e) \iff (f)$ follow immediately by Theorem~\ref{downar} and Corollary~\ref{downar1}.
\end{proof}

\subsection{Approximation}\label{ap1}
We will now study the following problem:
\begin{quote}
When is the limit $\bigcap_{K\geq 1}X_K$ of a descending family $\{X_K : K\geq 1\}$ of measure-theoretic subordinate subshifts also a measure-theoretic subordinate subshift?
\end{quote}
Assume that each $X_K$, $K\geq 1$, is a measure-theoretic subordinate subshift with base measure $\nu_K$ and that we have $X_K\supseteq X_{K+1}$ for each $K\geq 1$. Notice that it follows from Theorem~\ref{downar} that one can find $x_K\in \{0,1\}^\Z$ such that each $x_K$ is generic for $\nu_K$ and $x_{K+1}\leq x_K$ for $K\geq 1$. Indeed, we have $\nu_{K+1}=M_\ast(\nu_K\vee\kappa_K)$ and we just pick $y_K$ such that $(x_K,y_K)$ is generic for the joining $\nu_K\vee\kappa_K$. The point $x_{K+1}:=M_\ast(x_K,y_K)\leq x_K$ is then generic for $\nu_{K+1}$. Let now $x$ be the coordinatewise limit of $(x_K)$. Take $y\leq x$ and $\mu\in V(y)$. Since $y\leq x_K$ for each $K\geq 1$, we have $\mu\in \mathcal{M}(\bigcap_{K\geq 1}X_K)$. Hence
\[
\bigcup_{y\leq x}V(y)\subseteq \mathcal{M}(\bigcap_{K\geq 1}X_K),
\]
cf.\ Proposition~\ref{charakt} (e). Let $\nu$ be an accumulation point of $(\nu_K)$. Now, take $\mu\in\mathcal{M}^e(\bigcap_{K\geq 1}X_K)$. Then $\mu\in \mathcal{M}^e(X_K)$ for each $K\geq 1$ and thus we have $\mu=M_\ast(\nu_K\vee\kappa_K)$. Clearly, passing to a subsequence, we can assume that $\nu_K\vee\kappa_K\to \nu\vee\kappa$ for a joining of $\nu$ with some measure $\kappa$ (dependent on $\mu$). Therefore, 
\[
\mathcal{M}^e(\bigcap_{K\geq 1}X_K)\subseteq \{M_\ast(\nu\vee\kappa) : \kappa\in\mathcal{M}(\{0,1\}^\Z)\}.
\]
If $x$ would be (quasi-)generic for $\nu$ then, by Corollary~\ref{downar1} and by Proposition~\ref{charakt}, the proof of $\bigcap_{K\geq 1}X_K$ being a measure-theoretic subordinate subshift would be complete. However, a priori, there is no relation between $x$ and $\nu$.
\begin{definition}
Let  $(x_K)\subseteq \{0,1\}^\Z$ be a non-increasing sequence and let $x=x_\infty$ be its coordinatewise limit. We say that $(x_K)$ is \emph{good} if there exists an increasing sequence $(N_i)\subseteq\mathbb{N}$ such that:
\begin{itemize}
\item $x_K$ is quasi-generic along $(N_i)$ for some measure $\nu_K$, $K\in \N\cup\{\infty\}$,
\item $\nu_K(1)\to \nu:=\nu_{\infty}(1)$. 
\end{itemize}
\end{definition}

There is the following well-known related result.
\begin{lemma}\label{lemat_gene}
Suppose that $(x_K)\subseteq \{0,1\}^\Z$ is good. Then $\nu_K\to \nu$.
\end{lemma}
\begin{remark}
{The convergence $\nu_K\to \nu$ in the assertion of the above lemma follows by so-called d-bar convergence of sequences. We decided to include the straightforward computation to avoid introducing too much formalism that is not used in other parts of this paper. For some infomation on d-bar convergence, see e.g.~\cite{MR3669782} or~\cite{doktorat}.}
\end{remark}
\begin{remark}
To see why the assumption that $\nu_K(1)\to \nu:=\nu_{\infty}(1)$ is necessary, it suffices to take $x_K=\underbrace{0\dots 0}_{K}11\dots$ and $x=\mb{0}$. Then $\nu_K=\delta_{\mb{1}}$, while $\nu=\delta_{\mb{0}}$.
\end{remark}
\begin{proof}[Proof of Lemma~\ref{lemat_gene}]
To obtain $\nu_K\to \nu$, it suffices to show that for each $B\in \{0,1\}^{\ell}$, $\ell\geq 1$, we have $\nu_k(B) \to \nu_\infty(B)$.
           Indeed,
           \begin{align*}
              \lim_{i\to\infty}&\left|\frac{1}{N_{i}}\sum_{n=0}^{N_{i}-1}\delta_{S^{n}x_K}(B)-\frac{1}{N_{i}}\sum_{n=0}^{N_{i}-1}\delta_{S^{n}x_\infty}(B)\right|\leq \lim_{i\to\infty}\frac{1}{N_{i}}\sum_{n=0}^{N_{i}-1}\left|\mathbf{1}_{x_K[n,n+\ell-1]=B}-\mathbf{1}_{x_\infty[n,n+\ell-1]=B} \right|\\
              &\leq \lim_{i\to\infty}\frac{1}{N_{i}}\sum_{n=0}^{N_{i}-1}\mathbf{1}_{x_K[n,n+\ell-1]\neq x_\infty[n,n+\ell-1]}\leq \ell (\nu_K(1) - \nu_\infty(1))\to 0 \text{ as }K\to\infty.
           \end{align*}
\end{proof}

\begin{proposition}\label{prop_gene}
Suppose that $(x_K)\subseteq \{0,1\}^\Z$ is good and let $x$ be its coordinatewise limit. 
If $\overline{[\mathbf{0},x_K]}$ is a measure-theoretic subordinate subshift with base measure $\nu_K$ (corresponding to $(N_i)$ from the definition of good sequences) then both $\bigcap_{K\geq 1}\overline{[\mathbf{0},x_K]}$ and $\overline{[\mathbf{0},x]}$ are measure-theoretic subordinate subshifts with base measure $\nu_\infty=\lim_{K\to\infty}\nu_K$.
\end{proposition}
\begin{proof}
The assertion for $\bigcap_{K\geq 1}\overline{[\mb{0},x_K]}$ follows by the discussion preceding Lemma~\ref{lemat_gene} and the lemma itself. Moreover, since $\nu_\infty \in\mathcal{M}(X_x)$ as $x$ is a quasi-generic point for $\nu_\infty$, it follows that $M_\ast (\nu_\infty\vee \kappa)\in\mathcal{M}(\overline{[\mathbf{0},x]})$ for any $\kappa\in\mathcal{M}(\{0,1\}^\Z)$ and, thus,
\[
\{M_\ast (\nu_\infty\vee\kappa) : \kappa\in\mathcal{M}(\{0,1\}^\Z)\}\subseteq\mathcal{M}(\overline{[\mathbf{0},x]})\subseteq\mathcal{M}(\bigcap_{K\geq 1}[\mb{0},x_K]) = \{M_\ast (\nu_\infty\vee\kappa) : \kappa\in\mathcal{M}(\{0,1\}^\Z)\}.
\]
\end{proof}

\begin{remark}
The assumptions of Proposition~\ref{prop_gene} are satisfied, e.g., when $x=\eta$ for some $\mathscr{B}\subseteq\N$ and $x_K=\mathbf{1}_{\mathcal{F}_{\mathscr{B}_K}}$. For the details, see Remark~\ref{commentsB} in Section~\ref{NNN}.
\end{remark}

\section{Sandwich measure-theoretic subordinate subshifts}\label{BCS}
Let again $ X \subseteq \{0,1\}^\Z$ be a subshift and let $N\colon(\{0,1\}^\Z)^3\to\{0,1\}^\Z$ be given by
\[
N(w,x,y)=(1-y)\cdot w+y\cdot x.
\]
Again, it will be convenient to translate the definition of a sandwich measure-theoretic subordinate subshift to the language of joinings: $ X $ is called a \emph{sandwich measure-theoretic subordinate subshift} if there exists $\rho\in \mathcal{M}(\{0,1\}^\Z\times \{0,1\}^\Z)$ such that
\begin{equation}\label{miary}
\mathcal{M}( X )=\{N_\ast(\rho\vee \kappa) : \kappa\in\mathcal{M}(\{0,1\}^\Z)\}.
\end{equation}
We will denote by $\nu_0$ and $\nu_1$ the projection of $\rho$ onto the first and second coordinate, respectively (hence, $\rho$ is a certain joining of $\nu_0$ and $\nu_1$). 

\subsection{Base measure}

Notice first that if $ X $ is a sandwich measure-theoretic subordinate subshift with a pre-base measure $\rho=\nu_0\vee\nu_1$ then $\nu_0,\nu_1\in\mathcal{M}( X )$.
Indeed, for any finite $A\subseteq \Z$, we have
\begin{align*}
N_\ast&(\rho\otimes \delta_{\mathbf{1}})(\{x : x_i=1 \text{ for }i\in A\})\\
&=\rho\otimes \delta_{\mathbf{1}}(\{(w,x,y) : (w+y(x-w))_i=1 \text{ for }i\in A:\})\\
&=\rho(\{(w,x) : x_i=1 \text{ for }i\in A\})=\nu_1(\{x : x_i=1\text{ for }i\in A\}).
\end{align*}
Thus, by the inclusion-exclusion principle, $N_\ast(\rho\otimes \delta_{\mathbf{1}})=\nu_1$. By the same token, $N_\ast(\rho\otimes \delta_{\mathbf{0}})=\nu_0$.

\begin{remark}\label{re:11}
Given a sandwich measure-theoretic subordinate subshift $X$, there might be many pre-base measures. The simplest example here is $\{0,1\}^\Z$, where e.g.\ both $\delta_{\mathbf{0}}\otimes \delta_{\mathbf{1}}$ and $(\sigma\otimes id)_\ast (\delta_{(01)^\infty})$ satisfy~\eqref{miary}. However, if $ X $ is a sandwich measure-theoretic subordinate subshift then one can find $\rho$ satisfying~\eqref{miary} such that $\rho$ lives in the ``above the diagonal'' part of the space, that is $\rho$ satisfies $\rho(\mathbf{A})=1$, where $\mathbf{A}:=\{(w,x) : w\leq x\}$. Indeed let 
$J\colon X\times X\times \{0,1\}^Z\mapsto (\{0,1\}^\Z)^3$ be defined by
\[
J\colon (w,x,y)\mapsto (\min(w,x),\max(w,x),\widetilde{y}),
\]
where $\widetilde{y}(n)=y(n)$ when $w(n)<x(n)$ and $\widetilde{y}(n)=1-y(n)$ otherwise. Then
\(
N\circ J=N
\)
and it follows immediately that $J_\ast(\rho)$ satisfies~\eqref{miary} and is supported on $\mathbf{A}$. In other words, for any sandwich measure-theoretic subordinate subshift a base measure always exists.
\end{remark}

\subsection{Examples}\label{ppp}
Each measure-theoretic subordinate subshift turns out to be a sandwich measure-theoretic subordinate subshift. More precisely, we have the following.
\begin{proposition}[cf.\ Proposition~\ref{zwiazek}]
Let $X\subseteq \{0,1\}^\Z$ be a measure-theoretic subordinate subshift whose base measure is $\nu$. Then $X$ is a sandwich measure-theoretic subordinate subshift for which $\delta_{\boldsymbol{0}}\otimes \nu$ is a base measure.
\end{proposition}
\begin{proof}
Take $\mu\in\mathcal{M}(X)$. Then $\mu=M_\ast(\nu\vee\kappa)$ for some joining $\nu\vee\kappa$. Take $\rho:=\delta_{\boldsymbol{0}}\otimes \nu$ and notice that 
\[
\delta_{\boldsymbol{0}}\otimes (\nu\vee\kappa)=(\delta_{\boldsymbol{0}}\otimes \nu)\vee\kappa=\rho\vee\kappa.
\] 
Moreover, for a.e.\ $(w,x,y)$ with respect to any measure of the above form, we have
\[
M\circ \pi_{2,3}(w,x,y)=M(x,y)=N(\mathbf{0},x,y)=N(w,x,y)
\]
and thus
\[
N_\ast(\rho\vee \kappa)=N_\ast(\delta_{\boldsymbol{0}}\otimes (\nu\vee\kappa))=M_\ast \circ (\pi_{2,3})_\ast (\nu\vee\kappa)=M_\ast (\nu\vee\kappa),
\]
which completes the proof.
\end{proof}
However, we have more than that. Let us make now a list of examples of sandwich measure-theoretic subordinate subshifts parallel to that in Section~\ref{przyk}.
\begin{enumerate}[(A)]
\item The full shift, with the base measure $\delta_\mathbf{0}\otimes \delta_\mathbf{1}$ (this is immediate by the preceding proposition).
\item At the end of Section~\ref{bicobisu}, we will show the following.
\begin{proposition}[Cf.\ Example (B) in Section~\ref{przyk}]\label{biokre}
Let $Z\subseteq \{0,1\}^\Z\times \{0,1\}^\Z$ be closed and invariant under the action of $\sigma\times \sigma$. If $Z$ is uniquely ergodic with $\mathcal{M}(Z,\sigma\times\sigma)=\{\rho\}$ then $N(Z\times \{0,1\}^\Z)$ is a sandwich measure-theoretic subordinate subshift with pre-base measure $\rho$. In particular, if $w\leq x$ are both periodic then $[w,x]$ is a sandwich measure-theoretic subordinate subshift.
\end{proposition}
\item For any $\mathscr{B}\subseteq \mathbb{N}$ such that $\eta^*$ is a regular Toeplitz sequence, $X_\eta$ is a sandwich measure-theoretic subordinate subshift.  See Remark~\ref{commentsB1} in Section~\ref{NNNNN} for the details.
\end{enumerate}

\subsection{Generic points viewpoint}\label{ggg}
Clearly, for any sandwich measure-theoretic subordinate subshift, one can always find an ergodic base measure. Moreover, the marginals of any base measure are always ergodic. This is a consequence of the following more involved result.
\begin{proposition}\label{zwiazek}
Suppose that $ X \subseteq \{0,1\}^\Z$ is a sandwich measure-theoretic subordinate subshift with base measure~$\rho=\nu_0\vee\nu_1$. Then $\widetilde{ X }$ is a measure-theoretic subordinate subshift with base measure $\nu_1$.
\end{proposition}
For the proof of Proposition~\ref{zwiazek}, we will need an analogue of Proposition~\ref{charakt}.
\begin{proposition} Let $ X \subseteq \{0,1\}^\Z$ be a subshift and suppose that $\rho$ is supported on $\mathbf{A}=\{(w,x) : w\leq x\}$. Then the following are equivalent:
\begin{enumerate}[(a)]
\item $\mathcal{M}( X )=\{N_\ast(\rho\vee\kappa) : \kappa\in\mathcal{M}(\{0,1\}^\Z)\}$ (i.e.\ $ X $ is a sandwich measure-theoretic subordinate subshift with base measure $\rho$),
\item $\mathcal{M}^e( X )\subseteq \{N_\ast(\rho\vee\kappa) : \kappa\in\mathcal{M}(\{0,1\}^\Z\}\subseteq \mathcal{M}( X )$,
\item $\mathcal{M}^e( X )\subseteq \bigcup_{w\leq y\leq x}V_{(N_i)}(y)\subseteq \mathcal{M}( X )$ for some pair $(w,x)$ quasi-generic for $\rho$ along $(N_i)$, such that $w\leq x$,
\item $\mathcal{M}^e( X )\subseteq \bigcup_{w\leq y\leq x}V_{(N_i)}(y)\subseteq \mathcal{M}( X )$ whenever the pair $(w,x)$   is quasi-generic for $\nu$ along $(N_i)$ and satisfies $w\leq x$.
\end{enumerate}
\end{proposition}
\begin{proof}
The arguments are the same as in the proof of Proposition~\ref{charakt}, one just needs to change $\nu$ for $\rho$, $y\leq x$ for $w\leq y\leq x$ and $M$ for $N$.
\end{proof}
\begin{proof}[Proof of Proposition~\ref{zwiazek}]
We will show that condition (d) from Proposition~\ref{charakt} holds. To obtain the first necessary inclusion, let $\mu\in\mathcal{M}^e(\widetilde{ X })$. Let $u\in \widetilde{X}$ be a generic point for $\mu$. Then there exists $v\in X$ such that $u\leq v$. Moreover, we clearly have $u=M(v,u)$. Let $\kappa\vee\mu$ be any member of $V((v,u))$. Since $X$ is a sandwich measure-theoretic subordinate subshift with base measure $\rho$, we have $\kappa=N_\ast(\rho\vee\kappa')$ for some ergodic joining of $\rho$ with some measure $\kappa'\in\mathcal{M}(\{0,1\}^\Z)$. Let $(w,x,y)\in\mathbf{A}\times \{0,1\}^\Z$ be generic for $\rho\vee\kappa'$. In particular, $x$ is generic for $\nu_1$. Then $N(w,x,y)$ is generic for $\kappa=N_\ast(\rho\vee\kappa')$. It follows by Theorem~\ref{downar} that there exists $z\in\{0,1\}^\Z$ such that $(N(w,x,y),z)$ is generic for $\kappa\vee\mu$. Finally, take any member of $V((w,x,y,z))$. Its image via $M\circ (N\times I)$ is quasi-generic for $\mu$ and $M\circ (N\times I)(w,x,y,z)$ is its generic point. It remains to notice that
\(
M\circ (N\times I)(w,x,y,z)\leq N(w,x,y)\leq x
\)
to deduce that $\mathcal{M}^e(\widetilde{ X })\subseteq \bigcup_{y\leq x}V(y)$.

To finish the proof we only need the following:
\[
\bigcup_{y\leq x}V(y) = \{M_\ast(\nu_1\vee\kappa) : \kappa \in\mathcal{M}(\{0,1\}^\Z)\}\subseteq \mathcal{M}(\widetilde{ X })
\]
The equality is given by Theorem \ref{downar} and the inclusion follows directly by the fact that $M( X \times \{0,1\}^\Z)=\widetilde{ X }$ and $\nu_1\in\mathcal{M}(X)$.
\end{proof}
\begin{corollary}\label{cor:1}
For any sandwich measure-theoretic subordinate subshift $ X $, the corresponding projections $\nu_0$ and $\nu_1$ of a base measure $\rho$ are ergodic and do not depend on the choice of $\rho$. Moreover, 
\[
\nu_1(1)=\sup_{\nu\in\mathcal{M}( X )}\nu(1) \text{ and }\nu_0(0)=\sup_{\nu\in\mathcal{M}( X )}\nu(0)
\]
and $\nu_0$ and $\nu_1$ are the unique members of $\mathcal{M}( X )$ realizing the above suprema.
\end{corollary}
\begin{proof}
This is an immediate consequence of Proposition~\ref{zwiazek} (for the assertions for $\nu_0$ one needs to exchange the role of symbols $0$ and $1$ and appropriately modify the maps that come into play).
\end{proof}

\subsection{Sandwich measure-theoretic subordinate subshifts vs sandwich hereditary and sandwich subordinate subshifts}\label{bicobisu}
\begin{proposition}[Cf.\ Proposition~\ref{sub} and Remark~\ref{portut2}]\label{bihe}
If $ X \subseteq\{0,1\}^\Z$ is a sandwich measure-theoretic subordinate subshift then there exists a bi-subordinate subshift $ X '\subseteq  X $ such that $\mathcal{M}( X )=\mathcal{M}( X ')$.
\end{proposition}
\begin{proof}
Notice that $\overline{[w,x]}=\overline{\{\sigma^ny : w\leq y\leq x\}}\subseteq  X $ is equivalent to the following:
\begin{equation}\label{miniteza}
\text{if }w[-n,n]\leq C\leq x[-n,n] \text{ then }C\in\mathcal{L}( X ) \text{ for }n\in\N,
\end{equation}
where $\mathcal{L}(X)=\bigcup_{n\geq 1}\mathcal{L}_n(X)$. Let $X$ be a sandwich measure-theoretic subordinate subshift. We will find a pair $(w,x)$ such that~\eqref{miniteza} holds for $X$. Let $\rho$ be a base measure for $ X $ and let $(w,x)\in \text{supp }\rho$ be generic for $\rho$.  There exists $y$ such that $(w,x,y)$ is generic for $\rho\otimes B_{1/2}$. Then $N(w,x,y)$ is generic for $\mu:=N_\ast(\rho\otimes  B_{1/2})$. Suppose that $A,B\in \mathcal{L}_{2n+1}( X )$ are such that $A\leq B$ and $\rho(A\times B)>0$. Then, for any $A\leq C\leq B$, we have $\mu(C)>0$. In particular, we have $C\in\mathcal{L}( X )$. We conclude that~\eqref{miniteza} holds since $\rho(w[-n,n]\times x[-n,n])>0$ (as $(w,x)
\in \text{supp }\rho$). We set $X':=N(\text{supp}(\rho)\times \{0,1\}^\Z)$.

Now, take $\mu\in\mathcal{M}(X)$. Then $\mu=N_\ast(\rho\vee\kappa )$ for some $\kappa$ and thus
\begin{multline*}
\mu(X')=(\rho\vee\kappa)(N^{-1}(N(\text{supp}(\rho))\times \{0,1\}^\Z))\\
\geq (\rho\vee\kappa)(\text{supp}(\rho)\times \{0,1\}^\Z)=\rho(\text{supp}(\rho))=1.
\end{multline*}
It follows that $\mathcal{M}(X')=\mathcal{M}(X)$.

\end{proof}

\begin{proposition}\label{uwazka}
Let $ X $ be a sandwich measure-theoretic subordinate subshift with base measure $\rho$ with marginals $\nu_0,\nu_1$. Then:
\begin{itemize}
\item $h_{top}( X )\geq \nu_1(1)-\nu_0(1)$,
\item if $\nu_0(1)=\nu_1(1)$ then $\nu_0=\nu_1$, $\mathcal{M}(X)=\{\nu_0\}=\{\nu_1\}$ and, thus, $h_{top}( X )=h(\nu_0)=h(\nu_1)$.
\end{itemize}
\end{proposition}
\begin{proof}
In view of Proposition~\ref{bihe}, we can assume without loss of generality that $ X $ is bi-subordinate. Let $(w,x)\in \text{supp}(\rho)$ with $w\leq x$ be a generic pair for $\rho$. Then the frequency of the pair $(0,1)$ appearing in $(w,x)$ equals
\[
\rho\left(\begin{matrix}0 \\ 1\end{matrix}\right)=(\nu_0\vee \nu_1)\left(\begin{matrix}0 \\ 1\end{matrix}\right)=\nu_1(1)-\nu_0(1).
\]
It follows that indeed
\(
h_{top}( X )\geq \nu_1(1)-\nu_0(1).
\)

Notice that if $\nu_0(1)=\nu_1(1)$  and $(w,x)$ with $w\leq x$ is generic for $\rho$, we get that $w=x$ along a full density subsequence. Thus, $\rho$ is the diagonal joining of two copies of $\nu_0=\nu_1$, i.e.\ $\rho(A\times B)=\nu_1(A\cap B)$. Moreover,
\begin{align*}
N_\ast(\rho\vee\kappa)(\{x : x_i=1 \text{ for }i\in A\})&=(\rho\vee\kappa) (\{(w,x,y) : (w(1-y)+xy)_i=1 \text{ for }i\in A\})\\
&=\nu_1(\{x : x_i=1 \text{ for }i \in A\})
\end{align*}
and we conclude that $\mathcal{M}(X)=\{N_\ast(\rho\vee\kappa)\}=\{\nu_1\}$.
\end{proof}

Last but not least, let us prove Proposition~\ref{biokre} (the proof goes along the same lines as the one of Corollary~1 in~\cite{MR3589826}. We provide the details for convenience).

\begin{proof}[Proof of Proposition~\ref{biokre}]
Let $\mu\in\mathcal{M}^e(N(Z\times\{0,1\}^\Z))$ and let $z\in N(Z\times \{0,1\}^\Z)$ be its generic point. We have $z=N(w,x,y)$, where $(w,x)\in Z$ and $y\in\{0,1\}^\Z$. By the unique ergodicity of $Z$, the pair $(w,x)$ is generic for $\rho$. Hence, $(w,x,y)$ is quasi-generic for a measure of the form $\rho\vee\kappa$ with $\kappa\in\mathcal{M}(\{0,1\}^\Z)$. It follows immediately that $z$ must be quasi-generic for $N_\ast(\rho\vee\kappa)$ and thus $\mu=N_\ast(\rho\vee\kappa)$. This yields
\[
\mathcal{M}^e(N(Z\times \{0,1\}^\Z))\subseteq \{N_\ast (\rho\vee\kappa) : \kappa \in\mathcal{M}(\{0,1\}^\Z)\}.
\]
The converse inclusion follows immediately by the form of the studied subshift $N(Z\times \{0,1\}^\Z)$.
\end{proof}

\subsection{Approximation}\label{ap2}
In this section we adress the following question.
\begin{quote}
When is the limit of a descending family of sandwich measure-theoretic subordinate subshifts also a sandwich measure-theoretic subordinate subshift?
\end{quote}
\begin{definition}
Let  $(w_K,x_K)\subseteq \{0,1\}^\Z\times \{0,1\}^\Z$ be such that $w_K\leq x_K$, $(w_K)$ is non-decreasing and $(x_K)$ is non-increasing (coordinatewise). Let $(w,x)=(w_\infty,x_\infty)$ be the coordinatewise limit of $(w_K,x_K)$. We say that $(w_K,x_K)$ is \emph{good} if there exists an increasing sequence $(N_i)\subseteq\mathbb{N}$ such that:
\begin{itemize}
\item $(w_K,x_K)$ is generic along $(N_i)$ for some measure $\rho_K$, $K\in\N\cup\{\infty\}$,
\item $\rho_K\left(\begin{matrix} 0\\ 1\end{matrix}\right)\to \rho\left(\begin{matrix} 0\\ 1\end{matrix}\right)$.
\end{itemize}
\end{definition}

\begin{proposition}[Cf.\ Proposition~\ref{prop_gene}]\label{III}
Suppose that $(w_K,x_K)\in \{0,1\}^\Z\times \{0,1\}^\Z$ is good and let $(w,x)$ be its coordinatewise limit. If $\overline{[w_K,x_K]}$ is a sandwich measure-theoretic subordinate subshift with base measure $\rho_K$ (corresponding to $(N_i)$ from the definition of good sequences) then both $\bigcap_{K\geq 1}\overline{[w_K,x_K]}$ and $\overline{[w,x]}$ are sandwich measure-theoretic subordinate subshifts with base measure $\rho=\lim_{K\to\infty}\rho_K$.
\end{proposition}
For the proof of Proposition~\ref{III}, we will need the following modification of Lemma~\ref{lemat_gene}.
\begin{lemma}\label{lemat_gene1}
Suppose that $(w_K,x_K)\in \{0,1\}^\Z\times\{0,1\}^\Z$ is good. Then $\rho_K\to \rho$.
\end{lemma}
\begin{proof}
The proof of analogous to that of Lemma~\ref{lemat_gene}.
\end{proof}
\begin{proof}[Proof of Proposition~\ref{III}]
We will prove first that
\begin{equation*}
\{N_\ast(\rho\vee\kappa) : \kappa\in\mathcal{M}(\{0,1\}^\Z)\}\subseteq \mathcal{M}(\overline{[w,x]}).
\end{equation*}
Take $\mu=N_\ast(\rho\vee\kappa)$. By Theorem~\ref{downar}, there exists $y\in\{0,1\}^\Z$ such that $\rho\vee\kappa\in V(w,x,y)$. Thus,
\[
\mu=N_\ast(\rho\vee\kappa)\in N_\ast(V(w,x,y))\subseteq\bigcup_{w\leq y\leq x}V(y)\subseteq \mathcal{M}(\overline{[w,x]}). 
\]

Now, we will show that
\begin{equation*}
\mathcal{M}(\bigcap_{K\geq 1}\overline{[w_K,x_K]})\subseteq \{N_\ast(\rho\vee\kappa) : \kappa\in\mathcal{M}(\{0,1\}^\Z)\}.
\end{equation*}
Take $\mu\in \mathcal{M}(\bigcap_{K\geq 1}\overline{[w_K,x_K]})$. Then $\mu=N_\ast(\rho_K\vee\kappa_K)$. In view of Lemma~\ref{lemat_gene1}, passing to a subsequence, we may assume that $\rho_K\vee\kappa_K\to\rho\vee\kappa$ for some joining of $\rho$ with a measure $\kappa\in\mathcal{M}(\{0,1\}^\Z)$. Thus, $\mu=N_\ast(\rho\vee\kappa)$, which completes the proof since $\overline{[w,x]}\subseteq \bigcap_{K\geq 1}\overline{[w_K,x_K]}$ (so $\mathcal{M}(\overline{[w,x]})\subseteq \mathcal{M}(\bigcap_{K\geq 1}\overline{[w_K,x_K]})$).
\end{proof}

\section{Topological pressure: $\varphi(x)=\varphi(x_0)$}\label{se:jedna}

\subsection{Measure-theoretic subordinate subshifts}\label{NNN}
This section is devoted to finding a formula for $\mathcal{P}_{X,\varphi}$, where $X$ is a measure-theoretic subordinate subshift with a zero entropy base measure and $\varphi(x)=\varphi(x_0)$. We then apply it to the hereditary closure of $\mathscr{B}$-free systems and compare with the earlier known related results. We will denote by $B_p$ the $(p,1-p)$-Bernoulli measure on $\{0,1\}^\Z$. The main results in this section are as follows.
\begin{theorem}\label{int1}\label{toto}
Suppose that $X\subseteq \{0,1\}^\Z$ is a measure-theoretic subordinate subshift with base measure $\nu$ of zero entropy. Then $h_{top}(X)=\nu(1)$. Moreover, $X$ is intrinsically ergodic and $M_\ast(\nu\otimes B_{1/2})$ is the unique measure of maximal entropy.
\end{theorem}
\begin{theorem}\label{int0}
Suppose that $X\subseteq \{0,1\}^\Z$ is a measure-theoretic subordinate subshift with base measure $\nu$ of zero entropy. Let $\varphi\colon X\to\R$ be of the form $\varphi(x)=\varphi(x_0)$. Then 
\[
\mathcal{P}_{X,\varphi}=\nu(0)\varphi(0)+\nu(1)\log(2^{\varphi(0)}+2^{\varphi(1)})
\]
and there is the following corresponding unique equilibrium state: $M_\ast(\nu\otimes B_p)$, where $p=2^{\varphi(0)}/(2^{\varphi(0)}+2^{\varphi(1)})$.
\end{theorem}
Although Theorem~\ref{int1} is a special case of Theorem~\ref{int0}, we state them separately and prove Theorem~\ref{int1} first. In fact, both proofs go along the same lines, however, it is easier to explain the main underlying ideas treating the simpler case first and then pass to a more complicated setup.

Denote by $\hat{Z}$ the set of $x\in \{0,1\}^\Z$ such that the support $\text{supp } x:=\{n\in\Z : x_n\neq 0\}$ of $x$ is unbounded both from below and from above. Given $x\in \hat{Z}$ and $y\in\{0,1\}^\Z$, let $\hat{y}_x\in \{0,1\}^\Z$ be the unique sequence consisting of the consecutive coordinates of $y$ that lie along the support of $x$ and such that $\hat{y}_x(0)=y(n)$ for $n=\min\{(\text{supp }x) \cap [0,\infty)\}$. Consider also
\[
\Psi_0\colon \hat{Z}\times \{0,1\}^\Z\to \hat{Z}\times \{0,1\}^\Z,\ \Psi_0(x,y)=(x,\hat{y}_x)
\]
and $\Phi_0\colon \hat{Z}\times \{0,1\}^\Z \to \{0,1\}^\Z$ given by the condition that
\[
\Phi_0(x,y) \text{ is the unique $0$-$1$ sequence such that }\Phi_0(x,y)\leq x \text{ and }\Phi_0(x,y)^{\widehat{}}_x=y.
\]
Last but not least, consider the following skew product: 
\[
\widetilde{\sigma}\colon (\{0,1\}^\Z)^2\to (\{0,1\}^\Z)^2,\ \widetilde{\sigma}(x,y)=(\sigma x,\sigma^{x_0}y).
\]
Take $C_0:=\{x \in \hat{Z} : x_0=1\}$ and consider the induced map $\widetilde{\sigma}_{C_0\times \{0,1\}^\Z}$ corresponding to $\widetilde{\sigma}$ and $C_0\times \{0,1\}^\Z$. Notice that
\begin{equation}\label{indukowane}
\widetilde{\sigma}_{C_0\times \{0,1\}^\Z}=\sigma_{C_0} \times \sigma.
\end{equation}

The main tool in the proofs of Theorem~\ref{toto} and Theorem~\ref{int0} will be the following result.
\begin{proposition}\label{diagramiki}
Suppose that $X\subseteq \{0,1\}^\Z$ is a measure-theoretic subordinate subshift with base measure $\nu$. Then the following two diagrams are commutative and surjective:
\begin{equation*}
\begin{tikzpicture}[baseline=(current  bounding  box.center)]
\node (LT) at (7,0) {$(x,y)$};
\node (LM) at (7,-1.5) {$\Psi_0(x,y)=(x,\hat{y}_x)$};
\node (LB) at (7,-3) {$\Phi_0(x,\hat{y}_x)=x\cdot y$};
\draw [|->] (LT) -- (LM);
\draw [|->] (LM) -- (LB);
\node (T) at (0,0) {$(X\times \{0,1\}^\Z,\sigma\times \sigma,\{\nu\vee\kappa : \kappa\in\mathcal{M}(\{0,1\}^\Z)\})$};
\node (M) at (0,-1.5) {$ (X\times \{0,1\}^\Z,\widetilde{\sigma},\{\widetilde{\mu} : (\pi_1)_\ast(\widetilde{\mu})=\nu\})$};
\node (B) at (0,-3) {$(X,\sigma)$};
\node (TR) at (-5,0) {};
\node (BR) at (-5,-3) {};
\draw (T) -- (-5,0) -- node[left] {$M$} (-5,-3);
\draw [->] (-5,-3) -- (B);
\draw[->]  (T) edge node[auto] {$\Psi_0$} (M)
		   (M) edge node[auto] {$\Phi_0$} (B);
\end{tikzpicture}
\end{equation*}
and
\begin{equation*}
\begin{tikzpicture}[baseline=(current  bounding  box.center)]
\node (LT) at (8,0) {$(x,y)$};
\node (LM) at (8,-1.5) {$(\pi_1\otimes M)(x,y)=(x,x\cdot y)$};
\node (LB) at (8,-3) {$\Psi_0(x,x\cdot y)=(x,\hat{y}_x)$.};
\node (LU) at (8,-4.5) {$\Phi_0(x,\hat{y}_x)=x\cdot y$.};
\draw [|->] (LT) -- (LM);
\draw [|->] (LM) -- (LB);
\draw [|->] (LB) -- (LU);
\node (T) at (0,0) {$(X\times \{0,1\}^\Z,\sigma\times \sigma,\{\nu\vee\kappa : \kappa\in\mathcal{M}(\{0,1\}^\Z)\})$};
\node (M) at (0,-1.5) {$(X\times \{0,1\}^\Z,\sigma\times \sigma,(\pi_1\otimes M)_\ast\{\nu\vee\kappa : \kappa\in\mathcal{M}(\{0,1\}^\Z)\})$};
\node (B) at (0,-3) {$(X\times \{0,1\}^\Z,\widetilde{\sigma},\{\widetilde{\mu} : (\pi_1)_\ast(\widetilde{\mu})=\nu\})$};
\node (U) at (0,-4.5) {$(X,\sigma)$};
\node (TR) at (-5.5,0) {};
\node (BR) at (-5.5,-4.5) {};
\draw (T) -- (-5.5,0) -- node[left] {$M$} (-5.5,-4.5);
\draw [->] (-5.5,-4.5) -- (U);
\draw[->]  (T) edge node[auto] {$\pi_1 \otimes M$} (M)
		   (M) edge node[auto] {$\Psi_0$} (B)
		   (B) edge node[auto] {$\Phi_0$} (U);
\end{tikzpicture}
\end{equation*}
\end{proposition}
\begin{proof}
It suffices to prove the commutativity (the surjectivity will then follow immediately by the fact that $X$ is a measure-theoretic subordinate subshift). We can assume without loss of generality that $\nu\neq\delta_{\boldsymbol{0}}$ (otherwise there is nothing to prove). Let $\hat{X}:=X\cap \hat{Z}$. Then clearly $(\nu\vee\kappa)(\hat{X}\times \{0,1\}^\Z)=\nu(\hat{X})=1$ for any joining $\nu\vee\kappa$. Moreover, we have
\[
 \Psi_0\circ (\sigma\times \sigma)=\widetilde{\sigma}\circ \Psi_0,\ \Phi_0\circ \widetilde{\sigma} = \sigma\circ \Phi_0  \text{ and }\Phi_0\circ \Psi_0=M\text{ on }\hat{X}\times \{0,1\}^\Z
\]
(for detailed calculation see the proof of Theorem 2.1.1 in~\cite{MR3356811}). This yields the commutativity of the first diagram. The commutativity of the second diagram requires that
\[
\Psi_0 \colon (X\times \{0,1\}^\Z,\sigma\times \sigma,(\pi_1\otimes M)_\ast\{\nu\vee\kappa : \kappa\in\mathcal{M}(\{0,1\}^\Z)\}) \to (X\times \{0,1\}^\Z,\widetilde{\sigma},\{\widetilde{\mu} : (\pi_1)_\ast(\widetilde{\mu})=\nu\})
\]
is a morphism, i.e.\ that -- apart from the equivariance condition $\Psi_0\circ (\sigma\times \sigma)=\widetilde{\sigma}\circ \Psi_0$ -- we have
\[
(\Psi_0)_\ast \circ (\pi_1\otimes M)_\ast (\nu\vee\kappa) \in \{\widetilde{\mu}\in\mathcal{M}(X\times \{0,1\}^\Z,\widetilde{\sigma}) : (\pi_1)_\ast(\widetilde{\mu})=\nu\}.
\]
However, this follows immediately since $\pi_1\circ \Psi_0 \circ (\pi_1\otimes M)(x,y)=x$.
\end{proof}

\begin{proof}[Proof of Theorem~\ref{toto}]
We can assume that $\nu\neq \delta_{\mathbf{0}}$ (otherwise there is nothing to prove since $\delta_{\mathbf{0}}$ is the unique invariant measure on $X$, cf.\ Remark~\ref{kiedyzero}). Notice first that by Remark~\ref{kiedyzero}, 
\begin{equation}\label{eqn:nier1}
\nu(1)\leq h(X).
\end{equation}
Moreover,
\begin{equation}\label{eqn:nier2}
h(X)\leq \sup \{h(\widetilde{\sigma},\widetilde{\mu}) : \widetilde{\mu}\in \mathcal{M}((\{0,1\}^\Z)^2,\widetilde{\sigma}) \text{ with }\widetilde{\mu}|_X=\nu\}.
\end{equation}
Now, take $\widetilde{\mu}\in\mathcal{M}((\{0,1\}^\Z)^2,\widetilde{\sigma})$ with $\widetilde{\mu}|_X=\nu$. By~\eqref{indukowane}, $\widetilde{\mu}_{C\times \{0,1\}^\Z}$ is $(\sigma\times \sigma)$-invariant and equal to some joining of the form $\nu_C\vee\kappa$. Moreover, we have
\[
h(\widetilde{\sigma}_{C\times \{0,1\}^\Z},\nu_C\vee \kappa)=h(\sigma_C\times \sigma,\nu_C\vee\kappa)=h(\sigma,\kappa)\leq 1
\]
since $h(\nu_C)=0$. Thus, by Abramov's formula,
\begin{equation}\label{eqn:nier3}
h(\widetilde{\sigma},\widetilde{\mu})=\nu(1)h(\sigma_C\times \sigma, \nu_C\vee\kappa)\leq\nu(1),
\end{equation}
with equality if and only if $\kappa=B_{1/2}$. Since $h(\nu_C)=0$, we have $\nu_C\vee\kappa=\nu_C\vee B_{1/2}=\nu_C\otimes B_{1/2}$. Combining~\eqref{eqn:nier1},~\eqref{eqn:nier2} and~\eqref{eqn:nier3} we get $h(X)=\nu(1)$. The uniqueness in~\eqref{eqn:nier3} implies the unique ergodicity of $X$.
\end{proof}

\begin{proof}[Proof of Theorem~\ref{int0}]
We can assume that $\nu\neq \delta_{\mathbf{0}}$ (otherwise there is nothing to prove). We will use the second diagram from the assertion of Proposition~\ref{diagramiki} -- let us now add potentials to this picture. Given $\varphi\colon X\to \R$, we define $\widetilde{\varphi}\colon X\times \{0,1\}^\Z\to \R$ by $\widetilde{\varphi}:=\varphi\circ \Phi_0$ and $\overline{\varphi}\colon X\times \{0,1\}^\Z\to \R$ by $\overline{\varphi}:=\widetilde{\varphi}\circ \Psi_0$. We have
\begin{equation*}
h(\sigma,\mu)+\int \varphi\, d\mu \leq h(\widetilde{\sigma},\widetilde{\mu})+\int \widetilde{\varphi}\, d\widetilde{\mu} 
\leq h(\sigma^{\times 2},(\pi_1\otimes M)_\ast(\nu\vee\kappa))+\int\overline{\varphi}\, d (\pi_1\otimes M)_\ast(\nu\vee\kappa),
\end{equation*}
whenever $\mu=\Phi_\ast(\widetilde{\mu})$ and $\widetilde{\mu}=(\pi_1\otimes M)_\ast (\nu\vee\kappa)$ (the corresponding integrals are just equal and the inequality between the entropies follows by the fact that we deal with factors). Therefore,
\begin{align*}
\mathcal{P}_{X,\varphi}&=\sup\left\{h(\sigma,\mu)+\int \varphi\, d\mu : \mu\in\mathcal{M}(X)\right\}\nonumber\\ 
&\leq \sup\left\{h(\widetilde{\sigma},\widetilde{\mu})+\int \widetilde{\varphi}\, d\widetilde{\mu} : \widetilde{\mu}\in\mathcal{M}(X\times \{0,1\}^\Z,\widetilde{\sigma})\text{ and }(\pi_1)_\ast(\widetilde{\mu})=\nu\right\}\label{nn1}\\
&\leq\sup\left\{ h(\sigma^{\times 2},(\pi_1\otimes M)_\ast(\nu\vee\kappa))+\int \overline{\varphi}\, d (\pi_1\otimes M)_\ast(\nu\vee\kappa) : \kappa\in\mathcal{M}(\{0,1\}^\Z) \right\}\nonumber\\
&=\sup\left\{ h(\sigma, M_\ast(\nu\vee\kappa))+\int\varphi\, d M_\ast(\nu\vee\kappa) : \kappa\in\mathcal{M}(\{0,1\}^\Z)\right\}=\mathcal{P}_{X,\varphi}
\end{align*}
(the first equality in the last line follows from $h(\sigma,\nu)=0$ and by the definition of $\overline{\varphi}$, while the second one holds since $X$ is a measure-theoretic subordinate subshift). Thus, to complete the proof, it suffices to compute 
\[
\sup\left\{h(\widetilde{\sigma},\widetilde{\mu})+\int \widetilde{\varphi}\, d\widetilde{\mu} : \widetilde{\mu}\in\mathcal{M}(X\times \{0,1\}^\Z,\widetilde{\sigma})\text{ and }(\pi_1)_\ast(\widetilde{\mu})=\nu\right\}
\]
and show that there is a unique measure realizing the above supremum. By Abramov's formula, we have
\begin{equation}\label{l1}
h(\widetilde{\sigma},\widetilde{\mu})=\nu(1)\cdot h(\widetilde{\sigma}_{C\times \{0,1\}^\Z},\widetilde{\mu}_{C\times \{0,1\}^\Z})=\nu(1)\cdot h(\sigma,(\pi_2)_\ast(\widetilde{\mu}_{C\times \{0,1\}^\Z}))
\end{equation}
(recall that $\widetilde{\sigma}_{C\times \{0,1\}^\Z}=\sigma_C\times \sigma$ and that $h(\sigma_C,\nu_C)=0$). Moreover,
\[
\widetilde{\varphi}(x,y)=\varphi\circ \Phi_0 (x,y),\ \text{with }(\Phi_0(x,y))_0=x_0\cdot y_0
\]
and it follows that
\[
\widetilde{\varphi}(x,y)=\varphi(M(x,y))
\]
since $\varphi(x)=\varphi(x_0)$.
Thus,
\begin{align*}
\int \widetilde{\varphi}\, d\widetilde{\mu}&=(\widetilde{\mu}(0,0)+\widetilde{\mu}(0,1))\varphi(0)+\widetilde{\mu}(1,0)\varphi(0)+\widetilde{\mu}(1,1)\varphi(1)\\
&=\nu(0)\varphi(0)+\widetilde{\mu}(1,0)\varphi(0)+\widetilde{\mu}(1,1)\varphi(1)
\end{align*}
Notice also that for $i\in\{0,1\}$ we have
\begin{equation}\label{l2}
(\pi_2)_\ast(\widetilde{\mu}_{C\times \{0,1\}^\Z})(i)=\widetilde{\mu}_{C\times \{0,1\}^\Z}(\ast ,i)=\frac{\widetilde{\mu}(1,i)}{\nu(1)}.
\end{equation}
Combining~\eqref{l1} and~\eqref{l2}, we obtain
\begin{align*}
h(\widetilde{\sigma},\widetilde{\mu})+\int \widetilde{\varphi}\, d\widetilde{\mu} &= \nu(0)\varphi(0)+\nu(1)\left(h(\sigma, (\pi_2)_\ast(\widetilde{\mu}_{C\times \{0,1\}^\Z}))+\int \varphi\, d (\pi_2)_\ast(\widetilde{\mu}_{C\times \{0,1\}^\Z})\right)\\
&\leq \nu(0)\varphi(0)+\nu(1)\log (2^{\varphi(0)}+2^{\varphi(1)}),
\end{align*}
with the equality only when $(\pi_2)_\ast(\widetilde{\mu}_{C\times \{0,1\}^\Z})$ is equal to the Bernoulli measure $B_{p}$ for $p=\frac{2^{\varphi(0)}}{2^{\varphi(0)}+2^{\varphi(1)}}$. Now, notice that
\[
(\pi_2)_\ast(\widetilde{\mu}_{C\times \{0,1\}^\Z})) = B_{p}\iff \widetilde{\mu}_{C\times \{0,1\}^\Z}=\nu_C\otimes B_{p}
\]
($h(\sigma_C,\nu_C)=0$ and thus $\nu_C\otimes B_{p}$ is the only joining of $\nu_C$ and $B_{p}$). The latter equality is, in turn, equivalent to $\widetilde{\mu}=\nu\otimes B_{p}$. It remains to notice that
\[
(\Phi_0)_\ast(\nu\otimes B_{p})=(\Phi_0)_\ast \circ (\Psi_0)_\ast (\nu\otimes B_{p})=M_\ast (\nu\otimes B_{p}).
\]
\end{proof}

\begin{remark}[$\mathscr{B}$-free systems]\label{commentsB}
We will see now how to use Proposition~\ref{prop_gene} to conclude that $\widetilde{X}_\eta$ is always a measure-theoretic subordinate subshift with $\nu_\eta$ being its base measure and Theorem~\ref{toto} to obtain a short and elegant proof of the intrinsic ergodicity of $\widetilde{X}_\eta$. 

Fix $\mathscr{B}\subseteq \mathbb{N}$ and recall the notation $\mathscr{B}_K=\{b\in\mathscr{B} : b<K\}$. Let $\eta_K:=\mathbf{1}_{\mathcal{F}_{\mathscr{B}_K}}$. It follows by Theorem~\ref{daverd} (and by the very definition of $\eta$ and $\eta_K$) that 
\begin{equation}\label{isgood}
\text{$(\eta_K)$ is good (and its coordinatewise limit equals $\eta$)}.
\end{equation}
Moreover, each $\eta_K$ is generic for $\nu_{\eta_K}$, while 
\begin{equation}\label{eq:qg}
\eta \text{ is quasi-generic for }\nu_\eta.
\end{equation} 
It follows by Lemma~\ref{lemat_gene} that 
\begin{equation}\label{eq:gra}
\nu_{\eta_K}\to\nu_\eta.
\end{equation} 
It suffices to use Proposition~\ref{prop_gene} to conclude that $\bigcap_{K\geq 1}{[\mb{0},\eta_K]}$ and $\overline{[\mb{0},\eta]}=\widetilde{X}_\eta$ are measure-theoretic subordinate subshifts with base measure $\nu_\eta$. Thus, we can apply Theorem~\ref{toto} and obtain the intrinsic ergodicity of $\widetilde{X}_\eta$, with $\nu=\nu_\eta$. (Moreover, Theorem~\ref{int0} also holds in this setup.)

Note that the subshift $\widetilde{X}_\eta$ was known to be intrinsically ergodic earlier than it was proved that $\widetilde{X}_\eta$ is a measure-theoretic subordinate subshift. Let us go more into details. The proofs of Theorem~\ref{int1} (and Theorem~\ref{int0}) draw upon ideas from the old paper of Marcus and Newhouse~\cite{MR550415}. In case of the hereditary closure of $\mathscr{B}$-free subshifts the first result on invariant measures was that of Peckner~\cite{MR3430278}: he showed that the square-free subshift is intrinsically ergodic. In fact, his proof has the same flavour as what we do here (he gives a reference to~\cite{MR550415}) and some of the ideas behind it led later to the description of invariant measures of the hereditary closure of $\mathscr{B}$-free subshifts, first in the Erd\"os case in~\cite{MR3356811} and then for all sets $\mathscr{B}\subseteq \mathbb{N}$ in~\cite{MR3803141}. 
\end{remark}

\subsection{Sandwich measure-theoretic subordinate subshifts}\label{NNNNN}
This section is devoted to finding a formula for $\mathcal{P}_{X,\varphi}$, where $X$ is a sandwich measure-theoretic subordinate subshift with a zero entropy base measure and $\varphi(x)=\varphi(x_0)$. We then apply it to $\mathscr{B}$-free systems and compare with the earlier known related results. The main result in this section is as follows.
\begin{theorem}[Cf.\ Theorem~\ref{int0}]\label{HHH}
Suppose that $X\subseteq \{0,1\}^\Z$ is a sandwich measure-theoretic subordinate subshift with base measure $\rho$ of zero entropy. Let $\varphi\colon X\to\R$ be of the form $\varphi(x)=\varphi(x_0)$. Then 
\begin{equation}\label{wz}
\mathcal{P}_{X,\varphi}=\rho(0,0)\varphi(0)+\rho(1,1)\varphi(1)+\rho(0,1)\log(2^{\varphi(0)}+2^{\varphi(1)})
\end{equation}
and there is the following corresponding unique equilibrium state: $N_\ast(\rho\otimes B_p)$, where $p=2^{\varphi(0)}/(2^{\varphi(0)}+2^{\varphi(1)})$.
\end{theorem}
Note that as an immediate consequence, we obtain the following corollary.
\begin{corollary}[Cf.\ Theorem~\ref{toto}]\label{GGG}
Suppse that $X\subseteq \{0,1\}^\Z$ is a sandwich measure-theoretic subordinate subshift with base measure $\rho$ of zero entropy. Then $h_{top}(X)=\rho(0,1)=\nu_1(1)-\nu_0(1)$. Moreover, $X$ is intrinsically ergodic and $N_\ast(\rho\otimes B_{1/2})$ is the unique measure of maximal entropy.
\end{corollary}

Before we explain how to prove Theorem~\ref{HHH}, we need to introduce some notation. Recall that $\mathbf{A}=\{(w,x)\in (\{0,1\}^\Z)^2: w\leq x\}$ is the set of points in $(\{0,1\}^\Z)^2$ located above the diagonal. Let $S\colon \mathbf{A}\to \{0,1\}^\Z$ be the coordinatewise subtraction: $S(w,x)=x-w$. Recall that 
\[
\hat{Z}=\{x\in\{0,1\}^\Z : \text{ supp }x\text{ is unbounded from below and from above} \}.
\]
Let $\hat{\mathbf{A}}:=\{(w,x)\in \mathbf{A} : x-w \in \hat{Z}\}=\mathbf{A}\cap S^{-1}\hat{Z}$. Consider
\[
\Psi\colon S^{-1}\hat{Z}\times \{0,1\}^\Z \to S^{-1}\hat{Z}\times \{0,1\}^\Z,\ \Psi(w,x,y)=(w,x,\hat{y}_x)
\]
and $\Phi\colon S^{-1}\hat{Z}\times \{0,1\}^\Z \to \{0,1\}^\Z$ given by the condition that 
\begin{equation*}
\text{$\Phi(w,x,y)$ is the unique 0-1-sequence such that $w\leq \Phi(w,x,y)\leq x$ and $\Phi(w,x,y)^{\widehat{}}_{x-w}=y$.}
\end{equation*}
(Thus, $\Phi_0(x,y)=\Phi(\mathbf{0},x,y)$ and $\Psi_0(x,y)=\Psi(\mathbf{0},x,y)$, cf.\ Section~\ref{NNN}.) Last but not least, consider the following skew product:
\[
\widetilde{\sigma^{\times 2}}\colon (\{0,1\}^\Z)^3 \to (\{0,1\}^\Z)^3,\ \widetilde{\sigma^{\times  2}}(w,x,y)=(\sigma w,\sigma x,\sigma^{x_0-w_0}y).
\]
Take $C:=\{(w,x)\in \hat{\mathbf{A}} : w_0=0<1=x_0\}$ and consider the induced map $\widetilde{\sigma^{\times 2}}_{C\times \{0,1\}^\Z}$ corresponding to $\widetilde{\sigma^{\times 2}}$ and $C\times \{0,1\}^\Z$. Notice that
\[
\widetilde{\sigma^{\times 2}}_{C\times \{0,1\}^\Z}=(\sigma^{\times 2})_C \times \sigma.
\]

We will also use the following proposition whose proof is completely analogous to that of Proposition~\ref{diagramiki} (we skip it, leaving the details to the reader).
\begin{proposition} \label{diagramiki2}
Suppose that $X$ is a sandwich measure-theoretic subordinate subshift with base measure $\rho$. The following two diagrams are commutative and surjective:
\begin{equation*}
\begin{tikzpicture}[baseline=(current  bounding  box.center)]
\node (T) at (0,0) {$(X\times X\times \{0,1\}^\Z,\sigma^{\times 3},\{\rho\vee\kappa : \kappa\in\mathcal{M}(\{0,1\}^\Z)\})$};
\node (M) at (0,-1.5) {$ (X\times X\times \{0,1\}^\Z,\widetilde{\sigma^{\times 2}},\{\widetilde{\mu} : (\pi_{1,2})_\ast(\widetilde{\mu})=\rho\})$};
\node (B) at (0,-3) {$(X,\sigma)$};
\node (TR) at (-5,0) {};
\node (BR) at (-5,-3) {};
\draw (T) -- (-5,0) -- node[left] {$N$} (-5,-3);
\draw [->] (-5,-3) -- (B);
\draw[->]  (T) edge node[auto] {$\Psi$} (M)
		   (M) edge node[auto] {$\Phi$} (B);
\end{tikzpicture}
\end{equation*}
and
\begin{equation*}
\begin{tikzpicture}[baseline=(current  bounding  box.center)]
\node (T) at (0,0) {$(X\times X\times \{0,1\}^\Z,\sigma^{\times 3},\{\rho\vee\kappa : \kappa\in\mathcal{M}(\{0,1\}^\Z)\})$};
\node (M) at (0,-1.5) {$(X\times X\times \{0,1\}^\Z,\sigma^{\times 3},(\pi_{1,2}\otimes N)_\ast\{\rho\vee\kappa : \kappa\in\mathcal{M}(\{0,1\}^\Z)\})$};
\node (B) at (0,-3) {$(X\times X\times \{0,1\}^\Z,\widetilde{\sigma^{\times 2}},\{\widetilde{\mu} : (\pi_{1,2})_\ast(\widetilde{\mu})=\rho\})$};
\node (U) at (0,-4.5) {$(X,\sigma)$};
\node (TR) at (-6,0) {};
\node (BR) at (-6,-4.5) {};
\draw (T) -- (-6,0) -- node[left] {$N$} (-6,-4.5);
\draw [->] (-6,-4.5) -- (U);
\draw[->]  (T) edge node[auto] {$\pi_{1,2} \otimes N$} (M)
		   (M) edge node[auto] {$\Psi$} (B)
		   (B) edge node[auto] {$\Phi$} (U);
\end{tikzpicture}
\end{equation*}
\end{proposition}

\begin{proof}[Proof of Theorem~\ref{HHH}]
Suppose first that $\nu_0(1)=\nu_1(1)$. Then it follows by Proposition~\ref{uwazka} that
\[
\mathcal{P}_{X,\varphi}=\int \varphi\, d\nu_1 + h(\nu_1)=\nu_1(0)\varphi(0)+\nu_1(1)\varphi(1),
\]
which agrees with~\eqref{wz} as in this case $\rho$ is the diagonal self-joining of $\nu_0=\nu_1$ (in particular, $\rho\left(\begin{matrix}0 \\ 1\end{matrix}\right)=0$).

For the case when $\nu_0(1)<\nu_1(1)$, we will use the second diagram from Proposition~\ref{diagramiki2}. Let us now add potentials to this picture. Given $\varphi\colon X\to \R$, we define $\widetilde{\varphi}\colon X\times \{0,1\}^\Z\to \R$ by $\widetilde{\varphi}:=\varphi\circ \Phi$ and $\overline{\varphi}\colon X\times \{0,1\}^\Z\to \R$ by $\overline{\varphi}:=\widetilde{\varphi}\circ \Psi$. We have
\begin{equation*}
h(\sigma,\mu)+\int \varphi\, d\mu \leq h(\widetilde{\sigma^{\times 2}},\widetilde{\mu})+\int \widetilde{\varphi}\, d\widetilde{\mu} 
\leq h(\sigma^{\times 2},(\pi_{1,2}\otimes N)_\ast(\rho\vee\kappa))+\int\overline{\varphi}\, d (\pi_{1,2}\otimes N)_\ast(\rho\vee\kappa),
\end{equation*}
whenever $\mu=\Phi_\ast(\widetilde{\mu})$ and $\widetilde{\mu}=(\pi_{1,2}\otimes N)_\ast (\rho\vee\kappa)$ (the corresponding integrals are just equal and the inequality between the entropies follows by the fact that we deal with factors). Therefore,
\begin{align*}
\mathcal{P}_{X,\varphi}&=\sup\left\{h(\sigma,\mu)+\int \varphi\, d\mu : \mu\in\mathcal{M}(X)\right\}\nonumber\\ 
&\leq \sup\left\{h(\widetilde{\sigma^{\times 2}},\widetilde{\mu})+\int \widetilde{\varphi}\, d\widetilde{\mu} : \widetilde{\mu}\in\mathcal{M}(X\times X\times \{0,1\}^\Z,\widetilde{\sigma^{\times 2}})\text{ and }(\pi_{1,2})_\ast(\widetilde{\mu})=\rho\right\}\label{nn1}\\
&\leq\sup\left\{ h(\sigma^{\times 3},(\pi_{1,2}\otimes N)_\ast(\rho\vee\kappa))+\int \overline{\varphi}\, d (\pi_{1,2}\otimes N)_\ast(\rho\vee\kappa) :\kappa\in\mathcal{M}(\{0,1\}^\Z) \right\}\nonumber\\
&=\sup\left\{ h(\sigma, N_\ast(\rho\vee\kappa))+\int\varphi\, d N_\ast(\rho\vee\kappa):\kappa\in\mathcal{M}(\{0,1\}^\Z) \right\}=\mathcal{P}_{X,\varphi}
\end{align*}
(the first equality in the last line follows from $(\pi_{1,2}\otimes N)_\ast(\rho\vee\kappa)=\rho\vee(N_\ast(\rho\vee\kappa))$, $h(\sigma^{\times 2},\rho)=0$ and by the definition of $\overline{\varphi}$, while the second one holds since $X$ is a sandwich measure-theoretic subordinate subshift). Thus, to complete the proof, it suffices to compute 
\[
\sup\left\{h(\widetilde{\sigma^{\times 2}},\widetilde{\mu})+\int \widetilde{\varphi}\, d\widetilde{\mu} : \widetilde{\mu}\in\mathcal{M}(X\times X\times \{0,1\}^\Z,\widetilde{\sigma^{\times 2}})\text{ and }(\pi_{1,2})_\ast(\widetilde{\mu})=\rho\right\}
\]
and show that there is a unique measure realizing the above supremum. By Abramov's formula, we have
\begin{equation}\label{l1a}
h(\widetilde{\sigma},\widetilde{\mu})=\rho(C)\cdot h(\widetilde{\sigma^{\times 2}}_{C\times \{0,1\}^\Z},\widetilde{\mu}_{C\times \{0,1\}^\Z})=(\nu_1(1)-\nu_0(1))\cdot h(\sigma,(\pi_3)_\ast(\widetilde{\mu}_{C\times \{0,1\}^\Z}))
\end{equation}
(recall that $\widetilde{\sigma^{\times 2}}_{C\times \{0,1\}^\Z}=(\sigma^{\times 2})_C\times \sigma$, $\widetilde{\mu}_{C\times \{0,1\}^\Z}=\rho_C \vee (\pi_3)_\ast(\widetilde{\mu}_{C\times \{0,1\}^\Z})$ and  $h((\sigma^{\times 2})_C,\rho_C)=0$). Moreover,
\[
\widetilde{\varphi}(w,x,y)=\varphi\circ \Phi (w,x,y),\ \text{with }(\Phi(w,x,y))_0=(N(w,x,y))_0
\]
and it follows that
\[
\widetilde{\varphi}(w,x,y)=\varphi(N(w,x,y)).
\]
Thus,
\begin{align*}
\int \widetilde{\varphi}\, d\widetilde{\mu}&=(\widetilde{\mu}(0,0,*)+\widetilde{\mu}(0,1,0))\varphi(0)+(\widetilde{\mu}(0,1,1)+\widetilde{\mu}(1,1,*))\varphi(1)\\
&=\rho(0,0)\varphi(0)+\rho(1,1)\varphi(1)+\widetilde{\mu}(0,1,0)\varphi(0)+\widetilde{\mu}(0,1,1)\varphi(1).
\end{align*}
Notice also that for $i=1,2$, we have
\begin{equation}\label{l2a}
(\pi_3)_\ast(\widetilde{\mu^{\times 2}}_{C\times \{0,1\}^\Z})(i)=\widetilde{\mu^{\times 2}}_{C\times \{0,1\}^\Z}(\ast, \ast ,i)=\frac{\widetilde{\mu^{\times 2}}(0,1,i)}{\rho(0,1)}.
\end{equation}
Combining~\eqref{l1a} and~\eqref{l2a}, we obtain
\begin{align*}
h(\widetilde{\sigma^{\times 2}},\widetilde{\mu})+\int \widetilde{\varphi}\, d\widetilde{\mu} =& \rho(0,0)\varphi(0)+\rho(1,1)\varphi(1)\\
&+\rho(0,1)(h(\sigma, (\pi_3)_\ast(\widetilde{\mu^{\times 2}}_{C\times \{0,1\}^\Z}))+\int \varphi\, d (\pi_3)_\ast(\widetilde{\mu^{\times 2}}_{C\times \{0,1\}^\Z}))\\
\leq& \rho(0,0)\varphi(0)+\rho(1,1)\varphi(1)+\rho(0,1)\log (2^{\varphi(0)}+2^{\varphi(1)})
\end{align*}
with the equality only when $(\pi_3)_\ast(\widetilde{\mu^{\times 2}}_{C\times \{0,1\}^\Z}))$ is equal to the Bernoulli measure $B_{p}$ with parameter $p=\frac{2^{\varphi(0)}}{2^{\varphi(0)}+2^{\varphi(1)}}=B_p(0)$. Now, notice that
\[
(\pi_3)_\ast(\widetilde{\mu^{\times 2}}_{C\times \{0,1\}^\Z})) = B_{p}\iff \widetilde{\mu}_{C\times \{0,1\}^\Z}=\rho_C\otimes B_{p}
\]
($h((\sigma^{\times 2})_C,\rho_C)=0$ and thus $\rho_C\otimes B_{p}$ is the only joining of $\rho_C$ and $B_{p}$). The latter equality is, in turn, equivalent to $\widetilde{\mu^{\times 2}}=\rho\otimes B_{p}$. It remains to notice that
\[
\Phi_\ast(\rho\otimes B_{p})=\Phi_\ast \circ \Psi_\ast (\rho\otimes B_{p})=N_\ast (\rho\otimes B_{p}).
\]
\end{proof}

\begin{remark}[$\mathscr{B}$-free systems]\label{commentsB1}
As in Remark~\ref{commentsB}, one can use Corollary~\ref{GGG} to obtain a short and elegant proof of the intrinsic ergodicity of ${X}_\eta$, provided that $\eta^*$ is a regular Toeplitz sequence. Indeed, fix $\mathscr{B}\subseteq \mathbb{N}$ such that $\eta^*$ is a regular Toeplitz sequence and recall that $\mathscr{B}_K=\{b\in\mathscr{B} : b<K\}$. Let $\eta_K:=\mathbf{1}_{\mathcal{F}_{\mathscr{B}_K}}$ and recall that one can find periodic sequences $\underline{\eta}_K$ such that~\eqref{last2},~\eqref{last1} and~\eqref{last} hold. In other words, 
\begin{equation}\label{isgood1}
\text{$(\underline{\eta}_K,\eta_K)$ is good (and its coordinatewise limit equals $(\eta^*,\eta)$).}
\end{equation}
Moreover, each $(\underline{\eta}_K,\eta_K)$ is generic for some measure $\rho_K$ as a periodic point, while 
\begin{equation}\label{limitmeasure}
\begin{split}
&(\eta^*,\eta) \text{ is quasi-generic along $(N_i)$ for a certain joining of $\nu_{\eta^*}$ and $\nu_\eta$},\\
&\text{denoted by }\nu_{\eta^*}\triangle\nu_\eta.
\end{split}
\end{equation}
It follows by Lemma~\ref{lemat_gene1} that 
\begin{equation}\label{eq:proksym}
\rho_K\to \nu_{\eta^*}\triangle\nu_\eta.
\end{equation} 
It suffices to use Proposition~\ref{III} to conclude that $\bigcap_{K\geq 1}{[\underline{\eta}_K,\eta_K]}$ and $\overline{[\eta^*,\eta]}$ are sandwich measure-theoretic subordinate subshifts with base measure $\nu_{\eta^*}\triangle\nu_\eta$. Moreover, we have
\[
\overline{[\eta^*,\eta']}=X_{\eta'}\subseteq X_\eta \subseteq \overline{[\eta^*,\eta]}
\]
(see Proposition 2.1 in~\cite{ADJ}; implicit in~\cite{MR4280951} -- see Corollary 1.2 therein) and $\nu_{\eta^*}\triangle\nu_{\eta'}=\nu_{\eta^*}\triangle\nu_\eta$ (see (22) in~\cite{ADJ}). Thus, 
\[
\mathcal{P}(\overline{[\eta^*,\eta']})=\mathcal{P}(X_{\eta'})=\mathcal{P}(X_\eta)=\mathcal{P}(\overline{[\eta^*,\eta]})
\]
and it follows immediately that also $X_\eta$ is a sandwich measure-theoretic subordinate subshift with base measure $\nu_{\eta^*}\triangle\nu_\eta$ (whenever $\eta^*$ is a regular Toeplitz sequence). The intrinsic ergodicity of $X_\eta$ follows immediately by Corollary~\ref{GGG}. (Clearly, Theorem~\ref{HHH} also holds in this setup with $\rho=\nu_{\eta^*}\triangle\nu_\eta$.)

Recall that the subshift $X_\eta$ for $\eta^*$ being a regular Toeplitz sequence was first proved to be a sandwich measure-theoretic subordinate subshift in~\cite{ADJ}. The same goes for the unique ergodicity of $X_\eta$.
\end{remark}

\section{Topological pressure: $\varphi(x)=\varphi(x_0,x_1)$}\label{NNNN}
In this section we will deal with a more complicated setup, namely when the potential $\varphi$ depends on two consecutive coordinates, i.e.\ $\varphi(x)=\varphi(x_0,x_1)$. Let
\[
M:=\left(\begin{matrix}2^{\varphi(0,0)} & 2^{\varphi(0,1)} \\ 2^{\varphi(1,0)} & 2^{\varphi(1,1)}\end{matrix} \right).
\]
We will distinguish two cases, depending on whether $M$ is invertible or not. 

Suppose that $\det M=0$, i.e.\ $\varphi(0,0)+\varphi(1,1)=\varphi(1,0)+\varphi(0,1)$. Recall that for each $\mu\in\mathcal{M}(\{0,1\}^\Z)$, we have $\mu(01)=\mu(0)-\mu(00)=\mu(10)+\mu(00)-\mu(00)=\mu(10)$. Thus, for any subshift $X\subseteq \{0,1\}^\Z$ and any $\mu\in\mathcal{M}(X)$, we have
\begin{align*}
\int_X \varphi\, d\mu&=\sum_{i,j\in\{0,1\}}\mu(ij)\varphi(i,j)\\
&=\mu(01)(\varphi(0,1)+\varphi(1,0))+\mu(00)\varphi(0,0)+\mu(11)\varphi(1,1)\\
&=\mu(01)(\varphi(0,0)+\varphi(1,1))+\mu(00)\varphi(0,0)+\mu(11)\varphi(1,1)\\
&=\mu(0)\varphi(0,0)+\mu(1)\varphi(1,1)=\int\psi\, d\mu,
\end{align*}
where $\psi(i)=\varphi(i,i)$ for $i=0,1$ and $\mathcal{P}_{X,\varphi}=\mathcal{P}_{X,\psi}$.

From now on we will make no assumptions on $\det M$. Let $\lambda^+$, $\lambda^-$ be the eigenvalues of $M$, with $\lambda^+>|\lambda^-|$ (i.e.\ $\lambda^+$ is the Perron-Frobenius eigenvalue of $M$).

\subsection{Full shift}
We start by presenting some well-known results and methods for the full 0-1-shift. To begin with, we have
\begin{equation}\label{wynik}
\mathcal{P}_{\{0,1\}^\Z,\varphi}=\log \lambda^+.
\end{equation}
There are various ways to prove this, e.g., using Walters' method~\cite{MR390180}, we encode the full $0$-$1$ shift as a finite type shift over the alphabet consisting of $0$-$1$-blocks of length two, i.e.\ $\{00,01,10,11\}$. The product of the corresponding incidence matrix and the diagonal matrix with the entries determined by the values of $\varphi$ equals
\begin{equation*}
\left(\begin{matrix}
1 & 1 & 0 & 0 \\ 0 & 0 & 1 & 1 \\ 1 & 1 & 0 & 0 \\ 0 & 0 & 1 & 1 
\end{matrix} \right)
\left(\begin{matrix} 2^{\varphi(0,0)} & 0 & 0 & 0 \\ 0 & 2^{\varphi(0,1)} & 0 & 0\\ 0 & 0 & 2^{\varphi(1,0)} & 0 \\ 0 & 0 & 0 & 2^{\varphi(1,1)} \end{matrix}
\right)
=
\left(\begin{matrix}
2^{\varphi(0,0)} & 2^{\varphi(0,1)} &  0 & 0 \\
0 & 0 & 2^{\varphi(1,0)} & 2^{\varphi(1,1)}\\
2^{\varphi(0,0)} & 2^{\varphi(0,1)} &  0 & 0\\
0 & 0 & 2^{\varphi(1,0)} & 2^{\varphi(1,1)}
\end{matrix} \right).
\end{equation*}
The characteristic polynomial $Q(\lambda)$ of this matrix equals $\lambda^2\cdot P(\lambda)$, where 
\[
P(\lambda)=\lambda^2-(2^{\varphi(0,0)}+2^{\varphi(1,1)})\lambda+\det M.
\] 
Notice here that $P(\lambda)$ is nothing but the characteristic polynomial of $M$. Therefore, the corresponding Perron-Frobenius eigenvalues are the same both for the $4$ by $4$ matrix presented above and for $M$, we will denote them by $\lambda^+$ and conclude that~\eqref{wynik} indeed holds. 

The above method, however, is difficult to adjust to the $\mathscr{B}$-free setting. Therefore, let us look at the problem of computing $\mathcal{P}_{\{0,1\}^\Z,\varphi}$ in a different way. We will begin by introducing some notation. For $a,b\in \{0,1\}$, let
\[
\mathcal{L}_{a,b}^n:=\{A=a_1\dots a_{n}\in \{0,1\}^n : a_1=a, a_{n}=b\}.
\]
Let
\[
Z_{a,b}^n:=\sum_{A\in\mathcal{L}_{a,b}^n}\prod_{\ell=1}^{n-1}M_{a_{\ell},a_{\ell+1}}=\left(M^{n-1} \right)_{a,b} \text{ for }n\geq 1.
\]
By the very definition, we have
\[
\mathcal{P}_{\{0,1\}^\Z,\varphi}=\lim_{n\to\infty}\frac{1}{n}\log \sum_{A\in \{0,1\}^n}2^{\sup_{{x}\in A} \varphi^{(n)}({x})}.
\]
Notice that for any ${x}\in A=a_1\dots a_{n-1}$, we have 
\(
|\varphi^{(n)}({x})-\varphi^{(n)}({0a_1\dots a_{n-1}0})|\leq 4  \|\varphi\|_{sup}.
\)
This gives us
\[
\mathcal{P}_{\{0,1\}^\Z,\varphi}=\lim_{n\to\infty}\frac{1}{n}\sum_{A\in \{0,1\}^{n-1}}2^{\varphi^{(n)}(0A0)}=\lim_{n\to\infty}\frac{1}{n}\log Z_{0,0}^{n+1}=\lim_{n\to\infty}\frac{1}{n}\log Z_{0,0}^n.
\]
\begin{lemma}\label{wzoreczek1} 
There exist two real numbers $\lambda_-<\lambda_+$, $\lambda_+>0$, such that 
for $a,b\in \{0,1\}$, there exist constants $C_{a,b}^+$, $C_{a,b}^-$ such that for every $n$
\[
Z_{a,b}^n=C_{a,b}^+\lambda_+^{n-1}+C_{a,b}^-\lambda_{-}^{n-1}.
\]
Moreover, for every $a,b\in\{0,1\}$, $C_{a,b}^+>0$.
\end{lemma}
\begin{proof}
Recall that
\(
Z_{a,b}^n=(M^{n-1})_{a,b}.
\)
Matrix $M$ is positive, hence by the Perron-Frobenius Theorem it has two distinct real eigenvalues $\lambda_-<\lambda_+$, the larger of which (by the absolute value) is positive. Thus, $M$ can be diagonalized:
\[
M=I D I^{-1}, \text{ where }D=\left(\begin{matrix}\lambda_+ & 0 \\ 0 & \lambda_- \end{matrix} \right).
\]
Therefore,
\[
M^{n-1}=I D^{n-1} I^{-1}
\]
and hence all of the elements of $M^{n-1}$ are linear combinations of $\lambda_+^{n-1}$ and $\lambda_-^{n-1}$, with coefficients provided by $I$ and $I^{-1}$ (and thus independent of $n$). The coefficients for the larger eigenvalue are positive because the matrix $M$, and hence $M^{n-1}$ for every $n$, is positive.
\end{proof}
It follows immediately by the above lemma that
\[
\mathcal{P}_{\{0,1\}^\Z,\varphi}=\lim_{n\to\infty}\frac{1}{n}\log Z_{0,0}^n=\log \lambda^+.
\]

\subsection{Sandwich subordinate subshifts with periodic bounds}\label{DDD}
The next step before we pass to the $\mathscr{B}$-free setting is a simpler setup of sandwich subordinate subshifts with periodic bounds, i.e., $w\leq x$, $w,x\in\{0,1\}^\Z$, are periodic and we consider $\overline{[w,x]}=[w,x]$. Let $s$ be the minimal common period of $w$ and $x$. The extreme case is when $w=\mathbf{0}$ and $x=\mathbf{1}$, where we deal with the full shift. Let us now assume that $(w,x)\neq (\mathbf{0},\mathbf{1})$. Then, in particular, there exists $i\in \Z$ such that $w_i=x_i$ and without loss of generality we can assume that $i=0$. We have
\begin{align*}
\mathcal{P}_{[w,x],\varphi}&=\lim_{n\to\infty}\frac{1}{n}\log \sum_{A\in\mathcal{L}_n([w,x])}2^{\sup_{y\in A}\varphi^{(n)}(y)}\\
&=\lim_{n\to\infty}\frac{1}{ns}\log \sum_{i=0}^{s-1}\sum_{\sigma^i w[0,ns-1]\leq A \leq \sigma^i x[0,ns-1]}2^{\sup_{y\in A}\varphi^{(ns)}(y)}\\
&=\lim_{n\to\infty}\frac{1}{ns}\log \sum_{w[0,ns-1]\leq A\leq x[ns-1]}\sum_{i=0}^{s-1}2^{\sup_{\sigma^i y\in A}\varphi^{(ns)}(y)}.
\end{align*}
It follows that
\[
\mathcal{P}_{[w,x],\varphi}=\lim_{n\to\infty}\frac{1}{ns}\log\sum_{w[0,ns-1]\leq A\leq x[ns-1]}s 2^{\varphi^{(ns)}(Ax_0)}=\lim_{n\to\infty}\frac{1}{ns}\log\sum_{w[0,ns-1]\leq A\leq x[ns-1]} 2^{\varphi^{(ns)}(Ax_0)}.
\]
Thus, as $(w[0,ns-1],x[0,ns-1])$ is the concatenation of $n$ copies of $(w[0,s-1],x[0,s-1])$, 
\[
\mathcal{P}_{[w,x],\varphi}=\lim_{n\to\infty}\frac{1}{ns}\log\left(\sum_{w[0,s-1]\leq A\leq x[0,s-1]} 2^{\varphi^{(s)}(Ax_0)}\right)^n=\frac{1}{s}\log \sum_{w[0,s-1]\leq A\leq x[0,s-1]}2^{\varphi^{(s)}(Ax_0)}.
\]

Now, each block $w[0,s-1]\leq A=a_0\dots a_{s-1}\leq x[0,s-1]$ has certain ``fixed'' entries, namely $w_i=x_i$ implies $a_i=w_i=x_i$. Let $0= i_1<\dots <i_k<i_{k+1}=s$ be all positions in $[0,s]$ with $w_{i_j}=x_{i_j}=:f_j\in \{0,1\}$. The remaining positions on $A$ are ``free'', the symbols there can be either $0$ or $1$. It follows that
\[
\mathcal{P}_{[w,x],\varphi}=\frac{1}{s}\log \prod_{1\leq j\leq k}\sum_{w[i_j,i_{j+1}-1]\leq A_j\leq x[i_j,i_{j+1}-1]}2^{\varphi^{(i_{j+1}-i_j)}(A_j f_{j+1})}.
\]
(observe that $f_{k+1}:=f_1$). It follows that
\[
\mathcal{P}_{[w,x],\varphi}=\frac{1}{s}\log \prod_{1\leq j\leq k}Z_{f_j,f_{j+1}}^{i_{j+1}-i_j+1}.
\]
For $a,b\in\{0,1\}$ and $\ell\geq 2$, let 
\begin{equation}\label{wzornam}
m_{a,b}^\ell:=\rho\left(\begin{matrix}a & 0 & \dots & 0 & b \\ a & 1 & \dots & 1 & b \end{matrix} \right),
\end{equation}
where the pair $\begin{matrix}0\\ 1\end{matrix}$ appears $\ell-2$ times and $\rho\in\mathcal{M}((\{0,1\}^\Z)^2)$ is the measure
for which the pair $(w,x)$ is generic. Notice that
\begin{equation}\label{eqn:mab}
m_{a,b}^\ell=\frac{1}{s}\#\{1\leq j\leq k : i_{j+1}-i_j+1=\ell\}.
\end{equation}
Thus,
\begin{equation}\label{sumka}
\mathcal{P}_{[w,x],\varphi}=\sum_{a,b\in\{0,1\},\ell\geq 2}m_{a,b}^\ell \log Z_{a,b}^\ell
\end{equation}
(note that this sum is finite as $(w,x)$ is periodic and hence for large values of $\ell$ we have $m_{a,b}^\ell=0$). Moreover,~\eqref{eqn:mab} implies
\[
\sum_{a,b\in \{0,1\}}\sum_{\ell\geq 2}(\ell-1)m_{a,b}^\ell=1.
\]
Thus, taking into account Lemma~\ref{wzoreczek1}, we have proved the following result (to include the full shift case in the same formula, if $(w,x)=(\mathbf{0},\mathbf{1})$ we formally set $m_{a,b}^\ell\equiv 0$ for all $a,b$ and $\ell$).
\begin{theorem}\label{OKRESOWE}
Let $w,x\in\{0,1\}^\Z$ with $w\leq x$ be periodic and let $\varphi(x)=\varphi(x_0,x_1)$. Then
\[
\mathcal{P}_{[w,x],\varphi}=\log \lambda_+ + \sum_{a,b\in\{0,1\}}\sum_{\ell\geq 2}m_{a,b}^\ell \log C_{a,b}^++\sum_{a,b\in\{0,1\}}\sum_{\ell\geq 2}m_{a,b}^\ell \log\left(1+\frac{C_{a,b}^-}{C_{a,b}^+}\left( \frac{\lambda_-}{\lambda_+}\right)^{\ell-1} \right),
\]
where $C_{a,b}^\pm$ for $a,b\in\{0,1\}$, $\ell\geq 2$ are constants from the assertion of Lemma~\ref{wzoreczek1} and $m_{a,b}^{\ell}$ come from~\eqref{wzornam} for $\rho$ corresponding to $(w,x)$.
\end{theorem}

Notice that if $w=\mathbf{0}$ then $m_{a,b}^\ell=0$ unless $a=b=0$. Therefore,
\[
\sum_{a,b\in\{0,1\}}\sum_{\ell\geq 2}m_{a,b}^\ell=\sum_{\ell\geq 2}m_{0,0}^\ell=\rho\left(\begin{matrix}0 \\ 0\end{matrix} \right)
\]
as the union of cylinders contributing to $\sum_{\ell\geq 1}m_{0,0}^\ell$ equals $\begin{matrix}0 \\ 0\end{matrix}$. Thus, the following result is immediate.
\begin{corollary}\label{OKRESOWE1}
Let $x\in\{0,1\}^\Z$ be periodic and let $\varphi(x)=\varphi(x_0,x_1)$. Then
\[
\mathcal{P}_{[\mathbf{0},x],\varphi}=\log \lambda_+ +\nu(0) \log C_{0,0}^++\sum_{\ell\geq 2}m_{0,0}^\ell \log\left(1+\frac{C_{0,0}^-}{C_{0,0}^+}\left( \frac{\lambda_-}{\lambda_+}\right)^{\ell-1} \right),
\]
where $C_{0,0}^\pm$ for $\ell\geq 2$ are constants from the assertion of Lemma~\ref{wzoreczek1} and $m_{0,0}^\ell=\nu(01\dots 10)$ (with $\ell-2$ ones), for $\nu$ corresponding to $x$.
\end{corollary}

\subsection{Approximation by sandwich subordinatesystems with periodic bounds}\label{APEB}
We will need the following basic lemma.
\begin{lemma}
Let $(X_K)_{K\geq 1}\subseteq \{0,1\}^\Z$ be a descending family of subshifts and suppose that $X$ is a subshift such that $X\subseteq\bigcap_{K\geq 1}X_K$ and $\mathcal{M}(X)=\mathcal{M}(\bigcap_{K\geq 1}X_K)$. Then
\[
\mathcal{P}_{X,\varphi}=\lim_{K\to\infty}\mathcal{P}_{X_K,\varphi}
\]
for any continuous potential $\varphi$.
\end{lemma}
\begin{proof}
By the variational principle, $\mathcal{P}_{X,\varphi}=\mathcal{P}_{\bigcap_{K\geq 1}X_K,\varphi}$. Thus, we can assume without loss of generality that $X=\bigcap_{K\geq 1}X_K$. Since $X\subseteq X_K$ for each $K\geq 1$, we clearly have
\(
\mathcal{P}_{X,\varphi}\leq \liminf_{K\to\infty}\mathcal{P}_{X_K,\varphi}.
\)
On the other hand, take for each $K\geq 1$ a measure $\mu_K\in \mathcal{P}(X_K)$ such that
\[
h(X_K,\mu_K)+\int_{X_K}\varphi\, d\mu_K \geq \mathcal{P}_{X_K,\varphi}-\frac1K.
\]
We may assume without loss of generality that $\mu_K\to \mu$ weakly. Clearly, the limit measure $\mu$ is supported on $\bigcap_{K\geq 1}X_K$. This yields
\[
\limsup_{K\to\infty} \mathcal{P}_{X_K,\varphi}\leq \limsup_{K\to\infty}\left(h(X_K,\mu_K)+\int_{X_K}\varphi\, d\mu_K+\frac{1}{K} \right)\leq h(X,\mu)+\int_{X}\varphi\, d\mu\leq \mathcal{P}_{X,\varphi}.
\]
\end{proof}
The above lemma can be, in particular, applied to $X_K=[w_K,x_K]$, where $(w_K,x_K)$ is good and its elements are all periodic. Indeed, by Proposition~\ref{biokre}, $\mathcal{M}(\overline{[w,x]})$ is a sandwich measure-theoretic subordinate subshift and by Proposition~\ref{III}, we have $\mathcal{M}(\overline{[w,x]})=\mathcal{M}(\bigcap_{K\geq 1}[w_K,x_K])$. Thus, as an immediate consequence of Theorem~\ref{OKRESOWE} and Corollary~\ref{OKRESOWE1}, we obtain the following.
\begin{corollary}\label{R1}
Let $(w_K,x_K)\in \{0,1\}^\Z\times \{0,1\}^\Z$ be a good sequence of periodic points and let $(w,x)$ be its coordinatewise limit. Suppose that $\varphi(x)=\varphi(x_0,x_1)$. Then
\begin{align*}
\mathcal{P}_{\overline{[w,x]},\varphi}&=\mathcal{P}_{\bigcap_{K\geq 1}{[w_K,x_K]},\varphi}\\
&=\log \lambda_+ + \sum_{a,b\in\{0,1\}}\sum_{\ell\geq 2}m_{a,b}^\ell \log C_{a,b}^++\sum_{a,b\in\{0,1\}}\sum_{\ell\geq 2}m_{a,b}^\ell \log\left(1+\frac{C_{a,b}^-}{C_{a,b}^+}\left( \frac{\lambda_-}{\lambda_+}\right)^{\ell-1} \right),
\end{align*}
where $C_{a,b}^\pm$ for $a,b\in\{0,1\}$, $\ell\geq 2$ are constants from the assertion of Lemma~\ref{wzoreczek1} and $m_{a,b}^{\ell}$ come from~\eqref{wzornam} for $\rho$ corresponding to $(w,x)$.
\end{corollary}
\begin{corollary}\label{R2}
Let $x_K\in \{0,1\}^Z$ be a good sequence of periodic points and let $x$ be its coordinatewise limit. Suppose that $\varphi(x)=\varphi(x_0,x_1)$ is such that $\det M\neq 0$. Then $\mathcal{P}_{\overline{[\mb{0},x]},\varphi}=\mathcal{P}_{\bigcap_{K\geq 1}\overline{[\mb{0},x_K]},\varphi}$ and
\begin{align*}
\mathcal{P}_{\overline{[\mb{0},x]},\varphi}=\log \lambda_+ +\nu(0) \log C_{0,0}^++\sum_{\ell\geq 2}m_{0,0}^\ell \log\left(1+\frac{C_{0,0}^-}{C_{0,0}^+}\left( \frac{\lambda_-}{\lambda_+}\right)^{\ell-1} \right),
\end{align*}
where $C_{0,0}^\ell$ for $\ell\geq 2$ are constants from the assertion of Lemma~\ref{wzoreczek1} and $m_{0,0}^\ell=\nu(01\dots 10)$ (with $\ell-2$ ones), for $\nu$ corresponding to $x$.
\end{corollary}

\subsection{$\mathscr{B}$-free systems}\label{apb}
\subsubsection{General results}
In this part we will apply the results from Section~\ref{APEB} to the $\mathscr{B}$-free setting. All the tools have already been prepared, so let us state and prove our results.
\begin{corollary}\label{W1}
Suppose that $\mathscr{B}\subseteq \mathbb{N}$ is such that $\eta^*$ is a regular Toeplitz sequence and $\varphi(x)=\varphi(x_0,x_1)$. Then
\begin{align*}
&\mathcal{P}_{{X}_\eta,\varphi}=\sum_{a,b\in\{0,1\},2\leq \ell\leq \min\mathscr{B}+1}m_{a,b}^\ell Z_{a,b}^\ell\\
=&\log \lambda_+ + \sum_{a,b\in\{0,1\}}\sum_{2\leq \ell\leq \min\mathscr{B}+1}m_{a,b}^\ell \log C_{a,b}^++\sum_{a,b\in\{0,1\}}\sum_{2\leq \ell\leq \min\mathscr{B}+1}m_{a,b}^\ell \log\left(1+\frac{C_{a,b}^-}{C_{a,b}^+}\left( \frac{\lambda_-}{\lambda_+}\right)^{\ell-1} \right),
\end{align*}
where $C_{a,b}^\ell$ for $a,b\in\{0,1\}$, $2\leq \ell\leq \min\mathscr{B}+1$ are constants from the assertion of Lemma~\ref{wzoreczek1} and $m_{a,b}^\ell$ are given by the same formula as in~\eqref{wzornam} with $\rho=\nu_{\eta^*}\triangle\nu_\eta$.
\end{corollary}
\begin{proof}
The assertion follows directly by~\eqref{isgood1},~\eqref{limitmeasure},~\eqref{eq:proksym} and by Corollary~\ref{R1}.
\end{proof}
\begin{corollary}\label{W11}
Let that $\mathscr{B}\subseteq \mathbb{N}$ and suppose that $\varphi(x)=\varphi(x_0,x_1)$. Then
\begin{align*}
\mathcal{P}_{\widetilde{X}_\eta,\varphi}&=\sum_{2\leq\ell\leq\min\mathscr{B}+1}m_{0,0}^\ell Z_{0,0}^\ell\\
&=\log \lambda_+ +
\nu_\eta(0) \log C_{0,0}^++\sum_{2\leq\ell\leq\min\mathscr{B}+1}\nu_\eta(0\underbrace{1\dots 1}_{\ell-2}0) \log\left(1+\frac{C_{0,0}^-}{C_{0,0}^+}\left( \frac{\lambda_-}{\lambda_+}\right)^{\ell-1} \right),
\end{align*}
where $C_{0,0}^\pm$ for $2\leq \ell\leq 1+\min\mathscr{B}$ are constants from the assertion of Lemma~\ref{wzoreczek1}.
\end{corollary}
\begin{proof}
The assertion follows directly by~\eqref{isgood},~\eqref{eq:qg},~\eqref{eq:gra} and by Corollary~\ref{R2}.
\end{proof}

\subsubsection{Special case: $2\in\mathscr{B}$}
Let us now compare our results with the formula obtained in~\cite{LC} in the Erd\"os case, under the additional assumption that $2\in\mathscr{B}$. Lin and Chen showed that in this case for $\varphi=a_{00}\mathbf{1}_{00}+a_{01}\mathbf{1}_{01}+a_1\mathbf{1}_1$ we get
\begin{equation}\label{chinczycy}
\mathcal{P}_{\widetilde{X}_\eta,\varphi}=a_{00}(1-2d)+d\log(2^{a_1+a_{01}}+2^{2a_{00}}),
\end{equation}
where $d=\nu_\eta(1)$. The first formula in Corollary~\ref{W11} in this special case yields the same:
\begin{align*}
\mathcal{P}_{\widetilde{X}_\eta,\varphi}&=\nu_\eta(00)\log Z_{0,0}^1+\nu_\eta(010)\log Z_{0,0}^2=\nu_\eta(00)\log Z_{0,0}^1+\nu_\eta(01)\log Z_{0,0}^2\\
&=a_{00}(1-2d)+d\log(2^{a_1+a_{01}}+2^{2a_{00}})
\end{align*}
as $\nu_\eta(00)=1-\nu_\eta(01)-\nu_\eta(10)-\nu_\eta(11)=1-2\nu_\eta(1)=1-2d$.

In fact, we can have more than that. Suppose that $2\in\mathscr{B}$, $\mathscr{B}$ is primitive and fix $\varphi\colon \widetilde{X}_\eta\to\mathbb{R}$. Then
\[
\mathcal{P}_{\widetilde{X}_\eta,\varphi}=\max\left\{\varphi(\mathbf{0}), \sup\left\{h(\nu,\sigma)+\int \varphi\, d\nu : \delta_{\mathbf{0}}\neq\nu\in\mathcal{M}^e(\widetilde{X}_\eta,\sigma)\right\}\right\}.
\]
Let $X:=\{x\in\widetilde{X}_\eta : x\neq \mathbf{0}\text{ and }x|_{2\mathbb{Z}}\equiv 0\}$. For any $\delta_{\mathbf{0}}\neq\nu\in\mathcal{M}^e(\widetilde{X}_\eta,\sigma)$, $\nu(X\cup \sigma X)=1$ and $X\cap\sigma X=\emptyset$. Moreover, $\sigma^2$ is the induced map corresponding to $\sigma$ and $X$. Let $\nu_X(A)=2\nu(A)$ for any measurable set $A\subseteq X$. Then $h(\sigma,\nu)=\frac{1}{2}h(\sigma^2,\nu_X)$ and $\int \varphi\, d\nu=\frac{1}{2}\int \varphi^{(2)}\, d\nu_X$. Moreover, we have the following measure-theoretic isomorphism:
\[
I\colon (X,\nu_X,\sigma^2)\to (\widetilde{X}_{\eta^{(2)}},I_\ast\nu_X,\sigma),
\]
where $\eta^{(2)}:=\mathbf{1}_{\mathcal{F}_{\mathscr{B}\setminus \{2\}}}$ and $I$ is given by $(I(x))_k=x_{2k+1}$ for $k\in\Z$ (clearly, $I$ is measurable and injective, its inverse is given by $(I^{-1}(x))_k=x_{(k-1)/2}$ for $k$ odd and $(I^{-1}(x))_k=0$ for $k$ even). Moreover,
\[
\{I_\ast \nu_X :\nu\in\mathcal{M}^e(X)\}=\mathcal{M}^e(\widetilde{X}_{\eta^{(2)}})\setminus \{\delta_{\mathbf{0}}\}.
\]\
It follows immediately that
\[
h(\nu_X,\sigma^2)=h(I_\ast \nu_X,\sigma)\text{ and }\int \varphi^{(2)}\, d\nu_X=\int \varphi^{(2)}\circ I^{-1}\, d I_\ast \nu_X.
\]
Thus,
\begin{equation}\label{nowe}
\mathcal{P}_{\widetilde{X}_\eta,\varphi}=\max\left\{\frac{1}{2}\varphi^{(2)}(\mathbf{0}),\frac{1}{2}\sup_{\nu\in\mathcal{M}^e(\widetilde{X}_{\eta^{(2)}})\setminus\{\delta_{\mathbf{0}}\}} \right\}=\frac{1}{2}\mathcal{P}_{\widetilde{X}_{\eta^{(2)}},\varphi^{(2)}\circ I^{-1}}.
\end{equation}

Now, if $\varphi(x)=\varphi(x_0,x_1)$ then $\varphi^{(2)}\circ I^{-1}(x)=\varphi^{(2)}(0,x_0)+\varphi(x_0,0)$. Since $\nu_{\eta^{(2)}}(1)=2\nu_\eta(1)$, we immediately recover again formula~\eqref{chinczycy}. Moreover, if $\varphi(x)=\varphi(x_0,x_1,x_2,x_3)$ then $\varphi^{(2)}\circ I^{-1}(x)=\varphi(0,x_0,0,x_1)+\varphi(x_0,0,x_1,0)$ and one can easily combine formula~\eqref{nowe} with Corrolary~\ref{W11} to compute~$\mathcal{P}_{\widetilde{X}_\eta,\varphi}$; the matrix corresponding to $\varphi^{(2)}\circ I^{-1}$ equals
\[
\left(\begin{matrix}2^{2\varphi(0,0,0,0)}& 2^{\varphi(1,0,0,0)+\varphi(0,1,0,0)}\\ 2^{\varphi(0,0,1,0)+\varphi(0,0,0,1)} & 2^{\varphi(1,0,1,0)+\varphi(0,1,0,1)}\end{matrix} \right).
\]

\subsubsection{Asymptotics}
A natural question arises, what happens with the second and third term in formulas from Corollary~\ref{W1} and Corollary~\ref{W11} when $X_\eta$ or $\widetilde{X}_\eta$ more and more ``resembles'' $\{0,1\}^\Z$. Let us concentrate on $\widetilde{X}_\eta$ as this subshift is easier to deal with (and, as will turn out, we will be able to cover the Erd\"os case only when the two subshifts coincide anyway). 

We would like to measure ``how close'' is $\widetilde{X}_\eta$ to $\{0,1\}^\Z$ in terms of the topological pressure. We claim that the correct way to ``measure'' this is to look at the value of $\nu_\eta(1)$. If $\nu_{\eta_K}(1)\to 1$ when $K\to \infty$ this implies that $\min \mathscr{B}_K \to \infty$, so, in particular 
\begin{equation}\label{jezyk}
\mathcal{L}_n(\widetilde{X}_{\eta_K})=\{0,1\}^n
\end{equation}
for any fixed $n$, whenever $K$ is large enough. Notice that condition~\eqref{jezyk} itself is not sufficient to obtain any meaningful results. Consider $\mathscr{B}^{(K)}=\{p\in\mathbb{P} : p\geq p_K\}$. Then~\eqref{jezyk} holds as $\mathscr{B}^{(K)} \to \infty$. However, the set of invariant measures living on $\widetilde{X}_{\eta_K}$ for $\eta^{(K)}:=\mathbf{1}_{\mathscr{B}^{(K)}}$ is the singleton $\{\delta_{\boldsymbol{0}}\}$ and the variational principle tells us that $\widetilde{X}_{\eta_K}$ differs substantially from $\{0,1\}^\Z$! 

Let us now see how the notion of the topological pressure can be used to quantify the resemblance of $\widetilde{X}_\eta=X_\eta$ to $\{0,1\}^\Z$ in the Erd\"os case.

\begin{theorem}\label{tempo}
For any $\varepsilon\in (0,2)$, within the family of Erd\"os sets $\mathscr{B}\subseteq \mathbb{N}$,
the third ingredient in formula for $\mathcal{P}_{{X}_\eta,\varphi}$ from Corollary~\ref{W1} is of order $\rm{o}((1-d)^\varepsilon)$  as $d\to 1$:
\[
\sum_{2\leq \ell\leq\min \mathscr{B}+1}\nu_\eta(0\underbrace{1\dots 1}_{\ell-2} 0)\log\left(1+\frac{C_{0,0}^-}{C_{0,0}^+}\left(\frac{\lambda_-}{\lambda_+}\right)^\ell \right) \ll (1-d)^\varepsilon.
\]
\end{theorem}
Before we give the proof, let us prepare some tools. Recall that
\[
d=\nu_{\eta}(1)=\prod_{b\in\mathscr{B}}\left(1-\frac{1}{b} \right)
\]
and notice that
\begin{equation}\label{eq:szacmiar}
\nu_{\eta}(0\underbrace{1\dots 1}_{\ell-2} 0)\leq \nu_{\eta}(0\underbrace{*\dots *}_{\ell-2} 0)\leq\sum_{b,b'\in\mathscr{B}}\frac{1}{b \cdot b'}=S^2,\text{ where }S=\sum_{b\in\mathscr{B}}\frac{1}{b}.
\end{equation}
\begin{lemma}\label{nowylemat}
Let $d_0\in (0,1)$. Then there exists $C>0$ such that for any $\mathscr{B}\subseteq\mathbb{N}$ that is Erd\"os, with $d\in [d_0,1]$, we have
\[
1-d \leq S\leq C (1-d).
\]
\end{lemma}
\begin{proof}
We have
\[
d=\prod_{b\in\mathscr{B}}\left(1-\frac{1}{b}\right)=\lim_{K\to\infty}\prod_{b\in\mathscr{B}_K}\left(1-\frac{1}{b} \right)\geq \lim_{K\to\infty}\left(1-\sum_{b\in\mathscr{B}_K}\frac{1}{b}\right)=1-\sum_{b\in\mathscr{B}}\frac{1}{b}=1-S.
\]

To explain the second inequality from the assertion, we will use the inequality of the arithmetic and geometric mean. For all $K\in\N$, we have
\[
\prod_{b\in \mathscr{B}_K}\left(1-\frac{1}{b} \right)\geq d, 
\]
whence, for $n_K=\# \mathscr{B}_K$,
\[
d^{1/n_K}\leq \sqrt[n_K]{\prod_{b\in\mathscr{B}_K}\left(1-\frac{1}{b} \right)}\leq \frac{1}{n_K}\sum_{b\in\mathscr{B}_K} \left(1-\frac{1}{b} \right)=1-\frac{1}{n_K}S.
\]
This yields
\[
S\leq n_K-n_Kd^{1/n_K}.
\]
Notice that 
\[
\lim_{n\to\infty}(n-n\sqrt[n]d)=\lim_{n\to\infty}\frac{1-d^{\nicefrac1n}}{\nicefrac1n}=\lim_{h\to 0}\frac{1-d^h}{h}=\lim_{h\to 0}(-\log d \cdot d^h)=-\log d.
\]
We are looking for $C>0$ such that for $d\in [d_0,1]$, $-\log d\leq C(1-d)$. However, for $f(x)=-\frac{\ln x}{1-x}$, limit from the left at $1$ equals $1$), $f$ is continuous, so such $C>0$ indeed exists.  
\end{proof}
\begin{proof}[Proof of Theorem~\ref{tempo}]
Fix $\varepsilon \in (0,2)$ and let $\ell_0\simeq (1-d)^{-2+\varepsilon}$ (e.g.\ $\ell_0=\lfloor (1-d)^{-2+\varepsilon} \rfloor$). We split the sum under consideration into two parts, in the first one, $\ell$ runs from $2$ to $\ell_0$, in the second one -- from $\ell_0+1$ to $\min \mathscr{B}+1$. We will estimate each of them separately. To shorten the notation, we put $c:=C_{0,0}^-/C_{0,0}^+$.

We have
\begin{align}
\begin{split}\label{eq:numer}
&\left|\sum_{\ell=2}^{\ell_0}\nu_\eta(0 \underbrace{1\dots 1}_{\ell-2} 0)\log\left(1+c\left(\frac{\lambda^-}{\lambda^+} \right)^\ell \right) \right|\\
&\ \ \ \ \ \leq \sum_{\ell=2}^{\ell_0}\nu_\eta(0 \underbrace{1\dots 1}_{\ell-2} 0)\left|\log\left(1+c\left(\frac{\lambda^-}{\lambda^+} \right)^\ell\right) \right|
\leq (1-d)^{-2+\varepsilon}\cdot S^2 \cdot \text{const}.
\end{split}
\end{align} 
To justify this, recall~\eqref{eq:szacmiar} and let us explain how to find the constant appearing in the last line in the above formula. Consider
\[
\left|\log\left(1+c\left(\frac{\lambda^-}{\lambda^+} \right)^\ell\right) \right|\cdot \left|\log\left(1+c\left(\frac{\lambda_-}{\lambda_+} \right)^{\ell+1}\right) \right|^{-1}.
\]
If $\ell$ is large (in terms of $c$ and $\nicefrac{\lambda_-}{\lambda_+}$) then the above quantity is of the same order as
\[
\left| c\left(\frac{\lambda_-}{\lambda_+} \right)^\ell\right|\cdot \left| c\left(\frac{\lambda_-}{\lambda_+} \right)^{\ell+1}\right|^{-1}=\left|\frac{\lambda_+}{\lambda_-}\right|>1.
\]
Thus, 
\[
\max_{2\leq \ell\leq \ell_0}\left|\log\left(1+c\left(\frac{\lambda_-}{\lambda_+} \right)^\ell\right) \right|=\left|\log\left(1+c\left(\frac{\lambda_-}{\lambda_+} \right)^k\right) \right|
\]
for some $k$ bounded from above by a constant that depends on $c$ and $\nicefrac{\lambda_-}{\lambda_+}$ only and~\eqref{eq:numer} indeed follows. It remains to use Lemma~\ref{nowylemat} to conclude that
\[
\left|\sum_{\ell=2}^{\ell_0}\nu_\eta(B_\ell 0)\log\left(1+c\left(\frac{\lambda_-}{\lambda_+} \right)^\ell \right) \right| \leq \text{const}\cdot (1-d)^\varepsilon.
\]

Now, notice that
\[
\left|\sum_{\ell=\ell_0+1}^{\min\mathscr{B}+1}\nu_\eta(B_\ell 0)\log\left(1+c\left(\frac{\lambda_-}{\lambda_+} \right)^\ell \right) \right|\leq \sum_{\ell=\ell_0+1}^{\min\mathscr{B}+1}\nu_\eta(B_\ell 0) \max_{\ell_0+1\leq \ell\leq \min\mathscr{B}+1}\left|\log\left(1+c\left(\frac{\lambda_-}{\lambda_+} \right)^\ell \right)\right|,
\]
where
\[
\max_{\ell_0+1\leq \ell\leq \min\mathscr{B}+1}\left|\log\left(1+c\left(\frac{\lambda_-}{\lambda_+} \right)^\ell \right)\right|\leq\left|\log\left(1+c\left(\frac{\lambda_-}{\lambda_+} \right)^{\ell_0} \right)\right|,
\]
provided that $\ell_0$ is large enough (i.e.,\ $1-d$ is small enough). Using again that $\log(1+x)\simeq x$ for small values of $x$, it follows that for $\ell_0$ large enough, 
\[
\left|\log\left(1+c\left(\frac{\lambda_-}{\lambda_+} \right)^{\ell_0} \right)\right|\simeq \left( \frac{\lambda_-}{\lambda_+}\right)^{\ell_0}\ll P(1-d),
\]
for any polynomial $P$. To see that the last relation in the formula above indeed holds, set $\delta:=\nicefrac{\lambda_-}{\lambda_+}$. It suffices to show that
\[
\delta^{\ell_0}\ll (1-d)^{(2-\varepsilon)k} \text{ for any }k\geq 1,
\]
which is equivalent to
\[
\delta^{\nicefrac{1}{(1-d)^{2-\varepsilon}}}\ll (1-d)^{(2-\varepsilon)k} \text{ as }d\to 1.
\]
Substituting $x:=\nicefrac{1}{(1-d)^{2-\varepsilon}}$, we need to show 
\[
\delta^x \ll \frac{1}{x^k} \text{ for }k\geq 1\text{ as }x\to\infty.
\]
This is, however, true by de l'Hospital's rule:
\[
\frac{\left(\nicefrac{1}{\delta}\right)^x}{x^k}\simeq \left(\log\nicefrac{1}{\delta} \right)^k\cdot\left(\nicefrac{1}{\delta} \right)^x\to \infty\text{ as }x\to\infty.
\]
\end{proof}

\appendix
\section{Probabilistic approach}

\subsection{Basic notions}
In this section we translate our setup to the probabilistic language. 

\paragraph{Random variables, processes}
By a \emph{random variable} $X$ we will mean any measurable function taking values in a measurable \emph{state space} $\mathcal{A}$; we will write then $X\in\mathcal{A}$. Usually, $\mathcal{A}\subseteq \mathbb{R}$, but we will also need random variables taking values in sequence spaces. For simplicity, we will assume that all random variables are defined on a common probability space $(\Omega, \mathcal{F}, \mathbb{P})$. A discrete stochastic \emph{process} $\mathbf{X} = (X_i)_{i \in \Z}$ is a family of a random variables $X_i$ taking values in a common state space $\mathcal{A}$.  Usually, $|\mathcal{A}|<\infty$, but we will also need countably-valued processes as a tool. Recall that a process $\mathbf{X}$ is called \emph{stationary} if the distribution of any $k$-tuple $(X_{i},X_{i+1},\dots,X_{i+k-1})$ does not depend on $i\in\mathbf{Z}$. Recall also that given a subshift $Z$, each choice of $\kappa\in \mathcal{M}(Z)$ determines a stationary process $\mathbf{X}$ (with distribution $\kappa$; we write $\mathbf{X}\sim \kappa$) and vice versa. We will often identify processes with their distributions and write, e.g., $\mathbf{X}\in\mathcal{M}(Z)$.

\paragraph{measure-theoretic subordinate subshifts}
Fix a stationary process $\mb{X}\in \mathcal{M}(\{0,1\}^\Z,\sigma)$ and let
        \begin{equation}\label{defiN}
           \mathcal{N}_{\mb{X}}
            = \left\{\mb{X}'\cdot \mb{Y}: (\mb{X}',\mb{Y})\in \mathcal{M}(\{0,1\}^\Z\times \{0,1\}^\Z,\sigma\times \sigma) \text{ and }\mb{X}\sim \mb{X}'\right\}.
        \end{equation}
        Since in this paper the processes are, in fact, treated as measures (we only care about the distribution of $\mb{X}$ and not its particular realization), it makes sense to write
        \begin{equation}\label{defiN}
 \mc{N}_{\mb{X}} = \left\{\mb{X}\cdot \mb{Y} : (\mb{X},\mb{Y}) \in \mc{M}({\{0,1\}^\Z\times \{0,1\}^\Z},\sigma\times \sigma) \right\}.
\end{equation}
If $Z\subseteq \{0,1\}^\Z$ is a subshift and $\mb{X}\in\mathcal{M}(Z)$ then one can view $\mathcal{N}_{\mb{X}}$ as a subset of $\mathcal{M}(\widetilde{Z})$. Moreover, if $\mathcal{M}(Z)=\mathcal{N}_{\mb{X}}$ for some $\mb{X}$, it means precisely that $Z$ is a measure-theoretic subordinate subshift. 

\paragraph{Shannon's entropy}
        Here, unless stated directly otherwise, all random variables for us will be discrete (countably valued).
        Recall that if $X \in \mc{X}$, $Y \in \mc{Y}$ and a set $A$ is such that $\pof[A] > 0$ then we define the \emph{Shannon entropy} and the \emph{Shannon conditional entropy} via
        \begin{align*}
            & \h[X][A] = -\sum_{x\in \mc{X}} \pof[X = x][A] \log \pof[X = x][A], \\
            & \ch[X][Y][A] = \sum_{y \in \mc{Y}}\pof[Y = y][A] \h[X][Y = y, A],
        \end{align*}
        respectively, where $\P_A(\cdot)=\pof[\cdot \ | \ A]$. If $\pof[A] = 1$ then we simply write $\h[X]$ and $\ch[X][Y]$.
                
        We will need the extension of the definition of $\ch[X][Y]$ to $\ch[X][\mc{G}]$, where $\mc{G}$ is a sub-$\sigma$-algebra of $\mc{F}$. There are several (equivalent) ways to do this, we choose to work with regular conditional distributions. Fix a random variable $X\in \mc{X}$ and a sub-$\sigma$-algebra $\G\subseteq \F$. We say that $p_{X\;|\;\G}(\cdot\;|\;\cdot)\colon \mathcal{B}(\mc{X})\times \Omega \to [0,1]$ is a \emph{regular conditional distribution} of $X$ given $\G$, if {for every $\omega \in \Omega$, $p_{X\;|\;\G}(\omega, \cdot)$ is a probability measure, for every $A \in  \F$, $p_{X\;|\;\G}(A, \cdot)$ is measurable and a.s.}
\[
p_{X\;|\; \G}(A\;|\;\omega)=\pof[X \in A\;|\; \G](w).
\]        

Since in this paper we consider processes corresponding to invariant measures for subshifts, the underlying space $(\Omega,\F,\mathbb{P})$ is standard and thus a regular conditional distribution of $X$ given $\G$ always exists (see, e.g., Lemma 5.8.1 in \cite{GRAY1}). We can also speak of a \emph{regular conditional distribution} of a random variable $X$ given another random variable $Y$, by considering $p_{X \;|\;Y}:=p_{X \;|\; \sigma(Y)}$, where $\sigma(Y)$ stands for the smallest sub-$\sigma$-algebra of $\mc{F}$ that makes $Y$ measurable. Note that here $Y$ can take values in an arbitrary measurable space, so, in particular, we can apply it to processes.

Let now $X\in\mc{X}$ be discrete and fix a sub-$\sigma$-algebra $\mathcal{G}$ of $\mc{F}$. We set
        \begin{equation*}
           \ch[X][\mc{G}] :=  \E \h[p_{X\;|\;\mathcal{G}}(\cdot\;|\; \omega)] = \int \h[p_{X\;|\;\mathcal{G}}(\cdot\;|\; \omega)] d\Pro(\omega).
        \end{equation*}
Moreover, for any random variable $Y$, we put
\begin{equation}\label{WZZZ}
\ch[X][Y]:=\ch[X][\sigma(Y)]=\E\h[p_{X\;|\;Y}(\cdot\;|\;y)]_{|_{y = Y}}.
\end{equation}
We will also use the following, slightly informal, notation: 
\[
\h[X][Y = y]:=\h[p_{X\;|\;Y}(\cdot\;|\; y)].
\]
Then formula~\eqref{WZZZ} becomes
\begin{equation}\label{niewarunkowa}
\ch[X][Y]=\E\h[X][Y = y]_{|{y = Y}}.
\end{equation}
(Note that this notation is consistent with the case when $Y$ takes finitely many values.)
Finally, notice that directly from the definition, if $\sigma(Y)=\sigma(Z)$ then $\ch[X][Y]=\ch[X][Z]$. Here $X$ is again discrete and $Y,Z$ take values in an arbitrary measurable space and the equality holds, whenever $Z$ is the image of $Y$ via a bimeasurable map between the respective measurable spaces.

We will need also a conditional version of~\eqref{niewarunkowa}. Before we formulate it, let us prove the following lemma. 
\begin{lemma}
Suppose that $(X,Y,Z)$ is a triple of random variables, taking values in arbitrary measurable spaces. Suppose that $p_{(X,Y) \;|\; Z}$ exists and let $(X^{(y)},Z^{(y)})\sim p_{(X,Z)\; | \; Y}(\cdot\;|\; y)$. Then
\begin{equation}\label{pedeefik}
p_{X^{(y)}\; | \; Z^{(y)}}(\cdot\;|\;z)=p_{X|(Y,Z)}(\cdot\;|\; (y,z)).\footnote{By this, we are saying that the left-hand side exists and we have the equality.}
\end{equation}
\end{lemma}
\begin{proof}
For any measurable sets $A,B,C$, we have
\begin{align*}
&\E_{Y,Z}p_{X^{(Y)}\;|\; Z^{(Y)}}(A\;|\;Z)\iof[Y\in B,Z\in C]=\E\iof[Y\in B]\E\left(\iof[Z\in C]p_{X^{(Y)}\;|\; Z^{(Y)}}(A \;|\;Z)\;|\;Y \right)\\
&=\E\iof[Y\in B]\left[\E\left[\iof[Z^{(y)}\in C]p_{X^{(y)}\;|\; Z^{(y)}} (A\;|\; Z^{(y)})\right]_{|_{y=Y}}\right]\\
&=\E\iof[Y\in B]\left[\E \iof[Z^{(y)}\in C]\E\left(X^{(y)}\in A\;|\; Z^{(y)} \right) \right]_{|_{y=Y}}\\
&=\E\iof[Y\in B]\left[\mathbb{P}\left(X^{(y)}\in A,Z^{(y)}\in C \right) \right]_{|_{y=Y}}=\E\iof[Y\in B]\E (\iof[X\in A,Z\in C]\;|\; Y)\\
&=\mathbb{P}(X\in A,Z\in C\;|\; Y\in B)=\E \iof[Y\in B]\iof[Z\in C]p_{X\;|\; (Y,Z)}(A\;|\; Y,Z).
\end{align*}
\end{proof}
It follows from the above lemma that if we set
\begin{equation}
	\mb{H}_{Y=y}(X\;|\;Z):=\ch[X^{(y)}][Z^{(y)}]
\end{equation}
then
\begin{align}
	\begin{split}\label{eins}
		\ch[X][Y,Z]&=\E \h[p_{X\;|\; (Y,Z)}(\cdot \;|\; (y,z))]_{|_{y=Y,z=Z}}\\
		&=\E\mb{H}(p_{X^{(y)}\;|\; Z^{(y)}}(\cdot\;|\; z))_{|_{y=Y,z=Z}} = \E\cmv[\mb{H}(p_{X^{(Y)}\;|\; Z^{(Y)}}(\cdot\;|\; z))_{|_{z=Z}}][Y] \\
		&= \E\cmv[\mb{H}(p_{X^{(Y)}\;|\; Z^{(Y)}}(\cdot\;|\; z))_{|_{z=Z^{(Y)},}}][Y] = \E\cmv[	\mb{H}_{Y=y}(X\;|\;Z)][Y] 
	\end{split}
	\end{align}

\paragraph{Entropy rate}
    If $\mb{X}$ is a process and $A\subseteq \Z$ then by $X_A$ we mean a {sequence} of random variables $(X_a)_{a\in A}$. We will often use this notation for integer intervals, denoted in the same way as the usual intervals in $\R$ (e.g.\ $[k,\ell)=\{i\in \Z : k\leq i<\ell\}$). Fix stationary processes $\mb{X} = \proc[X][\Z]$  and $\mb{Y} = \proc[Y][\Z]$ with at most countable alphabets such that $\h[X_0]$, $\h[Y_0] < \infty$ and recall that in such a case \emph{entropy rate} of $\mb{X}$ is given by
        \begin{align*}
           \h[\mathbf{X}] & = \lim\limits_{n\to\infty}\frac{1}{n}\h[X_{[0, n-1]}] = \inf_{n\in \N}\frac{1}{n}\h[X_{[0, n-1]}] \\
                          & = \lim\limits_{n\to\infty}\ch[X_n][X_{[0, n-1]}] = \ch[X_0][X_{(-\infty, -1]}]
        \end{align*}
        and, if additionally $(\mb{X}, \mb{Y}) = \left((X_i, Y_i)\right)_{i \in \Z}$ is stationary, \emph{the relative entropy rate} of $\mb{X}$ with respect to $\mb{Y}$ is given by
        \begin{align*}
           \ch[\mb{X}][\mb{Y}] & = \lim\limits_{n \to \infty} \frac{1}{n} \ch[X_{[0, n - 1]}][Y_{[0, n - 1]}] = \inf_{n \in \N} \frac{1}{n} \ch[X_{[0, n - 1]}][Y_{[0, n - 1]}] \\
                               & = \ch[X_0][X_{\oci{-\infty, -1}}, \mb{Y}].
        \end{align*}
        This notation is consistent with the one for the measure-theoretic entropy: if $\mb{X}$ is a stationary process over some finite alphabet, with distribution $\mu$ then $\ph[X]=h(\mu)$.

\paragraph{Topological pressure}\label{vps}
Let $Z\subseteq \{0,1\}^\Z$ be a subshift. The classical variational principle for the topological pressure~\eqref{VP1} in the probabilistic language can be stated as follows:
\[
           \mathcal{P}_{Z, \varphi} = \sup_{\mb{X}\in \mc{M}(Z)}\left[\h[\mb{X}] + \E \varphi(\mb{X})\right] .
\]
        Now, for any subset of invariant measures $\mc{N}\subseteq \mc{M}(\{0,1\}^\Z)$, we define
        \begin{equation*}\label{var_pri}
           V_{\mc{N}, \varphi} = \sup_{\mb{X} \in \mc{N}}\left[\h[\mb{X}] + \E \varphi(\mb{X})\right].
        \end{equation*}
        In particular, we will be interested in 
        \[
        V_{\mc{N}_\mb{X}, \varphi} =  \sup_{\mb{X}\cdot \mb{Y} \in \mc{N}_\mb{X}}\left[\h[\mb{X}\cdot \mb{Y}] + \E \varphi(\mb{X}\cdot \mb{Y})\right]
        \]
        since for a measure-theoretic subordinate subshift $Z$ (by the very definition) $\mc{N}_\mb{X}=\mc{M}(Z)$ and thus $V_{\mc{N}_\mb{X}, \varphi}=\mathcal{P}_{Z, \varphi}$ (in particular, for $\varphi\equiv 0$ we get $V_{\mc{N}_\mb{X},0}=h(Z)$).

\subsection{Entropy rate of a product}\label{entconv}
In this section, we will assume that
\begin{equation}\label{standing1}
\mb{X} \text{ is a stationary process with }X_i\in \{0,1\} \text{ and }\PP(X_0=1)>0
\end{equation}
(the assumption $\Pro(X_0=1)>0$ is made to avoid a degenerate situation, cf.\ \eqref{degen}). Whenever a pair of processes $(\mb{X},\mb{Y})$ is considered, we will assume that it is (jointly) stationary and $\mb{Y}$ is a real finitely-valued process.

Given $\mb{X}$, let $\mb{R}=\mb{R}(\mb{X})=\proc[R]$ be the process of consecutive \emph{arrival times} of~$\mb{X}$ to the state 1:
        \begin{equation}\label{procesR}
           R_i=\begin{cases}
              \inf\{j \ge 0 : X_j = 1\},       & i=0,     \\
              \inf\{j > R_{i - 1} : X_j = 1\}, & i\ge 1,  \\
              \sup\{j < R_{i + 1} : X_j = 1\}, & i\le -1.
           \end{cases}
        \end{equation}
        Process $\mb{R}$ is well-defined by the Poincar\'e recurrence theorem applied to $\mb{X}$ (for a.e.\ $\mb{x}$ such that $x_0=1$, both the negative and positive part of the support of $x$ is infinite).      
        
		The key ingredient in our method of computing the topological pressure is the following explicit formula for the entropy rate of a product.
        \begin{proposition}\label{theo q1}
Suppose that $(\mb{X},\mb{Y})$ is a finitely-valued stationary process, with $Y_i\in\R$, $X_i\in \{0,1\}$ and $\Pro(X_0=1)>0$. Then
           \begin{equation}\label{mc}
              \ch[\mb{X}\cdot\mb{Y}][\mb{X}]= \PP(X_0 = 1)\ch[Y_0][{Y}_{\mb{R}_{(-\infty, -1]}}, \mb{X}][X_0 = 1].
           \end{equation}
           In particular, if $\mb{X} \amalg \mb{Y}$ (that is, $\mb{X}$ and $\mb{Y}$ are independent) then~\eqref{mc} reduces to
           \begin{equation}\label{mc independent}
              \ch[\mb{X}\cdot\mb{Y}][\mb{X}]= \pof[X_0 = 1]\E_{X_0=1}\left[\mb{H}(Y_0\;|\; Y_{\mb{r}_{(-\infty,-1]}}) \right]_{|_{\mb{x}=\mb{X}}}.
           \end{equation}
        \end{proposition}
\begin{remark}
        Formula~\eqref{mc independent} was proved by us already in~\cite{MR4530196,doktorat}. In order to compute the integral from the right hand side of \eqref{mc independent} one must take the following steps: for almost $\P_{X_0 = 1}$ every realization of our return process $\mb{R}$ we compute $\E_{X_{0} = 1} \ch[Y_0][Y_{\{r_{-1}, r_{-2}, \cdots\}}]$, thus obtaining some function $f(\mb{r}_{\oci{-\infty, -1}})$; then, we find $\E_{X_0 = 1} f(\mb{R}_{\oci{-\infty, -1}})$. 
\end{remark}

\begin{proof}[Proof of Proposition~\ref{theo q1}]
For $\mb{M} = \mb{X} \cdot \mb{Y}$ we have
        \begin{equation*}
        \begin{split}
           \cph[M][X] &=  \ch[X_0 Y_0][\mb{M}_{\oci{-\infty, -1}}, \mb{X}]
           =\ch[X_0 Y_0][\mb{M}_{\oci{-\infty, -1}}, X_0, \mb{X}_{\neq 0}]\\
            &= \pof[X_0 = 1]\ch[Y_0][\mb{M}_{\oci{-\infty, -1}}, \mb{X}][X_0 = 1].
        \end{split}
        \end{equation*}
Now, under $X_0=1$, $(\mb{M}_{(-\infty,-1]},\mb{X})$ carries exactly the same information as $(\mb{Y}_{\mb{R}_{(-\infty,-1]}},\mb{X})$. Therefore, $\sigma(\mb{M}_{(-\infty,-1]},\mb{X})=\sigma(\mb{Y}_{\mb{R}_{(-\infty,-1]}},\mb{X})$ and thus
        \begin{equation*}
           \cph[M][X]    = \pof[X_0 = 1]\ch[Y_0][\mb{Y}_{\mb{R}_{\oci{-\infty, -1}}}, \mb{X}][X_0 = 1].
        \end{equation*}  
To complete the proof, it suffices to notice that
\begin{multline*}
\ch[Y_0][\mb{Y}_{\mb{R}_{\oci{-\infty, -1}}}, \mb{X}][X_0 = 1]=\E_{X_0=1}\left[\mb{H}_{\mb{X}=\mb{x}}(Y_0\;|\; \mb{Y}_{\mb{R}_{(-\infty,-1]}}) \right]_{|_{\mb{x}=\mb{X}}}\\
=\E_{X_0=1}\left[\mb{H}_{\mb{X}=\mb{x}}(Y_0\;|\; \mb{Y}_{\mb{r}_{(-\infty,-1]}}) \right]_{|_{\mb{x}=\mb{X}}}=\E_{X_0=1}\left[\mb{H}(Y_0\;|\; \mb{Y}_{\mb{r}_{(-\infty,-1]}}) \right]_{|_{\mb{x}=\mb{X}}},
\end{multline*}
where the last equality follows by the independence of $\mb{X}$ and $\mb{Y}$ (for $X$ and $Y$ independent and any measurable $A$, we have $p_{X\;|\; Y}(A\;|\; \cdot\;)=\mathbb{P}(X\in A\;|\; Y)=\mathbb{P}(X\in A)$).
\end{proof}

        Recall that if $\h[\mb{X}]=0$ then $\ch[\mb{X}\cdot \mb{Y}][\mb{X}]=\h[\mb{X}\cdot\mb{Y}]$. Therefore, the above theorem gives formulas for $\ph[X\cdot Y]$, as soon as $\ph[X] =0$.

\subsection{Probabilistic proof of Theorem~\ref{int0}}\label{potone}
Let $X\subseteq \{0,1\}^\Z$ be a measure-theoretic subordinate subshift and suppose that $\varphi\colon X\to\mathbb{R}$ depends only on one coordinate, i.e.\ $\varphi(\mb{x})=\varphi(x_0)$. Recall that $\mb{G}$ is called the \emph{Gibbs measure} associated with $\varphi$ if $\mb{G}$ is an i.i.d.\ process such that $\pof[G_i = x]$ is proportional to $2^{\varphi(x)}$:
        \[
           \pof[G_i = x] = 2^{\varphi(x)}/ \sum_{y\in\{0,1\}}2^{\varphi(y)}\text{ for }x\in \{0,1\}^\Z.
        \]
(In other words, $\mathbf{G}$ corresponds to the Bernoulli measure $B(p,1-p)$ on $\{0,1\}^\Z$ for $p=2^{\varphi(0)}/ \sum_{y\in\{0,1\}}2^{\varphi(y)}$.)
It is well-known (see, e.g., \cite{MR390180}) that for the full shift $\{0,1\}^\Z$ and $\varphi$ as above, 
\begin{equation}\label{rema gibbs suprema}
\mathscr{P}_{\{0,1\}^\Z,\varphi}=\log \sum_{x\in\{0,1\}}2^{\varphi(x)}
\end{equation}
and $\mb{G}$ is the unique equilibrium measure. In fact, we have
       \begin{align}
       \begin{split}\label{toobliczenie}
              \mathscr{P}_{\{0,1\}^\Z, \varphi} &= \sup_{\mb{X} \in \mc{P}(\{0,1\}^\Z)}\left[\h[\mb{X}] + \E \varphi(X_0)\right]\\
              & \le \sup_{X_0 \in \{0,1\}} \left[\h[X_0] + \E \varphi(X_0)\right] = \log \sum_{x\in\{0,1\}}2^{\varphi(x)}
           \end{split}
           \end{align}
and $\mb{G}$ is the unique measure for which the inequality above becomes an equality (this follows easily by Jensen's inequality, see for example the calculation below equation $(3)$ in~\cite{MR3220762}). It turns out that this result can be extended in the following way:        
      \begin{theorem}\label{theo pot}
        Let $\mb{X}$ be a stationary process, with $X_i\in\{0,1\}$ and $\Pro(X_0=1)>0$. Assume additionally that $\mb{H}(\mb{X})=0$ and that $\varphi(\mb{y}) = \varphi(y_0)$. Then
           \begin{equation}\label{topo pres for 1loc potential}
              V_{\mc{N}_\mb{X}, \varphi} = \sup_{\mb{X}\cdot \mb{Y} \in \mc{N}_\mb{X}} \left[\h[\mb{X}\cdot\mb{Y}] + \E \varphi(\mb{X}\cdot \mb{Y})\right]= (1 - d)\varphi(0) + d\lo[\sum_{x \in \{0,1\}}2^{\varphi(x)}][],
           \end{equation}
           where $d = \pof[X_0 = 1]$. Furthermore, there is a unique equilibrium measure for $\varphi$, given by $\mb{X}\cdot \mb{G}$, where $\mb{X} \amalg \mb{G}$ and $\mb{G}$ is the {Gibbs measure} associated with $\varphi$.
        \end{theorem}
        The assertion of Theorem~\ref{int0} follows directly from Theorem~\ref{theo pot}.
        \begin{remark}\label{tauw}
           Notice that by taking $\mb{Y}  = \mb{1}$ (that is, $d = 1$), we obtain~\eqref{rema gibbs suprema}. However, our proof relies on this special case -- more specifically, it uses~\eqref{toobliczenie}. 
        \end{remark}
\begin{remark}
For the proof of Theorem~\ref{theo pot}, We will need a standard argument on conditional mean values. Let $\mc{G}\subseteq \mc{F}$ be a sub-$\sigma$-field. If $Z\amalg \mc{G}$ and $W\in \mc{G}$ then
\begin{equation}\label{wzoreczek}
\E(g(Z,W) | \mc{G})=G(W),
\end{equation}
where $G(w)=\E g(Z,w)$.

\end{remark}

        \begin{proof}[Proof of Theorem~\ref{theo pot}]
Take $\mb{M}:=\mb{X}\cdot\mb{Y}\in\mc{N}_\mb{X}$. By Theorem~\ref{theo q1}, we have
                          \begin{multline*}
                             \ch[\mb{M}][\mb{X}]=\ch[\mb{X}\cdot\mb{Y}][\mb{X}]= \pof[X_0 = 1] \ch[Y_0][\mb{Y}_{\mb{R}_{\oci{-\infty, -1}}}, \mb{X}][X_{0} = 1] \\
                             \le \pof[X_0 = 1] \h[Y_0][X_{0} = 1].
                          \end{multline*}
                          Thus, using~\eqref{toobliczenie},
                                     \begin{align*}
              \cph[\mb{M}][\mb{X}] &+ \E \varphi(\mb{M})\\ & \le  \pof[X_0 = 1] \h[Y_0][X_{0} = 1] + \pof[X_0 = 1]\E_{X_0 = 1} \varphi(Y_0) + \pof[X_0 = 0] \varphi(0) \\
                                                               & =\pof[X_0 = 0] \varphi(0) + \pof[X_0 = 1] \left[\h[Y_0][X_{0} = 1]  + \E_{X_0 = 1} \varphi(Y_0)\right]    \\
                                                               & \le \pof[X_0 = 0] \varphi(0) + \pof[X_0 = 1] \log \left(\sum_{y \in \{0,1\}} 2^{\varphi(y)}\right).
           \end{align*}
                          Note that these inequalities become equalities iff (conditionally on $X_0 = 1$)
                          \begin{gather}\label{one}
                             Y_0 \amalg (\mb{Y}_{\mb{R}_{\oci{-\infty, -1}}}, \mb{X})
                          \end{gather}
                          and
                          \begin{equation}\label{two}
                             \pof[Y_i = y][X_0 = 1] = 2^{\varphi(y)}/\sum_{x\in\{0,1\} }2^{\varphi(x)} 
                          \end{equation}
                          for all $y\in\{0,1\}$.
                          Clearly, $\mb{Y}=\mb{G}$ satisfies these conditions.
        
                          Now we will show that if a process $\mb{Y}$ has properties \eqref{one} and \eqref{two} then $\mb{X}\cdot \mb{Y} \sim \mb{X}\cdot\mb{G}$ under $\P_{X_0 = 1}$. Clearly, it is enough to show that for any local potential $f$, we have
                          \begin{equation}\label{aim intrinsic}
                             \E f(\mb{X}_-\cdot \mb{Y}_-) = \E f(\mb{X}_-\cdot \mb{G}_-),
                          \end{equation}
                          where for any sequence $\mb{y}$, $\mb{y}_-$ stands for $\mb{y}_{\oci{-\infty, 0}}$. To see this take a Gibbs process $\mb{H}$ such that $\mb{H} \amalg(\mb{X}, \mb{Y})$ and $\mb{H} \sim \mb{G}$. Then for any finite non-positive interval of indices $0 \in A \subseteq \{i \in \Z : i \le 0\}$ and a corresponding binary vector $x \in \{0, 1\}^{|A|}$, by~\eqref{wzoreczek},~\eqref{one} and~\eqref{two}, we obtain
                          \begin{equation*}
                          	\E_{X_A = x} f(\mb{X}_- \mb{Y}_-) = \E_{X_A = x} \E_{X_A = x}(f(\mb{X}_{<0} \mb{Y}_{<0}, Y_0)|\mb{X}, \mb{Y}_{R_{<0}}) =\E_{X_A = x} f(\mb{X}_{\le r} \mb{Y}_{\le r}, 0^{r- 1}, H_0))
                          	\end{equation*}
                          where $r$ is the first negative index from $A$ such that $x_r = 1$.  Repeating this argument we obtain
                          \begin{equation}
	                          		\E_{X_A = x} f(\mb{X}_- \mb{Y}_-) = \E f(\mb{x_-}\mb{H_-})
                          \end{equation}
                          (for convenience's sake we extended $x$ to $\mb{x_-}$ putting zeros at the remaining negative coordinates) and hence $\E f(\mb{X}_-\cdot \mb{Y}_-) = \E f(\mb{X_-}\mb{H_-})$, independent of the choice of $\mb{Y}$.                    
        \end{proof}
        
\subsection*{Acknowledgements}
Research of the first author was supported by National Science Centre grant UMO-2019/33/B/ST1/00364 (Poland). Research of the last author was supported by National Science Centre grant 2019/33/B/ST1/00275 (Poland).
       
\bibliographystyle{acm}
\bibliography{basic}

\bigskip
\footnotesize
\noindent
Joanna Ku\l aga-Przymus\\
\textsc{Faculty of Mathematics and Computer Science}\\
\textsc{Nicolaus Copernicus University in Toru\'{n}}\\ 
\textsc{Chopina 12/18, 87-100 Toru\'{n}, Poland}\\
\texttt{joasiak@mat.umk.pl}\\
\\
\noindent
Micha\l{} D. Lema\'{n}czyk\\
\textsc{Faculty of Physics, Astronomy and Informatics}\\
\textsc{Nicolaus Copernicus University in Toru\'{n}}\\
\textsc{Grudzi\c{a}dzka 5, 87-100 Toru\'{n}, Poland}\\
\texttt{mlemanczyk@fizyka.umk.pl}\\
\\
\noindent
Micha\l{} Rams\\
\textsc{Institute of Mathematics}\\
\textsc{Polish Academy of Sciences} \\
\textsc{Śniadeckich 8, 00-656 Warszawa, Poland}\\
\texttt{rams@impan.pl}

\end{document}